\documentclass[10pt]{amsart}

\usepackage{hyperref,amsmath,amsthm,amsfonts,amscd,flafter,epsf,amssymb,ifsym,wasysym,color}
\usepackage{epsfig}
\usepackage{amsfonts}
\usepackage{amssymb}
\usepackage{amsmath}
\usepackage{amsthm}
\usepackage{latexsym}
\usepackage{amscd}
\usepackage{epsf}
\input{diagrams}
\newcommand{\DD}{\mathbb{D}}
\newcommand{\is}{\ov\imath}
\newcommand{\Rinn}{\ov \Ring}
\newcommand{\Rin}{\widetilde{\Ring}}
\newcommand{\lam}{\widetilde{\lambda}}
\newcommand{\Ar}{\mathrm{Area}_\eta}
\newcommand{\F}{\mathbb Z}

\newcommand{\Rr}{\mathfrak{R}}
\newcommand{\la}{\lambda}

\newcommand{\Ht}{\mathrm{H}}
\newcommand{\Hbb}{\mathbb{H}}

\newcommand{\ov}{\widehat}
\newcommand{\ovl}{\overline}
\newcommand{\ti}{\widetilde}
\newcommand{\ba}{\overline}

\newcommand{\Ringg}{\mathbb{B}}

\newcommand{\Hh}{\widehat{H}}

\newcommand{\ta}{\Ring}
\newcommand{\Sbb}{\mathbb{S}}
\newcommand{\tab}{(\Ring,\Hbb)}
\newcommand{\tabb}{(\Ringg,\Hbb)}
\newcommand{\tabl}{(\ovl\Ring,\ovl\Hbb)}

\newcommand{\gamm}{\tau}

\newcommand{\n}{\mathfrak{n}}
\newcommand{\Ncal}{\mathcal{N}}
\newcommand{\f}{\mathfrak{f}}
\newcommand{\g}{\mathfrak{g}}

\newcommand\RelSpinC{\SpinC}

\newcommand{\Tg}{\mathbb{T}_{\gamma}}
\newcommand{\Tai}{{\mathbb{T}}_{\alpha^i}}
\newcommand{\Taj}{{\mathbb{T}}_{\alpha^j}}
\newcommand{\Tbi}{{\mathbb{T}}_{\beta^i}}
\newcommand{\Tbjj}{{\mathbb{T}}_{\beta^j}}
\newcommand{\Ss}{\mathfrak{S}}
\newcommand{\Tt}{\mathfrak{T}}
\newcommand{\Bb}{\mathfrak{B}}
\newcommand{\Pp}{\mathfrak{P}}
\newcommand{\PP}{\mathbb{B}}

\newcommand{\Nod}{\mathcal{N}}
\newcommand{\Group}{\mathbb{G}}
\newcommand{\Or}{\mathfrak{o}}

\newcommand{\vsi}{\varsigma}

\newcommand{\Acal}{\mathcal{A}}
\newcommand{\Qcal}{\mathcal{Q}}
\newcommand{\Dcal}{\mathcal{D}}

\newcommand{\el}{\kappa}
\newcommand{\Pcal}{\mathcal{P}}

\newcommand\Dual{\mathcal D}
\newcommand\Duality\Dual

\newcommand{\relspinc}{{\underline{\spinc}}}

\newcommand\x{\mathbf x}
\newcommand\w{\mathbf w}
\newcommand\z{\mathbf z}
\newcommand\p{\mathbf p}
\newcommand\q{\mathbf q}
\newcommand\y{\mathbf y}

\newcommand\ModSphere{\ModFlow\left({\mathbb S}\longrightarrow
\Sym^{g-1}(\Sigma_{1})\times \Sym^2(\Sigma_{2})\right)}
\newcommand\ModSpheres\ModSphere
\newcommand\CF{CF}

\newcommand\gr{\mathrm{gr}}

\newcommand\UnparModSp{\widehat \ModSp}
\newcommand\UnparModFlow\UnparModSp

\newcommand\PD{\mathrm{PD}}

\newcommand{\spinc}{\mathfrak s}

\newcommand{\spinct}{\mathfrak t}

\newcommand\sM{\mathcal{M}}

\newcommand\ModMaps{\mathcal M}
\newcommand\ModSp\ModMaps

\newcommand\Ta{{\mathbb T}_{\alpha}}
\newcommand\Tb{{\mathbb T}_{\beta}}
\newcommand\Tc{{\mathbb T}_{\gamma}}
\newcommand\Td{{\mathbb T}_{\delta}}

\newcommand\alphas{\mbox{\boldmath$\alpha$}}

\newcommand\betas{\mbox{\boldmath$\beta$}}
\newcommand\gammas{\mbox{\boldmath$\gamma$}}

\newcommand\deltas{\mbox{\boldmath$\delta$}}

\newcommand\Ring{\mathbb A}

\newcommand\spincrel\relspinc
\newcommand\relspinct{\underline{\mathfrak t}}

\newtheorem{thm}{Theorem}[section]
\newtheorem{prop}[thm]{Proposition}
\newtheorem{cor}[thm]{Corollary}

\newtheorem{lem}[thm]{Lemma}

\newtheorem{example}[thm]{Example}
\newtheorem{defn}[thm]{Definition}

\newtheorem{remark}[thm]{Remark}

\def\endproof{\relax\ifmmode\expandafter\endproofmath\else
  \unskip\nobreak\hfil\penalty50\hskip.75em\hbox{}\nobreak\hfil\bull
  {\parfillskip=0pt \finalhyphendemerits=0 \bigbreak}\fi}
\def\endproofmath$${\eqno\bull$$\bigbreak}
\def\bull{\vbox{\hrule\hbox{\vrule\kern3pt\vbox{\kern6pt}\kern3pt\vrule}\hrule}}

\newcommand{\Q}{\mathbb{Q}}
\newcommand{\R}{\mathbb{R}}
\newcommand{\T}{\mathbb{T}}

\newcommand{\C}{\mathbb{C}}

\newcommand{\Z}{\mathbb{Z}}

\newcommand{\ModSWfour}{\mathcal{M}}
\newcommand{\ModFlow}{\ModSWfour}

\newcommand{\SpinC}{{\mathrm{Spin}}^c}

\newcommand\abuts\Rightarrow
\newcommand\Sym{\mathrm{Sym}}

\newcommand{\CFKT}{\text{CFK}}

\newcommand{\CFT}{\mathrm{CF}}
\newcommand{\HFT}{\mathrm{HF}}

\newcommand{\D}{\mathbb{D}}
\newcommand{\E}{\mathbb{E}}

\newcommand{\A}{\mathbb{A}}

\newcommand{\M}{\mathbb{M}}
\newcommand{\m}{\mathfrak{m}}

\newcommand{\Inv}{\mathcal{I}}
\newcommand{\lra}{\longrightarrow}
\newcommand{\ra}{\rightarrow}
\newcommand{\Sig}{\Sigma}

\newcommand{\Mod}{\mathcal{M}}

\newcommand{\dbar}{\overline{\partial}}

\newcommand{\Ker}{\text{Ker}}

\begin{document}

\title{A refinement of sutured Floer homology}%
\author{Akram S. Alishahi}
\address{Department of Mathematics, Sharif University of Technology, P. O. Box 11155-9415,
Tehran, Iran}
\email{akram.alishahi@gmail.com}
\author{Eaman Eftekhary}%
\address{School of Mathematics, Institute for Research in Fundamental Science (IPM),
P. O. Box 19395-5746, Tehran, Iran}%
\email{eaman@ipm.ir}

\begin{abstract}
We introduce a refinement of the Ozsv\'ath-Szab\'o complex
associated to a balanced sutured manifold $(X,\tau)$ by Juh\'asz \cite{Juh}.
An algebra $\Ring_\tau$ is associated to the boundary of a sutured manifold and
a filtration of its generators by $\Ht^2(X,\partial X;\Z)$ is defined.
For a fixed class $\spinc$ of a $\SpinC$ structure over the manifold
$\ovl X$, which is obtained from $X$ by filling out the sutures,
the Ozsv\'ath-Szab\'o chain complex $\CFT(X,\tau,\spinc)$
is then defined as a chain complex with coefficients in $\Ring_\tau$ and
filtered by $\SpinC(X,\tau)$. The filtered chain homotopy type
of this chain complex is  an invariant of $(X,\tau)$ and the
$\SpinC$ class $\spinc\in\SpinC(\ovl X)$. The construction
generalizes the construction of Juh\'asz. It plays the role of
$\CFT^{-}(X,\spinc)$ when $X$ is a closed three-manifold, and
the role of
$$\CFKT^-(Y,K;\spinc)=\bigoplus_{\relspinc\in\spinc}\CFKT^-(Y,K,\relspinc),$$
when the sutured manifold is obtained from a knot $K$ inside a three-manifold $Y$.
Our invariants generalize both the knot invariants of
Ozsv\'ath-Szab\'o and Rasmussen and the link invariants of Ozsv\'ath and Szab\'o.
We study some of the basic properties of the corresponding Ozsv\'ath-Szab\'o
complex, including the exact triangles, and some form of stabilization.
\end{abstract}
\maketitle
\tableofcontents
\newpage
\section{Introduction}\label{sec:intro}
\subsection{Introduction and the main results}
The introduction of Heegaard Floer homology by Ozsv\'ath and Szab\'o (\cite{OS-3m1},\cite{OS-3m2})
for closed three dimensional manifolds around the beginning of the millennium resulted in
very powerful tools for the study of various structures in low dimensional topology. In
particular, invariants for knots (c.f. \cite{OS-knot}, \cite{Ras} and \cite{Ef-LFH}), for links
\cite{OS-linkinvariant}, and for contact structures \cite{OS-contact} were constructed
using the fundamental idea of associating a chain complex to a pointed Heegaard diagram.
Moreover, four manifold invariants were constructed as some TQFT type homomorphisms
between the homology groups of the chain complexes associated
to the positive and negative boundary components \cite{OS-4m}.
The Ozsv\'ath-Szab\'o complexes associated with a closed three-manifold come in different flavors.
These are typically called \emph{hat, minus, plus }and \emph{infinity} modules. The other versions
may be re-constructed from the \emph{minus theory} if one also keeps track of the so called $U$-action.
Juh\'asz extended the hat version of Ozsv\'ath-Szab\'o complex to the context of
balanced sutured manifolds \cite{Juh}.
The sutured Floer homology of Juh\'asz detects taut sutured manifolds \cite{Juh-surface}, and may be used to define
a polytope associated with a sutured manifold which behaves well under taut surface
decompositions \cite{Juh-taut}.\\

In this paper, we extend the construction of Juh\'asz and construct a \emph{minus theory}
associated with a balanced sutured manifold. More precisely, let $(X,\tau)$ be a balanced
sutured manifold and let $\tau=\{\gamma_1,...,\gamma_\el\}$ be the set of sutures.
We will denote $\partial X-\tau$ by $\Rr(\tau)=\Rr^+(\tau)\cup\Rr^-(\tau)$, where
$\Rr^+(\tau)$ and $\Rr^-(\tau)$ are the positive and the negative part of the boundary, respectively.
We first associate an algebra $\Ring=\Ring_\tau$ to the boundary of $X$ as follows.
Let us assume that
$$\Rr^-(\tau)=\bigcup_{i=1}^k R_i^-,\ \ \&\ \
\Rr^+(\tau)=\bigcup_{j=1}^l R_j^+,$$
where $R_i^-$ and $R_j^+$ are the
connected components of $\Rr^-(\tau)$ and
$\Rr^+(\tau)$ respectively. Let $g_i^-$ denote the genus of $R_i^-$ and $g_j^+$ denote the genus
of $R_j^+$.
Consider the elements
\begin{displaymath}
\begin{split}
\la_i^-:=\prod_{\gamma_j\subset \partial R_i^-}\la_j,\ \ i=1,...,k,\ \ \ \ \&\ \ \ \
\la_i^+:=\prod_{\gamma_j\subset \partial R_i^+}\la_j,\ \ i=1,...,l,
\end{split}
\end{displaymath}
in the free $\F$-algebra $\F[\el]:=\langle \la_1,...,\la_\el\rangle$ generated by $\la_1,...,\la_\el$.
Let
\begin{displaymath}
\begin{split}
&\Ring_\tau:=\frac{\Big\langle \la_1,...,\la_\el\Big\rangle_\F}
{\Big\langle \la^+(\tau)-\la^-(\tau)\Big\rangle+
\Big\langle \la_i^+\ |\ g_i^+>0\Big\rangle+
\Big\langle \la_j^-\ |\ g_j^->0\Big\rangle},\\
&\text{where }\  \la^-(\tau)=\sum_{i=1}^k \la_i^-,\ \ \ \ \&\ \ \ \
\la^+(\tau)=\sum_{i=1}^l \la_i^+.
\end{split}
\end{displaymath}
 We will denote the set of monomials  $\prod_{i=1}^\el \la_i^{a_i}$
 by $G(\Ring)$, which forms a set of generators for $\Ring$.
One may define a natural morphism from $G(\Ring)$ to the $\Z$-module
$\Hbb=\Hbb_\tau:=\Ht^2(X,\partial X,\Z)$ by
\begin{displaymath}
\begin{split}
&\chi:G(\Ring)\lra \Hbb=\Ht^2(X,\partial X;\Z),\\
&\chi\Big(\prod_{i=0}^\el \la_i^{a_i}\Big):=
a_1\mathrm{PD}[\gamma_1]+...+a_\el\mathrm{PD}
[\gamma_\el],\ \ \forall \ a_1,..,a_\el\in\Z^{\geq 0}.
\end{split}
\end{displaymath}

Let $\ovl X=\ovl X^\tau$ be the three-manifold obtained by filling the sutures of
$(X,\tau)$ by attaching  $2$-handles to the sutures in $\tau$.
Fix a $\SpinC$ class $\spinc\in\SpinC(\ovl X)$.
Suppose that $(\Sig,\alphas,\betas,\z)$ is a Heegaard diagram for the sutured manifold
$(X,\tau)$, which is admissible in an appropriate sense.
Thus  $\Sig$ is a closed Riemann surface, $\alphas$ and $\betas$ are $\ell$-tuples of
disjoint simple closed curves, and $\z$ is a set of $\el$ marked points on $\Sig$.
If $\Sig^\circ=\Sig-\mathrm{nd}(\z)$ is the complement of a neighborhood of $\z$,
$X$ is obtained from $\Sig^\circ\times [-1,1]$ by attaching $2$-handles to
$\alphas\times\{-1\}$ and $\betas\times\{1\}$.
The Ozsv\'ath-Szab\'o chain complex $\CFT(X,\tau,\spinc)$ is then generated as a free
$\Ring_\tau$-module by those intersection points of the tori $\Ta,\Tb\subset \Sym^{\ell}(\Sig)$
associated with $\alphas$ and $\betas$ which correspond to the $\SpinC$ class $\spinc\in\SpinC(\ovl X)$.
The set $\pi_2^+(\x,\y)$ of positive Whitney disks  for generators
$\x,\y\in\Ta\cap\Tb$ is defined as usual, and we will have a map
\begin{displaymath}
\begin{split}
&\la_\z:\coprod_{\x,\y\in\Ta\cap\Tb}\pi_2^+(\x,\y)\lra G(\Ring)\\
&\la_\z(\phi):=\prod_{i=1}^\el \la_i^{n_{z_i}(\phi)},\ \ \
\forall\ \x,\y\in\Ta\cap\Tb,\ \&\  \forall\ \phi\in\pi_2^+(\x,\y).
\end{split}
\end{displaymath}
Here $n_{z_i}(\phi)$ denotes the coefficient of $z_i$ in the domain $\Dcal(\phi)$
associated with the Whitney disk $\phi$.
The differential $\partial$ of the complex $\CFT(X,\tau,\spinc)$ is defined by counting holomorphic disks
$\phi$ of Maslov index $1$ connecting the generators $\x$ and $\y$
of the complex, with an appropriate sign and weight $\la_\z(\phi)\in \Ring_\tau$.
The assignment of relative $\SpinC$ structures to the intersection points $\x\in\Ta\cap\Tb$
using $\z$ gives $\CFT(X,\tau,\spinc)$ the structure of a filtered $\tab$ chain complex
(see section~\ref{sec:algebra}
for a precise definition). The following is the main result of this paper.
\begin{thm}
The filtered $\tab$ chain homotopy type of the filtered
$\tab$ chain complex $\CFT(X,\tau,\spinc)$
is an invariant of the balanced sutured manifold $(X,\tau)$ and the $\SpinC$ class $\spinc\in\SpinC(\overline{X})$.
In particular, for any $\relspinc\in\spinc\subset\RelSpinC(X,\tau)$ the chain homotopy type of
the summand
$$\CFT(X,\tau,\relspinc)\subset\CFT(X,\tau,\spinc)=\bigoplus_{\relspinc\in\spinc}
\CFT(X,\tau,\relspinc)$$
is also an invariant of $(X,\tau,\relspinc)$.
\end{thm}

The above theorem implies that whenever we have a homomorphism $\rho:\Ring\ra \Ringg$ for a
ring $\Ringg$, the chain homotopy type of the complex
$$\CFT(X,\tau,\spinc;\Ringg)=\CFT(X,\tau,\spinc)\otimes_\Ring\Ringg$$
is also an invariant of the sutured manifold $(X,\tau)$. This complex is equipped
with filtration by $\SpinC(X,\tau)$ if the homomorphism $\rho$ respects
the filtration of the monomials of $\Ring$ by the elements of $\Hbb$.
In this case, it makes sense to talk about the following decomposition of $\CFT(X,\tau,\spinc;\Ringg)$:
$$\CFT(X,\tau,\spinc;\Ringg)=\bigoplus_{\relspinc\in\spinc\subset \SpinC(X,\tau)}
\CFT(X,\tau,\relspinc;\Ringg).$$
In particular, the homology groups
$$\mathrm{HF}(X,\tau,\relspinc;\Ringg)=H_*\left(\CFT(X,\tau,\relspinc;\Ringg),\partial\right),\ \
\forall\ \relspinc\in\SpinC(X,\tau)$$
may be defined, and are invariants of the sutured manifold and the relative $\SpinC$
class $\relspinc\in\SpinC(X,\tau)$.
As a special case, we may take $\Ringg=\F$ and let $\rho$  be the map sending all
the non-trivial monomials to zero. We will then recover the sutured Floer homology of Juh\'asz:
$$\mathrm{SFH}(X,\tau,\relspinc)=\mathrm{HF}(X,\tau,\relspinc;\Z),\ \
\forall\ \relspinc\in\SpinC(X,\tau).$$
Define a particular test ring $\Ringg_\tau$ for $\Ring_\tau$ by setting
\begin{displaymath}
\Ringg_\tau=\frac{\Big\langle \la_1,...,\la_\el\Big\rangle_\Z}
{\Big\langle\prod_{i=1}^\el \la_i^{n_i}\neq 1\ |\ n_i\in\Z^{\geq 0}\ \&\
\sum_{i=1}^\el n_i[\gamma_i]=0\ \text{in } \Ht_1(X;\Z)/\mathrm{Tors}\Big\rangle}.
\end{displaymath}
Clearly, there is a quotient map $\rho_\tau:\Ring_\tau\ra\Ringg_\tau$.
The following is a refinement of Juh\'asz' theorem 1.4 from \cite{Juh-surface}.
\begin{prop}
An irreducible balanced sutured manifold $(X,\tau)$ is taut if and only if the filtered
$(\Ringg_\tau,\Hbb_\tau)$ chain homotopy type of the complex
$$\CFT(X,\tau;\Ringg_\tau)=\bigoplus_{\spinc\in\SpinC(\ovl X)}\CFT(X,\tau;\spinc;\Ringg_\tau)$$
is non-trivial.
\end{prop}

For a knot $K$ inside a closed three-manifold $Y$, the boundary of the corresponding sutured
manifold $(X=Y-\mathrm{nd}(K),\tau)$ consists of a torus and $\tau$ consists of a pair of parallel
sutures on this torus. Thus, with the above notation,
$$k=l=1,\ \  g_1^+=g_1^-=0,\ \  \&\ \
 \la_1^+=\la_1^-=\la_1\la_2.$$
Thus, the algebra $\Ring$ is equal to
$\F[\la_1,\la_2]=\langle \la_1,\la_2\rangle_\F$, which gives
the $\Z\oplus\Z$ filtration associated with the knot $K$ inside the three-manifold $Y$.\\

The surgery exact triangle for the Ozsv\'ath-Szab\'o complexes associated with closed
three-manifolds may be extended to our setup. Namely, let $(X,\tau)$ be a sutured manifold
and $\gamma_1,\gamma_2\in\tau$ be two parallel sutures with opposite orientation which
form the common boundary of a cylindrical component $R_1^+\subset\Rr^+(\tau)$ and
a genus zero component $R_1^-\subset\Rr^-(\tau)$. Consider a simple closed curve
$$\la\subset  R_1^+\cup  R_1^-\cup\gamma_1\cup\gamma_2$$
which cuts $\gamma_1$ and $\gamma_2$ in a single transverse
point and remains disjoint from the rest of the sutures. Replacing $\gamma_1$ and $\gamma_2$
with two parallel copies of $\la$ (with opposite orientation) results in a new
sutured manifold $(X,\tau_\la)$.
Let $\la(n)$ be  (the homotopy class of) the simple closed curve obtained
from $\la$ by  twisting it $n$ times along $\gamma_1$ (or equivalently, $-n$ times along $\gamma_2$).
Correspondingly, we obtain the sutured manifold $(X,\tau_{\la(n)})$. When the choice of $\la$ is fixed,
we sometimes write $(X,\tau_n)$ for $(X,\tau_{\la(n)})$.
\\

The algebra associated with
all the sutured manifolds $(X,\tau_n)$ is the same. Let us denote this algebra by $\Ring$,
and assume that $\zeta_1,...,\zeta_\el$ are the generators of $\Ring$ which correspond to
the sutures. Furthermore, let $\zeta_1$ and $\zeta_2$ correspond to $\gamma_1$ and
$\gamma_2$ respectively. Note that in the relations ideal $I_\tau$ in
$\F[\el]=\langle \zeta_1,...,\zeta_\el\rangle_\F$ (which defines $\Ring$ as $\F[\el]/I_\tau$)
the generators either use $\zeta_1\zeta_2$, or they use none of $\zeta_1$ and $\zeta_2$. We may thus introduce
a new algebra $\Ringg$ as a quotient of  $\langle \la_0,\la_1,...,\la_\el\rangle_\F$
by an ideal $J_\tau$. The generators of $J_\tau$ are constructed from the generators of
$I_\tau$ by replacing $\zeta_j$ with $\la_j$ for $j=3,...,\el$ and replacing
$\zeta_1\zeta_2$ with $\la_0\la_1\la_2$. For $i=0,1,2$ we obtain embeddings  $\imath^i$
of $\Ring$ in $\Ringg$:
\begin{displaymath}
\imath^i:\Ring\ra \Ringg,\ \ \
\imath^i(\zeta_j)
=\begin{cases}
\la_i\ \ \ &\text{if }j=1\\
\frac{\la_0\la_1\la_2}{\la_i}\ \ \ &\text{if }j=2\\
\la_j\ \ \ &\text{if }3\leq j\leq \el\\
\end{cases}.
\end{displaymath}
We write $\Ring_i$ in order to refer to $\Ring$ as the sub-ring $\imath^i(\Ring)\subset \Ringg$.\\

 To keep the exposition simpler, we only consider the surgery triangle associated with the
 sutured manifolds $(X,\tau)$, $(X,\tau_0)$ and $(X,\tau_1)$. Let us denote by
 $\chi_j\in\Ht^2(X,\partial X;\Z)$ the Poincar\'e dual of the suture $\gamma_j$, for
 $j=3,...,\el$. Furthermore, let $\chi_0,\chi_1$ and $\chi_2$ denote the Poincar\'e duals
 of $\gamma_1$, $\la(0)$ and $-\la(1)$, respectively. Note that $\chi_0+\chi_1+\chi_2=0$ in
 $\Hbb=\Ht^2(X,\partial X;\Z)$. Define the filtration map by
 \begin{displaymath}
 \begin{split}
 &\chi:G(\Ringg)\lra \Ht^2(X,\partial X;\Z) \\
 &\chi\left(\prod_{j=0}^\el \la_i^{a_i}\right):=\sum_{j=0}^\el a_i \chi_i.
 \end{split}
\end{displaymath}
Associated with any $\SpinC$ class $\spinc\in\SpinC(\ovl X)$
let $\E_i(\spinc;\Ring_i)$ be the complex $\CFT(X,\tau,\spinc;\Ring_0)$, $\CFT(X,\tau_0,\spinc;\Ring_1)$, or
$\CFT(X,\tau_1,\spinc;\Ring_2)$ depending on whether $i=0,1$ or $2$.
Let $\E_i(\spinc;\Ringg)=\E_i(\spinc;\Ring_i)\otimes_{\Ring_i}\Ringg$.
\begin{thm}
With the above notation fixed, we have a triangle
\begin{displaymath}
\begin{diagram}
\E_0(\spinc;\Ringg)&&\rTo{\f^{\spinc}_2}&&
\E_1(\spinc;\Ringg)\\
&\luTo{\f_1^{\spinc}}&&\ldTo{\f_0^{\spinc}}&\\
&&\E_{2}(\spinc;\Ringg)&&
\end{diagram}
\end{displaymath}
of filtered $(\Ringg,\Hbb)$ chain maps such that $\f_1^\spinc\circ \f^\spinc_0$, $\f_2^\spinc\circ \f^\spinc_1$,
and $\f_0^\spinc\circ \f^\spinc_2$ are null homotopic. Moreover, $\E_i(\spinc;\Ringg)$ is filtered
$(\Ringg,\Hbb)$ chain homotopic to the mapping cone of $\f^\spinc_i$. In particular,
if there is a homomorphism $\rho_R:\Ringg\ra R$ to a ring $R$, taking the tensor product of the above
triangle with $R$ and computing the homology groups we obtain a long exact sequence in homology:
\begin{displaymath}
\begin{diagram}
\dots&\rTo{f^\spinc_1}& \mathrm{HF}(X,\tau,\spinc;R)&\rTo{\f^\spinc_2}&
\mathrm{HF}(X,\tau_0,\spinc;R)&\rTo{\f^\spinc_0}&
\mathrm{HF}(X,\tau_1,\spinc;R)&\rTo{\f^\spinc_1}&
\dots.
\end{diagram}
 \end{displaymath}
\end{thm}
If the homomorphism $\rho_R$ also respects the filtration by $\Hbb$, the above
exact sequence refines to an exact sequence corresponding to any of the relative
$\SpinC$ structures $\relspinc\in\SpinC(X,\tau)$.\\

\subsection{Previous results and the history}
Attempts on extending the Ozsv\'ath-Szab\'o invariants
to three manifolds with boundary, at least when the boundary
is equipped with some extra structure have been made through two different approaches. If a
parametrization  of the boundary surface is  fixed, the three-manifold is called
a bordered three manifold. Lipshitz, Ozsv\'ath and Thurston generalize
the hat version of the Ozsv\'ath-Szab\'o complex for bordered three-manifold
by first constructing a graded differential algebra corresponding to
the parameterized boundary, and then associating the Bordered Floer modules of type $A$ and $D$
to the bordered manifold, which are respectively an $\Acal_\infty$ module and a module over
the differential graded algebra (see \cite{LOT-BFH1},\cite{LOT-BFH2}).
Gluing of bordered three-manifolds for constructing closed three-manifolds is translated
to an appropriate tensor product construction on the corresponding Bordered Floer modules.
\\

In a different direction, if the boundary of a three-manifold $X$ is decorated with a set $\tau$
of suture, Juh\'asz associates a complex, the so called \emph{sutured Floer complex} to the sutured
manifold $(X,\tau)$ \cite{Juh}, provided that $(X,\tau)$ is balanced.
The complex generalizes the hat versions of the Ozsv\'ath-Szab\'o
complexes associated with closed three-manifolds and links inside three-manifolds.
The theory of  sutured manifolds was introduced in \cite{Gabai-foliations1} and developed
in \cite{Gabai-foliations2} and \cite{Gabai-foliations3} by D. Gabai in order to study the existence
of taut foliations on three-manifolds. Sutured manifolds are  oriented three-manifolds with boundary,
together with a set of oriented simple closed curves (the sutures) that divide the boundary into
positive and negative parts. Gabai defines the so called \emph{sutured manifold decomposition}
which consists of cutting the manifold along a properly embedded oriented surface $R$ and adding one
side of $R$ to the plus boundary and the other side to the minus boundary.
He shows that a sutured manifold carries a taut foliation if and only if there is a sequence of
decompositions that result in a product sutured manifold. Honda, Kazez, and Mati\'c
  generalized the theory of sutured manifold decomposition  for the study of tight
  contact structures on three-manifolds, and developed the convex decomposition theory \cite{HKM2}.
In addition to the introduction of  sutured Floer complex, Juh\'asz described how sutured
Floer complex changes through sutured manifold decomposition \cite{Juh-surface}.
As a consequence, he shows that a sutured manifold $(X,\tau)$ is taut if and only if the sutured Floer
homology group $\mathrm{SFH}(X,\tau)$ is non-trivial.\\

These results suggested a deep connection between sutured Floer theory of Juh\'asz and the
sutured manifold decomposition theory of Gabai, as well as the contact geometry of three-manifolds.
Subsequent developments included the study of sutured Floer polytope by Juh\'asz \cite{Juh-taut}
and introduction of contact invariants for contact three-manifolds with convex boundary
by Honda, Kazez and Mati\'c \cite{HKM}. This last invariant generalizes the contact invariant
of Ozsv\'ath and Szab\'o for  a closed contact three-manifold
defined in \cite{OS-contact}.

\subsection{Outline of the paper}
The paper is organized as follows. In section~\ref{sec:surgery} we review some of the basic notions, including
the sutured manifolds, the corresponding Heegaard diagrams, and the $\SpinC$ structures on sutured manifolds.
We will also review some of the main constructions studied in this paper, including surgery
and filling the sutures.\\

In section~\ref{sec:admissibility} we investigate a notion of admissibility for Heegaard diagrams, which makes
it possible to construct an Ozsv\'ath and Szab\'o complex using Heegaard Floer theory.
The admissibility condition
is slightly weaker, in a sense, than the strong admissibility of Ozsv\'ath and Szab\'o in the context of
closed three-manifolds. However, it is strong enough for the construction of Ozsv\'ath-Szab\'o complex to work.
We show that all balanced sutured manifolds admit admissible Heegaard diagrams corresponding
to any $\SpinC$ class.\\

In section~\ref{sec:algebra} we develop the language of chain complexes filtered by a module, and
make some simple algebraic observations. Moreover, we construct an algebra associated with the
boundary of a balanced sutured manifold, as well as a filtration of its generators by classes in
$\Ht^2(X,\partial X;\Z)$. The algebra plays the role of the coefficient ring for the Ozsv\'ath-Szab\'o
chain complex associated with the balanced sutured manifold.\\

In section~\ref{sec:analytic-aspects} we study the orientability issues for the corresponding
moduli spaces. In particular, an appropriate orientation for the moduli spaces of boundary degenerations
is required so that the differential $\partial$ of the associated Ozsv\'ath-Szab\'o chain complex
satisfies $\partial^2=0$. Analyzing the analytic aspects of the theory thus requires  some new
techniques which are developed in section~\ref{sec:analytic-aspects}.\\

In section~\ref{sec:chain-complex} we construct the chain complex associated with
an admissible  Heegaard diagram for the balanced sutured manifold $(X,\tau)$.
We show that the filtered chain homotopy type of this complex is invariant under Heegaard
moves, and is independent of the choice of the path of almost structure on the symmetric product
of the Heegaard surface. The choice of the algebra associated with the boundary plays a very
crucial role both in defining the chain complex and proving the invariance of the
filtered chain homotopy type.\\

In section~\ref{sec:stabilization} we study how the filtered chain homotopy type of the
Ozsv\'ath-Szab\'o complex associated with a balanced sutured manifold $(X,\tau)$ changes
when we add two parallel copies of an existing suture to the boundary
with appropriate orientation. The operation is
called the {\emph{stabilization}} of the sutured manifold $(X,\tau)$.
When $(X,\tau)$ corresponds to a knot $K$ inside a closed three-manifold $Y$,
the stabilization corresponds to considering multi-pointed Heegaard diagrams for
defining the knot Floer complex, and the stabilization formula generalizes
the relation between usual Ozsv\'ath-Szab\'o complexes and the multi-pointed ones.\\

Finally, in section~\ref{sec:exact-triangle} we introduce a generalization of the
surgery triangle for balanced sutured manifolds. The freedom to choose many marked
points on the Heegaard diagram allows us to understand the chain maps in a better way, and
refine the existing triangles, and long exact sequences.

\newpage
\section{Background on sutured manifolds}\label{sec:surgery}
\subsection{Sutured manifolds and relative $\SpinC$ structures}
In this paper, we deal only with balanced sutured manifold, so we will modify the standard definition of sutured manifolds,
by throwing away the possibility of having a torus component in the suture.
\begin{defn}
A {\emph{sutured manifold}}  $(X,\gamm)$ is a compact oriented three-manifold $X$ with boundary $\partial X$, together
with a set of disjoint oriented simple closed curves $\gamm=\{\gamma_1,...,\gamma_\el\}$ on $\partial X$. We will denote by
$A(\gamma_i)$ a tubular neighborhood of $\gamma_i$ in $\partial X$, which will be an annulus.
We let $A(\gamm)=A(\gamma_1)\cup...\cup A(\gamma_\el)$.
Every component of $\Rr(\gamm)=\partial X- A(\gamm)^{\circ}$ is oriented (where $A(\gamm)^\circ$ denotes the interior
of $A(\gamm)$. Let  $\Rr(\tau)=\Rr^+(\tau)\cup\Rr^-(\tau)$ where
$\Rr^+(\gamm)$ denotes the union of components of $\Rr(\gamm)$
with the property that the orientation induced on $\tau$ as the boundary of $\Rr^+(\tau)$ agrees
with the orientation of $\tau$, while
$\Rr^-(\gamm)$ denotes the union of components of $\Rr(\gamm)$
with the property that the orientation induced on $\tau$ as the boundary of $\Rr^-(\tau)$ is the opposite of
the orientation of $\tau$. We assume that the orientation on the components of
$\Rr(\gamm)$ is compatible with the orientation of the boundary $\partial \Rr(\gamm)$ induced by the
sutures $\gamma_1,...,\gamma_\el$. A sutured manifold $(X,\gamm)$ is called {\emph{balanced}} if $X$ has no
closed components, $\chi(\Rr^+(\gamm))= \chi(\Rr^-(\gamm))$ and the induced map $\pi_0(A(\gamm))\ra \pi_0(\partial X)$
is surjective.
\end{defn}

\begin{defn}
A \emph{Heegaard diagram} is a tuple
$(\Sigma,\mbox{\boldmath${\alpha}$},\mbox{\boldmath${\beta}$},\bf{z})$
such that $(\Sigma,\mbox{\boldmath${\alpha}$},\mbox{\boldmath${\beta}$})$
is a balanced Heegaard diagram i.e. $\Sigma$ is a compact oriented
surface and $\mbox{\boldmath${\alpha}$}$ and $\mbox{\boldmath${\beta}$}$
are sets of disjoint oriented simple closed curves on $\Sigma$ where $|\mbox{\boldmath${\alpha}$}|=|\mbox{\boldmath$\beta$}|=\ell$, and
$${\bf z}=\Big\{z_1,...,z_\el\Big\}\subset \mathrm{int}\left(\Sigma-\bigcup\mbox{\boldmath
${\alpha}$}-\bigcup\mbox{\boldmath ${\beta}$}\right)$$
is a set of marked points such that each connected component of
$\Sigma-\mbox{\boldmath${\alpha}$}$ and $\Sigma-\mbox{\boldmath${\beta}$}$
contains at least one marked point.
\end{defn}

Every  Heegaard diagram $(\Sigma,\mbox{\boldmath${\alpha}$},
\mbox{\boldmath${\beta}$},\bf{z})$ uniquely defines a balanced
sutured manifold manifold as follows. Let
$\Sig^\circ=\Sig-D_1-...-D_\el$ denote the complement
of small disks $D_1,...,D_\el$ around $z_1,...,z_\el$,
where $\z=\{z_1,...,z_\el\}$. The three-manifold $X$ is obtained
from $\Sigma^\circ \times [-1,1]$ by attaching 3-dimensional 2-handles
along the curves $\alpha_i\times\{-1\}$ and $\beta_j\times\{1\}$ for
$i,j=1,...,\ell$). We may  define the set of sutures on the boundary of $X$ by
$$\tau=\Big\{\gamma_1,...,\gamma_\el\Big\},\ \
\gamma_i=\partial D_i\times \{0\}.$$
In this situation,
we  say that $(\Sigma, \mbox{\boldmath$\alpha$},\mbox{\boldmath$\beta$}
,{\bf z})$ is associated with the  the sutured three-manifold $(X,\tau)$.

\begin{prop}
For every balanced sutured manifold $(X,\tau)$,
there exists a Heegaard  diagram associated with it in the above sense.
\end{prop}

\begin{proof}
Let $(\Sigma_\tau,\mbox{\boldmath${\alpha}$},
\mbox{\boldmath${\beta}$})$ be a sutured Heegaard diagram for the
balanced sutured manifold
$(X,\tau)$ in the sense of \cite{Juh}. If
$\tau=\{\gamma_1,...,\gamma_\el\}$ consists of $\el$ sutures,
take $\Sigma$ to be the surface obtained from $\Sig_\tau$ by gluing $\el$ disks
$D_1,D_2,...,D_\el$ to it
along the boundary components corresponding to
$\gamma_1,...,\gamma_\el$. Let $z_i$ be the center of $D_i$, $i=1,...,\el$.
 Then $(\Sigma,\mbox{\boldmath${\alpha}$},\mbox{\boldmath$
 {\beta}$},\z=\{z_1,...,z_\el\})$ is a Heegaard diagram
 for $(X,\tau)$.
\end{proof}

\begin{prop}
If $(\Sigma_1,\mbox{\boldmath$\alpha$}_1,\mbox{\boldmath$\beta$}_1,{\bf z})$ and $(\Sigma_2,\mbox{\boldmath$\alpha$}_2,\mbox{\boldmath$\beta$}_2,{\bf w})$
are two  Heegaard diagrams for a balanced sutured manifold $(X,\tau)$, then they are
diffeomorphic after a finite set of Heegaard moves, which are
supported  away from the marked points.
\end{prop}
\begin{proof}
This is proposition~2.15 from \cite{Juh}.
\end{proof}
For the most part of this paper, we will identify $\Rr(\tau)=\Rr^+(\tau)\cup\Rr^-(\tau)$
as the connected components of $\partial X-\tau$. Thus the boundary of each connected component
$R\subset\Rr(\tau)$ may be identified as a union of curves in $\tau$.
In the few situations where the annuli $A(\gamma_i)$ are relevant, we will
emphasize them in the notation.\\

Suppose that $(X,\tau)$ is a balanced sutured manifold.
 One may define a nowhere vanishing vector field on
 $\partial X$ as follows. Let $v_\tau$ be a vector
 field (with values in $TX|_{\partial X}$)
 which points outward on $\Rr^+(\tau)\subset \partial
 X-A(\tau)=\Rr(\tau)$, and points inward on $\Rr^-(\tau)
 \subset \Rr(\tau)$. Furthermore,
 under the identification $A(\gamma_i)=\gamma_i\times [-1,1]$,
 let $v_\tau|_{A(\gamma_i)}$ be the vector field
 $\frac{\partial}{\partial t}$
 determining the unit tangent vector of the second factor,
 i.e. the interval $[-1,1]$. In fact, we have to perturb
 $v_\tau$ on a small neighborhood $\partial A(\tau)$ to
 make it continuous, but we typically drop this
 perturbation from our notation.
 \begin{defn}
 Suppose that the non-vanishing vector fields $v$ and $w$ on $X$ agree
 with $v_\tau$ on $\partial X$. We say that
 $v$ and $w$ are \emph{homologous} if there is a ball
 $B\subset X^\circ$ such that the restrictions of
 $v$ and $w$ to $X-B$ are homotopic relative the boundary
 of $X$. We define the space $\RelSpinC(X,\tau)$ of
 \emph{relative $\SpinC$ structures}
  on the sutured manifold $(X,\tau)$
 to be the space of homology classes of such nowhere
 vanishing vector fields on $X$ which agree with $v_\tau$
 on $\partial X$.
 \end{defn}
Note that $\RelSpinC(X,\tau)$ is an affine space
over $\Ht^2(X,\partial X,\Z)$.
Let us assume that the $\SpinC$ structure
$\relspinc\in\RelSpinC(X,\tau)$ is represented by
a nowhere vanishing vector field $v$, so
that $v|_{\partial X}=v_\tau$. Let us define
the first Chern class of $\relspinc$ to be the
first Chern class of the
oriented $2$-plane field $v^{\perp}$ over $X$,
which lives in $\Ht^2(X,\Z)$.
Let us denote the inclusion of $\partial X$ in
$X$ by $i:\partial X\ra X$. We thus get a
map $$i^*:\Ht^2(X,\Z)\ra \Ht^2(\partial X,\Z).$$
The first Chern class of the $2$-plane field
$v_\tau^{\perp}$ lives in $\Ht^2(\partial X,\Z)$
and $c_1(\relspinc)$ is thus included in
$$(i^*)^{-1}\Big(c_1(v_{\tau}^{\perp})\Big)\subset \Ht^2(X,\Z).$$

We may glue a solid cylinder $D^2\times [-1,1]$ to
each component $A(\gamma_i)$ of $A(\tau)$ along
$S^1\times [-1,1]$. This way, we obtain a three-manifold
with boundary, which will be denoted
by $\ovl{X}=\ovl{X}^{\tau}$. The set of homology classes
of no-where vanishing vector fields
on $\ovl{X}$ which point outward on the positive boundary
components of $\ovl{X}$ and
point inward on the negative boundary components of
$\ovl{X}$ will be denoted by $\SpinC(\ovl{X})$.
Note that $\SpinC(\ovl{X})$ is an affine space over
$\Ht^2(\ovl{X},\partial \ovl{X})$. Again,
if $\spinc\in\SpinC(\ovl{X})$ is a given $\SpinC$
structure represented by a nowhere
vanishing vector field $w$ as above, we have the
notion of the first Chern class associated with $\spinc$,
which is defined to be $c_1(w^{\perp})$.
Clearly, $c_1(\spinc)$, as defined, is an element of
$\Ht^2(\ovl{X},\Z)$.\\

There is a notion of restricting $\SpinC$ structures
from $(X,\tau)$ to $\ovl{X}$ as follows. If
$\relspinc\in\RelSpinC(X,\tau)$ is represented by the
no-where vanishing vector field $v$
(so that $v|_{\partial X}=v_\tau$), we may extend $v$
over each one of the glued cylinders
$D^2\times [-1,1]$. In fact, $v$ may be extended over
 $D^2\times [-1,1]$ by setting it equal to
$\frac{\partial}{\partial t}$, where $t$ denotes the
variable associated with the interval $[-1,1]$.
The new vector field $\ovl{v}$ determines a $\SpinC$
structure on $\ovl{X}$, which will be denoted
by $[\relspinc]$. This gives a well-defined map
$$[.]:\RelSpinC(X,\tau)\lra \SpinC(\ovl{X}).$$

Let us denote  the inclusion of $X$ in $\ovl{X}$ by
$\imath:X\ra \ovl{X}$. This inclusion gives a
map $$\imath^*:\Ht^2(\ovl{X},\Z)\lra \Ht^2(X,\Z).$$
From the definition of the first Chern class, we know
that if $[\relspinc]$ is represented by
$\ovl{v}$, $\relspinc$ is represented by
$v=\imath^{*}(\ovl{v})$, and thus
$$\imath^{*}\Big(c_1([\relspinc])\Big)=c_1(\relspinc).$$

\subsection{Associated sutured manifolds; surgery and filling the sutures}
Let $(X,\tau)$ be a balanced sutured manifold, with $\tau$
the set of sutures on the boundary of $X$, and
with $\Rr^+(\tau)$ and $\Rr^-(\tau)$ the union of positive,
respectively negative, components in
$\partial X-\tau$ as above. Suppose that $R^+$ is
a component of $\Rr^+(\tau)$, $R^-$ is a component of $\Rr^-(\tau)$,
and $\gamma_1$ and $\gamma_2$ are sutures in $\tau$ such that
$$\partial R^+ \cap \partial R^-=\big\{\gamma_1,\gamma_2\big\}.$$
Let $\lambda$ be a simple closed curve, which consists of a
union $\lambda=\lambda^+\cup\lambda^-$, with
$\lambda^\bullet\subset \ovl{R^\bullet},\ \bullet\in\{+,-\}$,
which cuts either of $\gamma_1$ and $\gamma_2$ in a
single transverse point. Sometimes we may assume that
the image of the homology class $[\lambda]\in\Ht_1(\partial X,\Z)$
in $\Ht_1(X,\Z)$ is trivial, i.e. that $\lambda$ bounds
a closed surface in $X$. In this situation, we will
say that $\lambda$ is homologically trivial in $X$.
It makes sense to
talk about Morse surgery along $\lambda$ on the sutured
manifold $X$ as will follow.\\

Consider two parallel copies of $\lambda$ which we will denote by
$$\lambda_1=\lambda_1^+\cup\lambda_1^-,\ \ \&\ \
\lambda_2=\lambda_2^+\cup\lambda_2^-.$$
Let $N_i$ be a
tubular neighborhood of $\gamma_i$ in $\partial X$, which may be identified
with the standard cylinder $N_i=S^1\times I$, where $I$ is the
unit interval. The pair of arcs $(\lambda_1\cup\lambda_2)\cap N_i$ may then
be pictured as $\{1,-1\}\times I$, where the area
between the parallel curves $\lambda_1$ and $\lambda_2$ is identified as
$$\big\{z\in S^1\ |\ \mathrm{Im}(z)>0\big\}\times I\subset S^1\times I=N_i.$$
We may change the arcs $\{1,-1\}\times I$ with the arcs
\begin{displaymath}
\begin{split}
&\big\{(e^{t\pi\sqrt{-1}},t)\ |\ t\in I=[0,1]\big\}\subset S^1\times I,\ \ \&\\
&\big\{(-e^{t\pi\sqrt{-1}},t)\ |\ t\in I=[0,1]\big\}\subset S^1\times I
\end{split}
\end{displaymath}
in both $N_1$ and $N_2$. With this change, the curves
$\lambda_1$ and $\lambda_2$ are replaced by new simple closed
curves $\delta_1$ and $\delta_2$. Let
$(X, \tau_\lambda)$ be the balanced
sutured manifold defined with
$$\tau_\lambda
=\big(\tau-\big\{\gamma_1,\gamma_2\big\}\big)\cup
\big\{\delta_1,\delta_2\big\}.$$
In order to see that $(X,\tau_\lambda)$ is balanced,
 it suffices to note that
$$\partial X-\tau_\lambda=\big((\partial X-\tau)-
(R^+\cup R^-)\big)\cup(S^+\cup S^-),$$
where $$S^+=\left(R^+-\mathrm{nd}(\lambda^+)\right)
\cup \mathrm{nd}(\lambda^-),\ \ \&\ \
S^-=\left(R^--\mathrm{nd}(\lambda^-)\right)\cup
\mathrm{nd}(\lambda^+),$$ and thus
$\chi(S^\bullet)=\chi(R^\bullet),\ \bullet\in\{+,-\}$.
The surgery is usually denoted by $(X,\tau)
\rightsquigarrow (X,\tau_\lambda)$ in this paper.
\\

Suppose that the simple closed curve $\lambda$ is chosen as above.
For $n=(n_1,n_2)$ in $\Z\oplus \Z$, let $\lambda(n)$ be the
simple closed curve obtained from $\lambda$ by winding it
 $n_1$ times around $\gamma_1$ (close to the intersection of
$\lambda$ with $\gamma_1$) and $n_2$ times around $\gamma_2$
(close to the intersection of
$\lambda$ with $\gamma_2$). The homology class represented by
 $\lambda(n)$ in $\Ht_1(\partial X,\Z)$ is $[\lambda]
 +n_1[\gamma_1]-n_2[\gamma_2]$.
 When there is no confusion, we will denote
$(X,\tau_{\lambda(n)})$ by $(X,\tau_n)$.\\

Let us assume that
$$H=\left(\Sig,\alphas=\{\alpha_1,...,\alpha_\ell\},
\betas=\{\beta_1,...,\beta_\ell\},\z=\{z_1,...,z_\el\}\right)$$
is a Heegaard diagram for the sutured manifold $(X,\tau)$.
Here $\Sig$ is a closed Riemann surface of genus
$g$, $\alphas$ and $\betas$ are $\ell$-tuples of simple
closed curves which are homologically linearly independent
in $\Sig-\z$, and $\z=\{z_1,...,z_\el\}$ is a
$\el$-tuple of marked points
determining the set $\tau$ of sutures on the
boundary of the three-manifold $X$.
Furthermore, let us assume that
$z_1,z_2$ are the marked points corresponding to
the sutures $\gamma_1$ and $\gamma_2$ respectively.
We may choose the Heegaard diagram so that $z_1$
and $z_2$ are in the same
connected component of $\Sig-\alphas-\betas_0$,
where $\betas_0$ denotes
the subset $\betas-\{\beta_\ell\}$ of $\betas$.
 In this diagram,
the framing $\lambda$ is determined as an arc
which joins $z_1$ to $z_2$ in $\Sig-\betas$, and may
be completed to a simple closed
curve by adding to it a short arc from $z_2$ to $z_1$,
meeting $\beta_\ell$ transversely in a single point, and staying
disjoint from $\alphas\cup\betas_0$.\\

The Heegaard diagram describing the surgery
$(X,\tau)\rightsquigarrow (X,\tau_\lambda)$ is
then obtained as follows. Let
$$H_\lambda=(\Sig,\alphas,\betas_\lambda
=\{\beta_1',...,\beta_{\ell}'\},\z),$$
where $\beta_i'$ for $i=1,...,\ell-1$ is an exact
Hamiltonian isotope of the curve $\beta_i$
which cuts it in a pair of canceling transverse
intersection points. Moreover, $\beta_\ell'$ is obtained
as an isotope of the simple closed curve on $\Sig$
associated with $\lambda$ which separates
the marked points $z_1$ and $z_2$ from each other.
Abusing the notation, we will sometimes denote $\beta_\ell'$ by $\lambda$.
This Heegaard diagram corresponds to the balanced sutured manifold $(X,\tau_\la)$.
The Heegaard triple
$$(\Sig, \alphas,\betas,\betas_\lambda,\z)$$
will then described the surgery  $(X,\tau)
\rightsquigarrow (X,\tau_\lambda)$ associated with
$\lambda$.\\

We are particularly interested in the case where the curves $\gamma_1$ and $\gamma_2$
in $\tau=\{\gamma_1,...,\gamma_\el\}$ used for the surgery are the boundary of a
a subset $B\subset \partial X$ which is homeomorphic to a cylinder $[0,1]\times S^1$
(and $\gamma_1$ and $\gamma_2$ are identified with $\{0\}\times S^1$ and  $\{1\}\times S^1$
respectively).  Then
for $n=(n_1,n_2)$ and $m=(m_1,m_2)$ in $\Z\oplus \Z$,
the sutured manifolds $(X,\tau_n)$ and $(X,\tau_m)$
are the same if $m_1+m_2=n_1+n_2$.
Thus in this case, it makes sense to talk about $(X,\tau_n)$ for $n\in\Z$, if the simple closed curve
$\lambda$ is fixed.\\

Let $(X,\tau)$ be a balanced sutured manifold as above.
Let $I$ denote a subset of $\{1,...,\el\}$. Consider the
sutured manifold $(X(I),\tau(I))$ obtained by filling out the sutures of $(X,\tau)$
corresponding to the subset $I$
 with solid cylinders $D^2\times [-1,1]$. In particular, we have
 $\ovl{X}= X(1,...,\el)$.
In terms of the Heegaard diagrams, if $(\Sig,\alphas,\betas,\z=\{z_1,...,z_\el\})$ is
a Heegaard diagram associated with $(X,\tau)$ so that $z_i$ corresponds to the
suture $\gamma_i$,
a diagram for $(X(I),\tau(I))$ will be the pointed Heegaard diagram
$$\left(\Sig,\alphas,\betas,\z-\big\{z_i\ |\ i\in I\big\}\right).$$

Let $(\Sigma,\mbox{\boldmath$\alpha$},\mbox{\boldmath$\beta$},{\bf z})$
be a  Heegaard diagram for the balanced sutured manifold $(X,\tau)$.
Consider the symmetric product
$$\mathrm{Sym}^\ell(\Sigma)=\frac{\Sigma^{\times \ell}}{S_\ell}
=\frac{\Sig\times ...\times\Sig}{S_\ell}$$
equipped with a path of complex structures of the form
$\big\{J_t=\mathrm{Sym}^\ell(j_t)\big\}_{t\in[0,1]}$,
which is induced from a path of complex structure $\{j_t\}_{t\in[0,1]}$
on $\Sigma$, such that the map $\Sigma^{\times \ell}\to \mathrm{Sym}^\ell (\Sigma)$
is $(j_t,J_t)$-holomorphic for all $t\in[0,1]$.
Then  $\mathbb{T}_{\alpha}=\alpha_1\times...\times\alpha_\ell$
and $\mathbb{T}_\beta=\beta_1\times...\times\beta_\ell$ are
totally real sub-manifolds of $\mathrm{Sym}^\ell(\Sigma)$.
We may define a map
$$\relspinc=\relspinc_\z:\Ta\cap \Tb \lra \RelSpinC(X,\tau),$$
which is defined by choosing a Morse function compatible with
the Heegaard diagram for the sutured manifold $(X,\tau)$, viewing an intersection point $\x\in\Ta\cap\Tb$ as
a set of flow lines joining index-$1$ critical points to index-$2$ critical points
of the Morse function, and perturbing the gradient vector field of the corresponding
Morse function in a neighborhood of this set of flow lines associated with $\x$ in order
to obtain a nowhere vanishing vector field on $X$ with the desired properties.\\

Denote the natural maps obtained by extending the
relative $\SpinC$ structures on sutured manifolds over the attached
solid cylinders by
\begin{displaymath}
\begin{split}
&s_I=s_I^{\tau}:\RelSpinC(X,\tau)\lra \RelSpinC(X(I),\tau(I)),\ \ \forall\ I\subset \{1,...,\el\}.
\end{split}
\end{displaymath}
In particular, $s_{\{1,...,\el\}}^\tau$ is the restriction map $[.]:\RelSpinC(X,\tau)\ra \SpinC(\ovl{X})$
defined before.
Note that there is an exact sequence
\begin{displaymath}
\begin{diagram}
0&\rTo &\big\langle \PD[\gamma_i]
\ |\ i\in I\big\rangle_\Z
&\rTo &\RelSpinC(X,\tau)&\rTo{s_I}&\RelSpinC(X(I),\tau(I))&\rTo&0.
\end{diagram}
\end{displaymath}
This sequence should be interpreted as follows. If two relative $\SpinC$ structures
$\relspinc,\relspinct\in\RelSpinC(X,\tau)$ satisfy $s_I(\relspinc)=s_I(\relspinct)$,
then the cohomology class $\relspinc-\relspinct$
is generated by the Poincar\'e duals of the sutures
corresponding to $I$.

\subsection{Relative $\SpinC$-structures and Heegaard diagrams}
Let the Heegaard diagram $(\Sigma,\mbox{\boldmath$\alpha$},\mbox{\boldmath$\beta$},{\bf z})$
 for the balanced sutured manifold $(X,\tau)$,
the symmetric product
$\mathrm{Sym}^\ell(\Sigma)$, the totally real tori $\Ta$ and $\Tb$,
and the path of complex structures
$\big\{J_t=\mathrm{Sym}^\ell(j_t)\big\}_{t\in[0,1]}$,
be as before.

\begin{defn}
Let $D\subset \C$ be the unit disk in the complex plane, and
$\x,\y\in \mathbb{T}_\alpha\cap\mathbb{T}_\beta$.
A {\emph{Whitney disk}} is a continuous map $\phi:D\to \mathrm{Sym}^\ell(\Sigma)$ such that
$\phi(-i)=\x,\phi(i)=\y$ and
\begin{displaymath}
\begin{split}
&\phi\big\{z\in\partial D\ |\ \mathrm{Re}(z)\ge 0\big\}\subset \mathbb{T}_\alpha\ \ \&\\
&\phi\big\{z\in\partial D\ |\ \mathrm{Re}(z)\le 0\big\}\subset \mathbb{T}_\beta.
\end{split}
\end{displaymath}
 The set of homotopy classes of Whitney disks connecting
$\x$ to $\y$ is denoted by $\pi_2(\x,\y)$.
For any homology class $\phi\in\pi_2(\x,\y)$,
we will denote the moduli space
of $\{J_t\}_t$-holomorphic representatives of
$\phi$ by $\Mod(\phi)$.
There exists a translation action of $\R$
on $\Mod(\phi)$. The quotient of
$\Mod(\phi)$ under this action will be
denoted by $\ov\Mod(\phi)$.
The Maslov index of $\phi$ is denoted by $\mu(\phi)$.
For $i\in\Z$, we will denote by $\pi_2^i(\x,\y)$
the subset of $\pi_2(\x,\y)$ which
consists of all $\phi$ with $\mu(\phi)=i$.
\end{defn}

It is known (\cite{OS-3m1}, and \cite{OS-linkinvariant})
that for any generic path
$\{J_t\}_t$ of complex structures, $\Mod(\phi)$
is a smooth manifold of dimension
$\mu(\phi)$, which is not necessarily compact.
In fact, this moduli space may be
compactified by adding the Gromov limits of
pseudo-holomorphic curves.
But the boundary strata which correspond to
 degenerations of the domain are not
necessarily of lower dimension. We will return to
this issue in section~\ref{sec:chain-complex}.

\begin{defn}
Let $\mathcal{D}_1,...,\mathcal{D}_m$ be the connected components of $\Sigma-\mbox{\boldmath$\alpha$}-\mbox{\boldmath$\beta$}$.
Each element of the free abelian group generated by
$\{\mathcal{D}_1,...,\mathcal{D}_m\}$ is called a
{\emph{domain}}. A domain $\mathcal{D}
=a_1\mathcal{D}_1+...+a_m\mathcal{D}_m$ is called \emph{positive},
denoted $\mathcal{D}\ge 0$,
if $a_i\ge 0$ for $1\le i \le m$. It is called
{\emph{periodic}} if its boundary is a sum of
$\alpha$ and $\beta$ curves.
\end{defn}

For every Whitney disk $\phi$ connecting intersection
points $\x$ and $\y$, the domain associated with
 $\phi$ is defined as follows:
$$\mathcal{D}(\phi)=\sum_{i=1}^{m}n_{p_i}(\phi)\mathcal{D}_i$$
where $p_i\in\mathcal{D}_i$ is a marked point.
Here $n_p(\phi)$ for a point $p\in\Sig-\alphas-\betas$
denotes the algebraic intersection number of
$\phi$ with the subvariety
$$\Delta_p=\Big\{(p_1,...,p_\ell)\in \mathrm{Sym}^\ell(\Sig)\ \Big|\ p_i=p,\
\text{for some }1\leq i\leq \ell\Big\}.$$
If the map $\phi$ is holomorphic then $\mathcal{D}(\phi)$
is obviously a positive domain.
We will denote by $\pi_2^+(\x,\y)$ the subset of $\pi_2(\x,\y)$ which
consists of all $\phi$ with $\Dcal(\phi)\geq 0$.
\\

If $\mathcal{P}$ is a periodic domain
we can associate to it a homology class in
$H_2(\ovl{X},\mathbb{Z})$. More precisely, let
$$\partial\mathcal{P}=\sum_{i=1}^{\ell}a_i\alpha_i
+\sum_{i=1}^{\ell}b_i\beta_i$$
and let $D_i$ be the union of $\alpha_i\times[-1,0]$ with
the core of the two-handles
attached to $\alpha_i\times\{-1\}$ in $X$. Similarly, let
 $D'_i$ be the union of $\beta_i\times[0,1]$ with the core of the
 two-handle attached to  $\beta_i\times\{1\}$. Define
$$H(\mathcal{P})=\mathcal{P}+\sum_{i=1}^{\ell}
a_iD_i+\sum_{i=1}^{\ell}b_iD'_i.$$

If $\phi\in\pi_2(\x,\x)$ is a Whitney disk connecting
$\x$ to itself, with $\x\in\Ta\cap\Tb$,
the domain $\Dcal(\phi)$ will be a periodic domain.
Conversely, any periodic domain
$\Pcal$ determines the class of a Whitney disk in
$\pi_2(\x,\x)$ for any
$\x\in\Ta\cap\Tb$. Thus the space of periodic
domains may be identified with
$\pi_2(\x,\x)$.\\

For each $\x\in \mathbb{T}_{\alpha}\cap\mathbb{T}_{\beta}$
let $\gamma_\x$ be the flow lines of a compatible Morse function connecting the index-$1$
critical points to the index-$2$ critical points passing through the union $\x$ of the intersection points
on $\Sig\times\{0\}\subset X$.

\begin{lem}
For $\x,\y\in \mathbb{T}_{\alpha}\cap\mathbb{T}_{\beta}$ we have $\relspinc(\x)-\relspinc(\y)=\mathrm{PD}(\epsilon(\x,\y))$ where
$\epsilon(\x,\y)=\gamma_\x-\gamma_\y\in \Ht_1(X,\mathbb{Z})$.
\end{lem}
\begin{proof}
This is lemma~4.7 from \cite{Juh}.
\end{proof}
\begin{cor}\label{cor:s(x)-s(y)}
If $\phi\in\pi_2(\x,\y)$ then we have
$$\relspinc(\x)-\relspinc(\y)=
\sum_{i=1}^\el n_{z_i}(\phi).\mathrm{PD}[\gamma_i].$$
\end{cor}
\begin{proof}
The disk $\phi$ gives a domain $\Dcal(\phi)$, with the
property that $\epsilon(\x,\y)$ is represented by
$$\partial (\Dcal(\phi))\in \Ht_1(X,\Z)
=\frac{\Ht_1\Big(\Sig-\Big\{z_1,...,z_\el\Big\},\Z\Big)}{\Big\langle \alpha_1,...,\alpha_\ell,\beta_1,...,\beta_\ell\Big\rangle_\Z}.$$
If $\epsilon_i$ denotes a small loop around $z_i\in \Sig$, the
domain $\Dcal(\phi)$  gives a $2$-chain connecting
$\epsilon(\x,\y)$ and $n_{z_1}(\phi)\epsilon_1+...+n_{z_\el}
(\phi)\epsilon_\el$. However, $\epsilon_i$ is
homologous to $\gamma_i$, and we thus have
$$\relspinc(\x)-\relspinc(\y)=\mathrm{PD}[\epsilon(\x,\y)]
=\sum_{i=1}^\el n_{z_i}(\phi)\mathrm{PD}[\gamma_i].$$
This completes the proof of the corollary.
\end{proof}

Let us finish this subsection with a lemma for
computing the Maslov index of a periodic domain.
Let
$$\Sigma-\mbox{\boldmath$\alpha$}=\bigcup_{i=1}^{k}A_i,
\ \ \&\ \ \Sigma-\mbox{\boldmath$\beta$}=\bigcup_{i=1}^{l}B_i,$$
and assume we have $m=k+l-1$ points $w_1,...,w_m$ on $\Sigma$
such that $w_i\in A_i\cap B_{1}$ for $1\le i\le k$,
and $w_{i+k}\in A_k\cap B_{i+1}$ for $1\leq i<l$.
\begin{lem}\label{lem:Maslov-1}
For any periodic domain $\mathcal{P}\in\pi_2(\x,\x)$
such that $n_{w_i}(\mathcal{P})=0$ for $1\le i\le m$ we have:
$$\mu(\mathcal{P})=\big\langle c_1([\relspinc(\x)])
,H(\mathcal{P})\big\rangle.$$
\end{lem}
\begin{proof}
Let $\Sigma_{\textbf w}=\Sigma-\mathrm{nd}(\bf{w})$,
with $\w=\{w_1,...,w_m\}$.
Now $(\Sigma_{\bf w},\mbox{\boldmath$\alpha$},
\mbox{\boldmath$\beta$})$ is a sutured Heegaard diagram
for a sutured manifold $X_{\bf w}$ which is obtained
from $\ovl X$ by removing neighborhoods of the flow lines
passing through $\bf{w}$. If $i:X_{\bf w}\to \ovl{X}$
is the embedding of $X_{\bf w}$ in $\ovl{X}$, then
$i^{-1}(\mathcal{P})$ is a periodic domain in
$(\Sigma_{\bf w},\mbox{\boldmath$\alpha$},\mbox{\boldmath$\beta$})$,
and by theorem 5.2 from \cite{Juh} we have
$$\mu(i^{-1}\mathcal{P})=\Big\langle c_1
\big(\relspinc_{\bf w}(\x)\big)
,H\big(i^{-1}(\mathcal{P})\big)\Big\rangle.$$
We have $i^*H(\mathcal{P})=H\big(i^{-1}(\mathcal{P})\big)$.
Thus it is enough to show that
\begin{equation}\label{eq:maslov-spinc}
c_1\big(\relspinc_{\bf w}(\x)\big)=i^*c_1\big([\relspinc(\x)]\big).
\end{equation}
Let $\nu$ be the vector field defining $\relspinc(\x)$, and
let $\ovl{\nu}$ be the extension of $\nu$ to $\ovl{X}$.
Then $i^*\ovl{\nu}$ is the vector field defining
$\relspinc_{\bf w}(\x)$ and thus equation~\ref{eq:maslov-spinc}
is satisfied.
\end{proof}
\begin{lem}\label{lem:Maslov-index}
For any periodic domain $\mathcal{P}\in\pi_2(\x,\x)$
 we have:
$$\mu(\mathcal{P})=\big\langle c_1(\relspinc(\x))
,H(\mathcal{P})\big\rangle.$$
\end{lem}
\begin{proof}
Moving the curves by isotopies does not change the two sides of the above
equality. We may thus assume that we have $m=k+l-1$ points
$w_1,...,w_m$ on $\Sigma$
such that $w_i\in A_i\cap B_{1}$ for $1\le i\le k$,
and $w_{i+k}\in A_k\cap B_{i+1}$ for $1\leq i<l$.
Let us denote $n_{w_i}(\Pcal)$ by $n_i$, and set
\begin{displaymath}
\Qcal=\Pcal-\Big(\sum_{i=1}^k n_i A_i\Big)-\Big(\sum_{i=1}^{l-1}(n_{i+k}-n_k)B_{i+1}\Big).
\end{displaymath}
Clearly $n_{w_i}(\Qcal)=0$ for $i=1,...,m$, and lemma~\ref{lem:Maslov-1} implies
(setting $\spinc=[\relspinc(\x)]$, and regarding $\Qcal$ as an element in $\pi_2(\x,\x)$)
\begin{equation}\label{eq:Maslov-P}
\begin{split}
\mu(\Qcal)&=\big\langle c_1(\spinc),H(\Qcal)\big\rangle\\
&=\big\langle c_1(\spinc),H(\Pcal)\big\rangle
-\Big(\sum_{i=1}^k n_i \big\langle c_1(\spinc), H(A_i)\big\rangle\Big)\\
&\ \ \ \ \ \ \ \ \ -\Big(\sum_{i=1}^{l-1}(n_{i+k}-n_k)\big\langle c_1(\spinc), H(B_{i+1})\big\rangle\Big)\\
&=\big\langle c_1(\spinc),H(\Pcal)\big\rangle
-\Big(\sum_{i=1}^k n_i \chi(A_i)\Big)-\Big(\sum_{i=1}^{l-1}(n_{i+k}-n_k)\chi(B_i)\Big).
\end{split}
\end{equation}
In the last equation we denote by $\chi(A_i)$ and $\chi(B_i)$ the expressions
$2-2g_{A_i}$ and $2-2g_{B_i}$, respectively, where $g_{A_i}$ and $g_{B_i}$ denote
the genera of the components in $\Rr^-(\tau)$ and $\Rr^+(\tau)$ which
correspond to $A_i$ and $B_i$, respectively.\\

On the other hand, the formula of Lipshitz (\cite{Robert-cylindrical}) may be used to compute
$\mu(A_i)$ and $\mu(B_j)$ as periodic domains in $\pi_2(\x,\x)$. As such,
we will have
\begin{equation}\label{eq:Maslov-Ai-Bi}
\mu(A_i)=\chi(A_i),\ \ i=1,...,k,\ \ \&
\ \ \mu(B_j)=\chi(B_i),\ \ j=1,...,l.
\end{equation}
Combining equations~\ref{eq:Maslov-Ai-Bi} and~\ref{eq:Maslov-P} we obtain
 \begin{displaymath}
\begin{split}
\mu(\Pcal)=
\big\langle c_1(\spinc),H(\Pcal)\big\rangle
&-\Big(\sum_{i=1}^k n_i \chi(A_i)\Big)-\Big(\sum_{i=1}^{l-1}(n_{i+k}-n_k)\chi(B_i)\Big)\\
&+\Big(\sum_{i=1}^k n_i \mu(A_i)\Big)+\Big(\sum_{i=1}^{l-1}(n_{i+k}-n_k)\mu(B_i)\Big)\\
=\big\langle c_1(\spinc),H(\Pcal)\big\rangle&.
\end{split}
\end{displaymath}
This completes the proof of the lemma.
\end{proof}
\newpage

\section{Algebra input}\label{sec:quasi-isomorphism}\label{sec:algebra}
\subsection{The $\ta$-chain complexes}
Let us assume that $\Ring$ is a (commutative)
finitely generated $\F$-algebra.
\begin{defn}
If $\PP$  is another (commutative) ring, together
with a homomorphism $\phi_\PP:\Ring\ra \PP$
we will call $(\PP,\phi_\PP)$ a {\emph{test ring}} for $\Ring$.
\end{defn}
In other words, a test ring is a ring that
has the structure of an $\Ring$-module.
We may then define $\phi_\PP(a)=a.1_\PP$ for $a\in\Ring$.
We will denote the image of an element $a\in\Ring$ in $\PP$
by $[a]^\PP$, and drop $\phi_\PP$ from the notation for simplicity.
\\

\begin{defn}
A chain complex with {\emph{coefficient ring}} $\Ring$,
or simply an $\ta$ {\emph{chain complex}}, is an $\Ring$-module
$C$, together with a homomorphism
of $\Ring$-modules $d:C\ra C$, such that $d\circ d=0$.
\end{defn}

 Let us assume that $(C,d)$ is an $\ta$ chain
 complex. Choose a test ring $\PP$ and let
 $C(\PP)=C\otimes_{\Ring}\PP$.
The differential $d$ of the complex $(C,d)$ induces a differential
$d^\PP:C(\PP)\ra C(\PP)$.

\begin{defn}
If $(C_1,d_1)$ and $(C_2,d_2)$ are $\ta$ chain complexes,
a homomorphism $f:C_1\ra C_2$ of $\Ring$-modules
is called an $\ta$ {\emph{chain map}}
if  $f\circ d_1=d_2\circ f$.
\end{defn}
The following lemma is an immediate consequence of the definitions.
\begin{lem}\label{lem:induced-chain-map}
Let $\Ring$ be as above and
suppose that $\PP$ is a test ring for $\Ring$.
If $(C_1,d_1)$ and $(C_2,d_2)$ are $\ta$ chain complexes and
$f:C_1\ra C_2$ is  an $\ta$ chain map, then $f$ induces a $\PP$ chain map
$$f^\PP:(C_1(\PP),d_{1}^\PP)\lra (C_2(\PP),d_{2}^\PP),$$
where $d_{i}^\PP$ denotes the differential induced by $d_i$
on  $C_i(\PP)$, $i=1,2$.
\end{lem}

Associated with any $\ta$-chain complex $(C,d)$, and any test ring $\PP$,
we consider the homology group
$$H_*(C,d;\PP):=H_*(C(\PP),d^\PP).$$
We may  denote this homology group by  $H_*(C;\PP)$,
if there is no confusion.
 If $(C_1,d_1)$ and $(C_2,d_2)$ are $\ta$ chain complexes and
$f:C_1\ra C_2$ is  an $\ta$ chain map, then the above lemma
implies that $f$ induces a homomorphism
$$f^\PP_*:H(C_1,d_1;\PP)\lra H(C_2,d_2;\PP).$$
\begin{defn}
An $\ta$-chain map $f:(C_1,d_1)\ra (C_2,d_2)$ between $\ta$ chain complexes
is called {\emph{null-homotopic}} if there is another
 $\ta$ chain map $H:(C_1,d_1)\ra (C_2,d_2)$
such that $f=H\circ d_1-d_2\circ H$.
$f$ is called a \emph{homotopy equivalence }of $\ta$ chain complexes
if there exist an $\ta$-chain map
$g:(C_2,d_2)\ra (C_1,d_1)$
such that $g\circ f-Id_{C_1}$ and $f\circ g- Id_{C_2}$ are null-homotopic.
$f:(C_1,d_1)\ra (C_2,d_2)$
is called a {\emph{quasi-isomorphism}} if the induced map
 $$f^\PP_*:H(C_1,d_1;\PP)\lra H(C_2,d_2;\PP)$$
 is an isomorphism for any test ring $\PP$.
 More generally, if $\mathfrak{B}$ is a family of test rings for $\Ring$,
 the $\ta$ chain map $f$ is called a $\mathfrak{B}${\emph{-isomorphism}} if
 $f^\PP_*$ is an isomorphism for any test ring $\PP\in \mathfrak{B}$.
 Two $\ta$ chain complexes
  $(C_1,d_1)$ and $(C_2,d_2)$ are {\emph{quasi-isomorphic}} if there is a
  third $\ta$ chain complex $(C,d)$, together with quasi-isomorphisms
  $f_i:(C_i,d_i)\ra (C,d)$, $i=1,2$. Similarly, we may define
  $\mathfrak{B}${\emph{-isomorphic}}
  $\ta$ chain complexes.
\end{defn}
\begin{lem}
If $f:(C_1,d_1)\ra (C_2,d_2)$ is a homotopy equivalence of $\ta$ chain
complexes, then $f$ is a quasi-isomorphism.
\end{lem}
\begin{proof}
If $g:(C_2,d_2)\ra (C_1,d_1)$ is the inverse of $f$
such that
$$g\circ f-Id_{C_1}=H\circ d_1-d_1\circ H,\ \
\&\ \ f\circ g- Id_{C_2}=K\circ d_2-d_2\circ K,$$
for homotopy maps $H$ and $K$, we obtain the
induced maps $f^\PP,g^\PP,H^\PP$ and $K^\PP$
over the induced complexes associated with any test
ring $\PP$. Thus, $f^\PP_*$ is an isomorphism
for any test ring $\PP$, and $g^\PP_*$ is its inverse.
\end{proof}
\subsection{The mapping cones of $\ta$-chain maps}
Most part of this sub-section is borrowed from Ozsv\'ath
and Szab\'o's \cite{OS-branched-double-cover} (subsection 4.1) with minor modifications.\\

If $(A_1,d_1)$ and $(A_2,d_2)$ are $\ta$  chain complexes
and $f:A_1\ra A_2$ is an $\ta$  chain map, we can form
the mapping cone $\M(f)$ of $f$, whose underlying complex
is the direct sum $A_1\oplus A_2$, which is equipped with
the differential
\begin{equation}
d_\M=\left(\begin{array}{cc}d_1&0\\ f&-d_2\end{array}\right).
\end{equation}
The chain complex $\M(f)$  inherits the structure of an
$\Ring$-module from $A_1$ and $A_2$, and its differential
respects the $\Ring$-module structure, since $d_1$ and $d_2$
 do so  and $f$ is an $\ta$ chain map. The following lemma
 follows immediately.\\

\begin{lem}
With the above notation, we have $\M(f)(\PP)=\M(f^\PP)$.
\end{lem}

There is a short exact sequence of $\ta$ chain complexes
\begin{displaymath}
\begin{diagram}
0&\rTo{}&A_2(\PP)&\rTo{\imath^\PP}&\M(f)(\PP)&\rTo{\pi^\PP}&A_1(\PP)&\rTo&0,
\end{diagram}
\end{displaymath}
induced from the natural sequence
\begin{displaymath}
\begin{diagram}
0&\rTo{}&A_2&\rTo{\imath}&\M(f)&\rTo{\pi}&A_1&\rTo&0.
\end{diagram}
\end{displaymath}
For each test ring $\PP$ for $\Ring$ we thus obtain
a long exact sequence in homology
\begin{displaymath}
\begin{diagram}
\dots&\rTo&H(A_2,d_2;\PP)&\rTo&H(\M(f),d_\M;\PP)&\rTo
&H(A_1,d_1;\PP)&\rTo{f^\PP_*}&H(A_2,d_2;\PP)&\rTo&\dots
\end{diagram}
\end{displaymath}
 The construction of the mapping cone is
natural in the sense that a commutative
diagram  of $\ta$ chain maps
\begin{displaymath}
\begin{diagram}
A_1&\rTo{f}&A_2\\
\dTo{\phi_1}&&\dTo{\phi_2}\\
B_1&\rTo{g}&B_2
\end{diagram}
\end{displaymath}
induces an $\ta$-chain map $\m(\phi_1,\phi_2):\M(f)\ra\M(g)$ such
that there is a homotopy commutative diagram with exact rows
\begin{displaymath}
\begin{diagram}
0&\rTo&A_2&\rTo{\imath_f}&\M(f)&\rTo{\pi_f}       &A_1&\rTo    &0\\
 &    &\dTo{\phi_2}&     &\dTo{\m(\phi_1,\phi_2)}&&\dTo{\phi_1}&&\\
0&\rTo&B_2&\rTo{\imath_g}&\M(g)&\rTo{\pi_g}       &B_1&\rTo    &0.\\
\end{diagram}
\end{displaymath}
The following lemma is the main algebraic ingredient
in the study of holomorphic triangles in this paper.

\begin{lem}\label{lem:quasi-isomorphism}{\bf{(c.f.
lemma 4.4 from \cite{OS-branched-double-cover})}}
Let $\{(A_i,d_i)\}_{i=1}^\infty$ be a collection of
$\ta$ chain complexes and $\{f_i:A_i\ra A_{i+1}\}$
be a collection of $\ta$ chain maps between these
complexes which satisfy the following two properties:\\
(1) There are $\ta$ homomorphisms $H_i:A_i\ra A_{i+2}$ such that
$$f_{i+1}\circ f_i=H_i\circ d_i+d_{i+2}\circ H_i,$$
i.e. $f_{i+1}\circ f_i$ is null-homotopic via
$\ta$ chain homotopy maps $H_i$.\\
(2)The difference
$$f_{i+2}\circ H_i-H_{i+1}\circ f_i:A_i\ra A_{i+3}$$
is a homotopy equivalence for $i=1,2,...$.\\
Then $\M(f_i)$ is homotopy equivalent to $A_{i+2}$
for $i\geq 2$. Moreover,
if $$f_{i+2}\circ H_i-H_{i+1}\circ f_i:A_i\ra A_{i+3}$$
is a $\mathfrak{B}$-isomorphism for some family
$\mathfrak{B}$ of test rings for $\Ring$ and
for $i=1,2,...$, then $\M(f_i)$ is $\Bb$-isomorphic
to $A_{i+2}$ for $i\geq 2$.
\end{lem}
\begin{proof}
The maps $\phi_i=(-1)^i\left(f_{i+2}\circ H_i-H_{i+1}\circ f_i\right):A_i\ra A_{i+3}$ are
$\ta$ chain maps,  making the following diagram
homotopy commutative
\begin{equation}\label{eq:commutativity}
\begin{diagram}
A_i&\rTo{f_i}&A_{i+1}\\
\dTo{\phi_i}&&\dTo{\phi_{i+1}}\\
A_{i+3}&\rTo{f_{i+3}}&A_{i+4}.
\end{diagram}
\end{equation}
In fact, using the first property in the
statement of the lemma we will have
$$\phi_{i+1}\circ f_i-f_{i+3}\circ \phi_i=(-1)^i\left(
(H_{i+2}\circ H_i)\circ d_i-d_{i+4}\circ (H_{i+2}\circ H_i)\right),$$
and $\phi_{i+1}\circ f_i-f_{i+3}\circ\phi_i$ is thus null-homotopic.
Let us denote $H_{i+2}\circ H_i$ by $L_i:A_i\ra A_{i+4}$.
We then define $\alpha_i:\M(f_i)\ra A_{i+2}$
and $\beta_i:A_i\ra \M(f_{i+1})$
by
$$\alpha_i(a_i,a_{i+1})=f_{i+1}(a_{i+1})-H_i(a_i),\ \ \&\ \
\beta_i(a_i)=(f_i(a_i),H_i(a_i))$$
 respectively. Then $\alpha_{i+1}\circ\beta_i=(-1)^i\phi_i$
 is a homotopy equivalence by the second property
above. All the squares in
the following diagram commute up to homotopy
\begin{equation}\label{eq:commutative-diagram}
\begin{diagram}
A_i&\rTo{f_{i}}&A_{i+1}&\rTo{\imath_{i+1}}
&\M(f_i)&\rTo{(-1)^{i+1}\pi_{i}}&A_i&\rTo{f_i} &A_{i+1}\\
\dTo{=} &    &\dTo{=}&     &\dTo{\alpha_i}
&&\dTo{\phi_i}&&\dTo{\phi_{i+1}}\\
A_i&\rTo{f_i}&A_{i+1}&\rTo{f_{i+1}}&A_{i+2}
&\rTo{f_{i+2}}&A_{i+3}&\rTo{f_{i+3}} &A_{i+4}\\
\dTo{\phi_i} &    &\dTo{\phi_{i+1}}&
&\dTo{\beta_{i+2}}&&\dTo{=}&&\dTo{=}\\
A_{i+3}&\rTo{f_{i+3}}&A_{i+4}&\rTo{(-1)^i\imath_{i+4}}
&\M(f_{i+3})&\rTo{\pi_{i+3}}&A_{i+3}&\rTo{f_{i+3}} &A_{i+4}.\\
\end{diagram}
\end{equation}
The commutativity of the two squares on the right
and the two squares on the left
already follows from the commutativity of the
square in equation~\ref{eq:commutativity}.
The definition of $\alpha_i$ and $\beta_{i+2}$
imply the equalities
$$f_{i+2}=\pi_{i+3}\circ \beta_{i+2},\ \ \&\ \
f_{i+1}=\alpha_i\circ \imath_{i+1}.$$
For the remaining two squares, let us define
\begin{displaymath}
\begin{split}
K^1_i:\M(f_i)\ra A_{i+3},\ \ \
&K^1_i(a_i,a_{i+1}):=H_{i+1}(a_{i+1}),\\
K^2_i:A_i\ra \M(f_{i+2}),\ \ \
&K^2_i(a_i)=(H_i(a_i),0)
\end{split}
\end{displaymath}
We can then compute
\begin{displaymath}
\begin{split}
&(-1)^{i+1}\phi_i\circ \pi_{i}-f_{i+2}\circ \alpha_i
=K^1_i\circ d_{\M_i}-d_{i+3}\circ K^1_i,\ \ \&\\
&\beta_{i+2}\circ f_{i+1}-(-1)^i\imath_{i+4}\circ \phi_{i+1}
=K_{i+1}^2\circ d_{i+1}+d_{\M_{i+3}}\circ K_{i+1}^2,
\end{split}
\end{displaymath}
where $d_{\M_i}$ denotes the differential of $\M_i=\M(f_i)$.
We first claim that $$F_i=\beta_{i+2}\circ
\alpha_i:\M(f_i)\ra \M(f_{i+3})$$
is a chain homotopy equivalence.
In fact, note that
\begin{displaymath}
\begin{split}
F_i(a_i,a_{i+1})&=\beta_{i+2}(f_{i+1}(a_{i+1})-H_i(a_i))\\
&=\Big(f_{i+2}\big(H_i(a_i)-f_{i+1}(a_{i+1})\big),
H_{i+2}\big(H_i(a_i)-f_{i+1}(a_{i+1})\big)\Big)\\
&=\m(\phi_i,\phi_{i+1})(a_i,a_{i+1})+\big(d_{\M_{i+3}}
\circ H^i+H^{i+3}\circ d_{\M_i}\big)(a_i,a_{i+1})\\
\text{where }&\begin{cases}
H^i(a_i,a_{i+1}):=\big(H_{i+1}(a_{i+1}),0\big),\\
\m(\phi_i,\phi_{i+1})(a_i,a_{i+1}):=\big((-1)^{i-1}\phi_i(a_i)
,(-1)^i\phi_{i+1}(a_{i+1})-L_i(a_i)\big)\end{cases}.
\end{split}
\end{displaymath}
Since $\m(\phi_i,\phi_{i+1})$ is a chain
homotopy equivalence, it follows that the same
is true for $F_i$. Since $\alpha_{i+3}\circ
\beta_{i+2}=(-1)^i\phi_{i+2}$ is a chain homotopy equivalence
as well, it follows that $\alpha_2$ is a
chain homotopy equivalence (one needs to use the fact that
$\alpha_{i+3},\beta_{i+2}$ and $F_i$ are  all
chain maps).\\

For the $\Bb$-isomorphism statement, note
that for any test ring $\PP\in\Bb$ for the ring
$\Ring$, we may replace the complexes $A_i$
with $A_i(\PP)$ and $\M(f_i)$ with
$\M(f_i^\PP)$ in the commutative
diagram~\ref{eq:commutative-diagram}. Then
 the maps induced on homology associated with the
first and the third row of the above diagram are exact.
From the five lemma, it follows that
the map induced on homology by
$\beta_{i+2}^\PP\circ \alpha_i^\PP$ is an isomorphism.
Since $\alpha_{i+3}\circ \beta_{i+2}=\phi_{i+2}$
is a $\Bb$-isomorphism, we conclude that $\beta_{i+2}$,
and hence $\alpha_{i}$ are $\Bb$-isomorphisms as well.
\end{proof}

\subsection{Filtration by a $\Z$-module}
Let us assume that $\Ring$ is an algebra over $\F$ which is
generated, as a free module over $\F$, by a set
$G(\Ring)$ of generators. We will assume that $1\in G(\Ring)$.
The choice of this  basis for $\Ring$ as
a free module over $\F$ will be implicit in our notation.
Furthermore, let $\Hbb$ be a $\Z$-module.
\begin{defn}
By a {\emph{filtration}} for $\Ring$ with values in $\Hbb$
we mean a choice of the basis $G(\Ring)$ for the free $\F$-module
$\Ring$, and a map
$$\chi:G(\Ring)\lra \Hbb$$
which satisfies $\chi(1)=0$ and $\chi(ab)=\chi(a)+\chi(b)$
for all $a,b\in G(\Ring)$.
The pair $(\Ring, \chi:G(\Ring)\ra \Hbb)$ is called
a {\emph{coefficient ring filtered by}} $\Hbb$.
We will typically drop $\chi$ and the choice of
$G(\Ring)$ from the notation, if there is no confusion, and
will denote the filtered ring by the pair $(\Ring,\Hbb)$.
\end{defn}
Suppose that $\phi_\Ringg:\Ring\ra \Ringg$ is
a test ring for $\Ring$ which is a free $\F$-module on its own with
basis $G(\Ringg)$, and that
$$\chi_\Ringg:G(\Ringg)\lra \Hbb$$
is a filtration for $\Ringg$.
\begin{defn}
We say that
$\chi_\Ringg$ is {\emph{compatible}}
with $\chi$ if
$$\phi_\Ringg(a)\in \big\langle \chi_\Ringg^{-1}
\big(\chi(a)\big)\big\rangle_\F\subset \Ringg,\ \
\forall \ \ a\in G(\Ring).$$
If this is the case, we will call the triple
$(\Ringg,\phi_\Ringg,\chi_\Ringg)$ a {\emph{filtered test ring}} for
$(\Ring,\Hbb)$. Again, when there is no confusion we will
denote this triple by $(\Ringg,\Hbb)$.
\end{defn}
Let us assume that $(C,d)$ is an $\Ring$ chain complex,  that
$\chi:G(\Ring)\ra \Hbb$ is a filtration for $\Ring$, and
 that $C$ is freely generated over $\Ring$ by  some subset $I$ of $C$.
\begin{defn}
We say that the $\ta$ chain complex $(C,d)$ is a {\emph{filtered}}
$\tab$ chain complex if there is a basis $I\subset C$
for $C$ over $\Ring$ and a {\emph{filtration}}
$$\chi:I\times I\ra \Hbb,$$
which satisfies:\\
1) $\chi(c_1,c_2)=-\chi(c_2,c_1)$ for all $c_1,c_2\in I$.\\
2) $\chi(c_1,c_2)+\chi(c_2,c_3)=\chi(c_1,c_3)$, for all $c_1,c_2,c_3\in I$.\\
3) For any $c\in I$, $d(c)=\sum_{i=1}^N a_ic_i$, with $c_1,...,c_N\in I$
and $a_1,...,a_N\in G(\Ring)$, such that
$$\chi(c,c_i)=\chi(a_i),\ \ \forall\ i\in\{1,...,N\}.$$
We are of course abusing the notation by denoting both the filtration of $C$
and the filtration of $\Ring$ by $\chi$.
\end{defn}
If $(C,d)$ is a filtered $\tab$ chain complex as above, one may think of
$\chi(c_1,c_2)$ as the difference $\chi(c_1)-\chi(c_2)$, where
$\chi:I\ra \Sbb$ and $\Sbb$ is
 an affine space  over the module $\Hbb$.
 With this new notation, $\chi(a.c)$ for $a\in G(\Ring)$ and $c\in I$ may be defined
 as $\chi(a)+\chi(c)\in \Sbb$.\\

 Clearly, taking the tensor
 product of  $(C,d)$ with any filtered test ring $(\Ringg,\Hbb)$ results in a
 filtered $(\Ringg,\Hbb)$ chain complex.\\

 \begin{defn}
 An $\ta$ chain map $f:(C_1,d_1)\ra (C_2,d_2)$ between $\tab$ chain complexes
 $(C_i,d_i)$ with basis $I_i$, $i=1,2$ and filtrations $\chi_1,\chi_2$ is called
 a {\emph{filtered}} $\tab$ chain map if for all $c,c'\in I_1$ we may write
 $$f(c)=\sum_{i=1}^N a_i b_i,\ \ f(c')=\sum_{i=j}^M a_j'b_j',\ \
 a_i,a_j'\in G(\Ring),\ \&\ b_i,b_j'\in I_2,$$
 such that for any $i=1,...,N$ and $j=1,...,M$ we have
 $$\chi_1(c,c')=\chi_2(b_i,b_j')+\chi(a_i)-\chi(a_j').$$
 \end{defn}

In particular, if for some affine space $\Sbb$ over $\Hbb$, there are maps
$\chi_i:I_i\ra \Sbb$ which satisfy $\chi_i(c_1,c_2)=\chi_i(c_1)-\chi_i(c_2)$ for
$i=1,2$ and $c_1,c_2\in I_i$, the above condition may be translated
to $\chi_2(a_ib_i)=\chi_1(c)$ for $i=1,...,N$, whenever  $f(c)=\sum_{i=1}^N a_i b_i$
with $a_i\in G(\Ring)$ and $b_i\in I_2$.\\

Similarly, we may define the notion of a chain homotopy  respecting
the filtrations (i.e. $\tab$ chain homotopy), and filtered $\tab$
chain homotopy equivalence. Mapping cones of filtered $\tab$ chain maps
are filtered $\tab$ chain complexes. Moreover, the following refinement of
lemma~\ref{lem:quasi-isomorphism} may be proved with with a similar argument.\\

\begin{lem}\label{lem:quasi-iso-filtered}
With the notation of  lemma~\ref{lem:quasi-isomorphism},
if the $\ta$ chain complexes $A_i$ are all filtered $\tab$ chain complexes,
the $\ta$ chain maps $f_i$, as well as the $\ta$-homomorphisms $H_i$ are
all $\tab$ filtered, and $f_{i+2}\circ H_i-H_{i+1}\circ f_i$ are all
filtered $\tab$ chain homotopy equivalences,
$\M(f_i)$ is filtered $\tab$ chain homotopy equivalent to $A_{i+2}$.
\end{lem}

\subsection{The algebra associated with the
boundary of a sutured manifold}\label{subsec:algebra-of-SM}
Let $(X,\tau)$ be a balanced sutured manifold. We will assume that
$$\Rr^-(\tau)=\bigcup_{i=1}^k R_i^-,\ \ \&\ \
\Rr^+(\tau)=\bigcup_{j=1}^l R_j^+.$$
Here $R_i^-$ and $R_j^+$ are the
connected components of $\Rr^-(\tau)$ and
$\Rr^+(\tau)$ respectively, for $i=1,...,k$
and $j=1,...,l$, and let $g_i^-$ denote the genus of $R_i^-$ and $g_j^+$ denote the genus
of $R_j^+$.
The set of sutures $\tau=\{\gamma_1,...,\gamma_\el\}$
determines an algebra over $\F$ as
follows.
Consider the free $\F$-algebra
$$\F[\el]:=\big\langle \la_1,...,\la_\el\big\rangle_\F
:=\F\big[\la_1,...,\la_\el\big],$$
and consider the following elements in it
\begin{displaymath}
\begin{split}
&\la^-=\la^-(\tau):=\sum_{i=1}^k \la(R_i^-),\ \ \&\ \
\la^+=\la^+(\tau):=\sum_{i=1}^l \la(R_i^+),\ \ \ \ \text{where}\\
&\la_i^-=\la(R_i^-):=\prod_{\gamma_j\subset \partial R_i^-}\la_j,\ \ \ \ \   1\leq i\leq k,\ \ \&\\
&\la_i^+=\la(R_i^+):=\prod_{\gamma_j\subset \partial R_i^+}\la_j,\ \ \ \ \   1\leq i\leq l.\\
\end{split}
\end{displaymath}
Consider the following quotients of $\langle \la_1,...,\la_\el\rangle$,
which are  all finitely generated algebras over $\F$:
\begin{displaymath}
\begin{split}
\Rin_\tau&:=\frac{\Big\langle \la_1,...,\la_\el\Big\rangle_\F}
{\Big\langle \la_i^-\ \big|\ g_i^->0\Big\rangle_{\F[\el]}+
\Big\langle \la_j^+\ \big|\ g_j^+>0\Big\rangle_{\F[\el]}},\\
\Ring_\tau&:=\frac{\Big\langle \la_1,...,\la_\el\Big\rangle_\F}
{\Big\langle \la^+(\tau)-\la^-(\tau)\Big\rangle_{\F[\el]}+
\Big\langle \la_i^-\ \big|\ g_i^->0\Big\rangle_{\F[\el]}+
\Big\langle \la_j^+\ \big|\ g_j^+>0\Big\rangle_{\F[\el]}},\ \ \&\\
\ov\Ring_\tau&:=\frac{\Big\langle \la_1,...,\la_\el\Big\rangle_\F}
{\Big\langle \la_i^-\ \big|\ i=1,...,k\Big\rangle_{\F[\el]}+
\Big\langle \la_j^+\ \big|\ j=1,...,l\Big\rangle_{\F[\el]}}.
\end{split}
\end{displaymath}
Clearly  $\Ring=\Ring_\tau, \ov\Ring=\ov\Ring_\tau$ and $\Rin=\Rin_\tau$
are all generated, as modules over $\F$, by
elements of the form $\prod_{i=1}^\el \la_i^{a_i}$, where $a_i$ are
non-negative integers. We will denote the set of all such
monomials by $G(\Ring)=G(\Rin)$. The ring $\Ring_\tau$ will be used
as the coefficient ring for the Ozsv\'ath-Szab\'o chain complexes
which we introduce in this paper, while $\Rin_\tau$ is used to make
the admissibility criteria for the Heegaard diagrams slightly stronger, and
is not really essential in our construction. The ring $\ov\Ring_\tau$
will appear when we discuss a generalization of Juh\'asz' surface
decomposition formula from \cite{Juh} later in \cite{AE-SFH}.\\

One may define a natural map, called the
Poincar\'e duality character in this paper, from $G(\Ring)$ to the $\Z$-module
$\Hbb=\Hbb_\tau:=\Ht^2(X,\partial X,\Z)$ by
\begin{displaymath}
\begin{split}
&\chi:G(\Ring)\lra \Hbb=\Ht^2(X,\partial X;\Z),\\
&\chi\Big(\prod_{i=0}^\el \la_i^{a_i}\Big):=
a_1\mathrm{PD}[\gamma_1]+...+a_\el\mathrm{PD}
[\gamma_\el],\ \ \forall \ a_1,..,a_\el\in\Z^{\geq 0}.
\end{split}
\end{displaymath}
As defined, $\chi$ is just a map from the set of
generators for $\F[\el]$ to $\Ht^2(X,\partial X,\Z)$.
However, since $\chi(\la(R_i^-))=-\mathrm{PD}[\partial R_i^-]=0$
and $\chi(\la(R_j^+))=\mathrm{PD}[\partial R_j^+]=0$
for all $i=1,...,k$ and $j=1,...,l$,
the map is well-defined on $G(\Ring)$ and $G(\Rin)$.
The Poincar\'e duality character gives  filtrations of
$\Ring$ and $\Rin$ by $\Hbb=\Ht^2(X,\partial X;\Z)$.
\\

We may also define a map from the set of positive
Whitney disks associated with a Heegaard diagram
$(\Sig,\alphas,\betas,\z)$ for $(X,\tau)$ to  $G(\Ring)$
and $G(\Rin)$
by computing the local multiplicities of
the domain associated with each  disk
at the marked points in $\z$:
\begin{displaymath}
\begin{split}
&\la=\la_\z:\coprod_{\x,\y\in\Ta\cap\Tb}\pi_2^+(\x,\y)\lra G(\Ring)\\
&\la(\phi):=\prod_{i=1}^\el \la_i^{n_{z_i}(\phi)},\ \ \
\forall\ \x,\y\in\Ta\cap\Tb,\ \&\  \forall\ \phi\in\pi_2^+(\x,\y).
\end{split}
\end{displaymath}
Similarly, we may define $\lam=\lam_\z:\pi_2^+(\x,\y)\ra G(\Rin)$.
Note theta $\la$ is just the composition of $\lam$ with the
quotient homomorphism $\Rin\ra \Ring$.
The composition $\chi(\lambda(\phi))$ in $\Ht^2(X,\partial X,\Z)$
will be denoted by $\Hh(\phi)$ for
any $\phi\in\pi_2^+(\x,\y)$.
Of course, the definition of
$\Hh(\phi)$ may be extended to arbitrary
$\phi\in\pi_2(\x,\y)$ by setting
$$\Hh(\phi)=\sum_{i=1}^\el n_{z_i}(\phi)\PD[\gamma_i].$$
Thus, corollary~\ref{cor:s(x)-s(y)} may be re-stated as
$$\phi\in\pi_2(\x,\y)\ \ \Rightarrow\ \
\relspinc(\x)-\relspinc(\y)=\Hh(\phi).$$

\begin{lem}\label{lem:no-nilpotents}
With the above notation, if $\la\in G(\Rin_\tau)$ is a monomial and
$\la^n=0$ for some positive integer $n\in\Z^+$, then $\la=0$ in $\Rin_\tau$.
In other words, $\Rin_\tau$ does not contain any non-trivial nilpotent monomials.
\end{lem}
\begin{proof}
It suffices to show that $\la^2=0$ implies $\la=0$. Since the monomials are linearly independent (over $\Q$) in
$\Z[\el]$, $\la^2=0$ implies that at least one monomial in
$$\Big\{\la_i^+\ \big|\ g_i^+>0\Big\}\bigcup\Big\{\la_i^-\ \big|\ g_i^->0\Big\}$$
divides $\la^2$ in $\Z[\el]$. Note that all the monomials in this set
are square free. Thus,
if $\la_i^+$ for some index $i$ with $g_i^+>0$
divides $\la^2$ in $\Z[\el]$, it should divide
$\la$ as well, and $\la=0$ in $\Rin$.
Similarly,  if $\la_i^-$ for some index $i$ with $g_i^->0$
divides $\la^2$  in $\Z[\el]$, then $\la=0$ in $\Rin$ and we are done.
\end{proof}

\begin{example}
If the boundary of the sutured manifold $(X,\tau)$
consists  of a torus $T$, and if
on  $T$ we have $2n$ parallel simple closed
curves $\mu_1,...,\mu_{2n}$, we may assume that
the corresponding sutured manifold is defined by the following data:
\begin{displaymath}
\begin{split}
&\tau=\{\mu_1,...,\mu_{2n}\},\ \ \mu_{2n+1}:=\mu_1,\\
&\Rr^+(\tau)=\bigcup_{j=1}^{n}R_{j}^+,\ \ \&\ \ \Rr^-(\tau)=\bigcup_{j=1}^{n}R_{j}^-\\
&\partial R_{j}^+=\mu_{2j-1}+\mu_{2j},\ \
\&\ \ \partial R_{j}^-=-\mu_{2j}-\mu_{2j+1},\ \ j=1,...,n.
\end{split}
\end{displaymath}
The sutured manifold
 $(X,\tau)$ determines a knot $K$ inside the closed
 three-manifold $Y$, obtained by gluing a $2$-handle to one of
$\mu_i$ and then a $3$-handle to the spherical boundary
of the new three-manifold. The torus $T$ may then be pictured as the boundary
of a neighborhood of $K$ in the resulting closed three-manifold.\\
Let $\la_j$ denote the variable associated with the suture $\mu_j$, $j=1,...,2n$.
Then, since all the components in $\Rr(\tau)$ are surfaces of genus zero,
the following relation is the only relation satisfied in $\Ring_\tau$:
$$\la_1\la_2-\la_2\la_3+...+\la_{2n-1}\la_{2n}-\la_{2n}\la_1=0.$$
In other words, we have
$$\Ring_\tau=\frac{\F\big[\la_1,...,\la_{2n}\big]}
{\big\langle \la_1\la_2+...+\la_{2n-1}\la_{2n}
=\la_2\la_3+...+\la_{2n}\la_1 \big\rangle}$$
In particular, for $n=1$, the above relation is trivial and
$\Ring_\tau=\F[\la_1,\la_2]$ is the coefficient ring used by Ozsv\'ath and Szab\'o
in defining $\mathrm{CF}^-(Y,K)$.
\end{example}

Suppose that $(X,\tau)$ is a balanced sutured manifold, and
$H=(\Sig,\alphas,\betas,\z=\{z_1,...,z_\el\})$ is a
Heegaard diagram associated with it, and
$\Ring_\tau$ be the
corresponding algebra.
Associated with the Heegaard diagram is a free
$\Ring_\tau$-module generated by the intersection points
$\x\in\Ta\cap\Tb$. We will denote this free
$\Ring_\tau$-module by $\CFT(X,\tau;H)$. We thus have
\begin{displaymath}
\begin{split}
\CFT(X,\tau;H):&=\Big\langle \x\ |\ \x\in \Ta\cap\Tb\Big\rangle_{\Ring_\tau}\\
&=\Big\langle a.\x\ |\ \x\in \Ta\cap\Tb,\ \&\ a\in G(\Ring_\tau)\Big\rangle_{\F}.
\end{split}
\end{displaymath}
The assignment of relative $\SpinC$ structures in $\Sbb=\Sbb_\tau=\RelSpinC(X,\tau)$
to the intersection points in $\Ta\cap\Tb$ by the map $\relspinc=\relspinc_\z$ may thus
be regarded as a filtration on the module $\CFT(X,\tau;H)$.
\newpage

\section{Admissible Heegaard diagrams}\label{sec:admissibility}
\subsection{The notion of $\spinc$-admissibility}
Let $(\Sigma,\mbox{\boldmath$\alpha$},\mbox{\boldmath$\beta$},{\bf z}=\{z_1,...,z_\el\})$
 be a Heegaard diagram for the balanced
sutured manifold $(X,\tau=\{\gamma_1,...,\gamma_\el\})$.
As before, we let $$\Sigma-\alphas=\coprod_{i=1}^{k}A_i,\ \ \&\ \
\Sigma-\betas=\coprod_{i=1}^{l}B_i.$$
\begin{defn}
Let $\overline{X}=X(1,...,\el)$ be the three manifold obtained by filling the
sutures in $\tau$. For $\spinc\in\SpinC(\overline{X})$, a
Heegaard diagram $(\Sig,\alphas,\betas,\z)$ is called $\spinc$-admissible if
for any nontrivial periodic domain $\Pcal$ with the property
$\langle c_1(\spinc),H(\Pcal)\rangle=0$
one of the following happens,
\begin{itemize}
\item[(a)] There is a point $w\in\Sig$ such that $n_{w}(\Pcal)<0$.
\item[(b)] We have $\Pcal\geq 0$ and $\lam(\Pcal)=0$ in $\Rin$.
\end{itemize}
\end{defn}
\begin{lem}\label{finiteD}
For $\spinc\in\SpinC(\ovl{X})$, let $(\Sig,\alphas,\betas,\z)$
be an $\spinc$-admissible Heegaard diagram for the balanced sutured manifold
$(X,\tau)$.
Then for any two intersection points $\x,\y\in\Ta\cap\Tb$ with
$\relspinc(\x),\relspinc(\y)\in\spinc\subset \RelSpinC(X,\tau)$,
and for any integer $j$ there are only finitely many
$\phi\in\pi_2(\x,\y)$ such that $\mu(\phi)=j$, $\Dcal(\phi)\geq 0$ and
$\lam(\phi)\neq0$.
\end{lem}
\begin{proof}
 Suppose that, for an integer $j$, there are infinitely many
 $\phi\in\pi_2(\x,\y)$ such that $\mu(\phi)=j$, $\Dcal(\phi)\geq 0$ and $\lam(\phi)\neq0$.
 Fix an element $\phi_0\in\pi_2(\x,\y)$ with these properties.
Any $\phi\in\pi_2(\x,\y)$ with these properties can then
be written as $\phi=\phi_0+\mathcal{P}$ where $\mathcal{P}\in\pi_2(\x,\x)$
and $\mu(\mathcal{P})=\langle c_1(\spinc),H(\Pcal)\rangle=0$.
Thus the set
 $$Q=\Big\{\mathcal{P}\in\pi_2(\x,\x)\ \Big|\ \mu(\mathcal{P})=0,\
 \mathcal{P}+\Dcal(\phi_0)\ge0, \lam(\phi_0+\Pcal)\neq 0\Big\}$$
 is not finite. Let $m$ denote the total number of domains in $\Sig-\alphas-\betas$, and
 $D_i$, for $i=1,...,m$, denote the corresponding domains. Consider
 $Q$ as a subset of the set of all lattice points
 in the  vector space
 $$V=\Big\langle \Pcal\in\pi_2(\x,\x)\ \big|\
 \langle c_1(\spinc),H(\Pcal)\rangle=0\Big\rangle_\R
 \subset \R^m.$$
 Here, $\R^m$ is the vector space generated by the domains ${D}_i, i=1,...,m$ over
 $\R$. If $Q$ is not finite,
 there is a sequence $(\Pcal_i)_{i=1}^\infty$ in $Q$, such that $\|\Pcal_i\|\to\infty$.
By passing to a subsequence if necessary, we may assume that
 the sequence $\{\frac{\Pcal_i}{\|\Pcal_i\|}\}_{i=1}^\infty$ is convergent on the
 unit ball inside $V$. Since $\|\Pcal_i\|\to\infty$, and $\Pcal_i+\Dcal(\phi_0)\geq 0$,
 it converges to a real vector in $\R^m$ with non-negative entries.
 Denote the limit of $(\frac{\Pcal_i}{\|\Pcal_i\|})$ by $\ti{\Pcal}$,
 which is a  periodic domain with non-negative real entries.
 There is a positive rational periodic domain $\Pcal$,
 sufficiently closed to $\ti{\Pcal}$, such that
 \begin{itemize}
\item[$\bullet$] The Maslov index $\mu(\Pcal)=0$, i.e. $\Pcal\in V$
\item[$\bullet$] If the coefficients of $\ti \Pcal$ in  some domain $D_i$ is zero,
the coefficient of  ${\Pcal}$ in $D_i$ is zero as well.
\end{itemize}
Thus for a sufficiently large number $M$, $M\ti{\Pcal}-\Pcal$ is a positive periodic domain.
After multiplying $\Pcal$ with an appropriate positive integer $N$,
we obtain a positive  periodic
 domain $N\mathcal{P}$ with integral coefficients, and  with Maslov index zero i.e.
 $\langle c_1(\spinc),H(N\Pcal)\rangle=0$. The $\spinc$-admissibility
 condition implies that $\lam(N\Pcal)=0$.
 Since
 $$\lim_{i\ra \infty}\frac{MN}{\|\Pcal_i\|}(\Dcal(\phi_0)+\Pcal_i)=MN\ti{\Pcal},$$
 and $\Dcal(\phi_0)+\Pcal_i\geq 0$, there exists a sufficiently large
 integer $K>0$ such that
 $$\frac{MN}{\|\Pcal_i\|}(\Dcal(\phi_0)+\Pcal_i)-N\Pcal\geq 0,\ \ \ \forall\ i>K.$$
 Note that $\|\Pcal_i\|\to\infty$, thus for an $K$ sufficiently large we have
 $$\frac{M}{\|\Pcal_i\|}<1,\ \ \&\ \
 N(\Dcal(\phi_0)+{\Pcal_i})-N\Pcal\geq 0,\ \ \forall\ i>K.$$
 The equality $\lam(N\Pcal)=0$ implies that
 $\lam(\phi_i)^N=\lam(N(\Dcal(\phi_0)+\Pcal_i))=0$ for any $i>K$.
 Since there are no nilpotent monomials in $\Rin_\tau$ by lemma~\ref{lem:no-nilpotents},
  we should have $\lam(\phi_i)=0$ for any $i>K$,
  which is in contradiction with the assumption that the map $\lam$ is nonzero over the classes $\phi_i$.
\end{proof}
\begin{remark}\label{remark:weak-admissibility}
When we use a test ring $\Ringg$ for $\Ring_\tau$ (which comes together
with a ring homomorphism $\rho_\Ringg:\Ring_\tau \ra \Ringg$)
as the ring of coefficients for the
chain complex, it suffices to assume that the Heegaard diagram $(\Sig,\alphas,\betas,\z)$
is admissible in the following weaker sense: If $\Pcal$ is a periodic domain with
$\Pcal\geq 0$ and $\langle c_1(\spinc), H(\Pcal)\rangle =0$, then
$\rho_\Ringg(\la(\Pcal))=0$. In particular, for $\Ringg=\Z$ this gives
us the notion of weak admissibility used by Juh\'asz \cite{Juh}. More generally, define
\begin{displaymath}
\Ringg_\tau=\frac{\Big\langle \la_1,...,\la_\el\Big\rangle_\Z}
{\Big\langle\prod_{i=1}^\el \la_i^{n_i}\neq 1\ |\ n_i\in\Z^{\geq 0}\ \&\
\sum_{i=1}^\el n_i[\gamma_i]=0\ \text{in } \Ht_1(X;\Z)/\mathrm{Tors}\Big\rangle}.
\end{displaymath}
Clearly $\Ringg_\tau$ is a quotient of $\Ring_\tau$. Let us denote the quotient map
by $$\rho_\tau:\Ring_\tau\lra \Ringg_\tau.$$
Any positive periodic domain
$\Pcal$ with $\la_\z(\Pcal)=\prod_{i=1}^\el \la_i^{n_i}$ determines a $2$-chain in
$X$ with boundary equal to $\sum_{i=1}^\el n_i\gamma_i$.
This implies that $\rho_\tau(\la_\z(\Pcal))=0$, unless $n_1=n_2=...=n_\el=0$.
Thus, the notion of admissibility for the coefficient ring $\Ringg_\tau$ is a
direct consequence of weak admissibility in the sense of Juh\'asz \cite{Juh}.
\end{remark}
\subsection{Existence of $\spinc$-admissible Heegaard diagrams}.
Performing special isotopies on the curves in $\alphas$, as in  \cite{OS-3m1},
produces $\spinc$-admissible Heegaard diagrams.
\begin{defn}
Let $\gamma$ be an oriented simple closed curve in $\Sigma$.
Consider the coordinate system $(t,\theta)\in (-\epsilon,\epsilon)
\times S^1$ in a neighborhood of $\gamma=\{0\}\times S^1$.
The diffeomorphism of $\Sigma$ obtained by integrating a vector
field $\zeta$ supported in this neighborhood of $\gamma$ with the property
$d\theta(\zeta)> 0$ is called {\emph{winding}} along $\gamma$.
Let $\alpha$ be a simple closed curve which intersects $\gamma$
in one point and $\phi$ be a winding around $\gamma$. If
$\phi(\alpha)$ intersects $\alpha$ in $2n$ points then we
say that $\phi$ is an isotopy which {\emph{winds}} $\alpha$ $n$-times around $\gamma$.
\end{defn}
\begin{lem}\label{lem:s-admissible}
Let $(X,\tau)$ be a balanced sutured manifold as before,
$\ovl{X}$ be the three-manifold obtained from
$(X,\tau)$ by filling the sutures,
and $\spinc\in\SpinC(\ovl{X})$ be a  $\SpinC$-structure.
Then $(X,\tau)$ admits an $\spinc$-admissible Heegaard diagram.
Moreover, every Heegaard diagram $(\Sig,\alphas,\betas,\z)$ for $(X,\tau)$ may be
modified to an $\spinc$-admissible Heegaard diagram by performing isotopies (supported
away from the marked points) on the curves in $\alphas$.
\end{lem}
\begin{proof}
Let $(\Sig,\alphas,\betas,\z)$ be a Heegaard diagram for $(X,\tau)$. Let
$$\Sig-\alphas=\coprod_{i=1}^{k} A_i,\ \ \&\ \ \Sig-\betas=\coprod_{i=1}^{l}B_i,$$
be the connected components in the complement of $\alphas$ and $\betas$ respectively.
For any $i\in\{1,...,k\}$ and $j\in\{1,...,l\}$
we can find a curve $\gamma$ such that $\gamma$
connects $\partial A_i$ to $\partial B_j$, and avoids $A_i$ and
$B_j$. By finger moving those $\alpha$ curves which intersect the
curve $\gamma$ (simultaneously) along $\gamma$, we create a new
Heegaard diagram with the property that $A_i\cap B_j\neq \emptyset$.
Repeating this procedure for all $i=1,...,k$ and $j=1,...,l$,
we may thus assume that $$A_i\cap B_j\neq\emptyset,\ \ \forall\ i=1,...,k,\ j=1,...,l.$$

Let ${\alphas}_0$ be a set of disjoint simple closed curves
on $\Sig$, disjoint from $\alphas$,
such that $\Sig-\alphas-{\alphas}_0$ has the same number of
connected components as $\Sig-\alphas$,
and all of its connected components have genus zero.
Furthermore, let $\alphas=\alphas_1\cup\alphas_2$ where
$$\alphas_2=\Big\{\alpha_i\in\alphas\ \big|\ \exists j\ \
\alpha_i\subset\partial A_j\Big\}.$$
For $i=0,1,2$, let us denote the number of elements in
$\alphas_i$ by $\ell_i$.
Thus, in particular, $\ell=\ell_1+\ell_2$.\\

We define a graph $G$ with $k$ vertices corresponding
to the domains $A_1,...,A_{k}$. The edges of $G$ correspond
to the elements of $\mbox{\boldmath$\alpha$}_2$, i.e. if
$\alpha\in \mbox{\boldmath$\alpha$}_2$ is a curve in
$\partial A_i\cap\partial A_j$, we put an edge in $G$ connecting $A_i$ to $A_j$
associated with $\alpha$. If $
=\Sigma[\alphas_1]$ is the surface obtained from $\Sigma$
by surgering out the elements of $\alphas_1$,
each loop in $G$ corresponds to a homologically nontrivial
simple closed curve in $\Sig[\alphas_1]$ which is disjoint from $\alphas_0$.
In other words, each loop in $G$ corresponds to a homologically nontrivial
simple closed curve in $\Sig[\alphas_0\cup\alphas_1]$.
Furthermore,
$h=\mathrm{dim}(\Ht_1(G,\mathbb{Z}))$
is the genus of $\Sig[\alphas_0\cup\alphas_1]$. One may easily compute
$h=\ell_2-k+1$.\\

Consider a set of pairwise disjoint simple closed curves
$$\gammas=\gammas_1\cup\gammas_2=\big\{\gamma^1,...,\gamma^{\ell_1}
\big\}\bigcup\big\{\gamma^{\ell_1+1},...,\gamma^{\ell-k+1}\big\}$$
on $\Sigma$ with the following properties.
First of all, we assume that $\gammas_1$ is a dual set for
$\alphas_1$ i.e. each element of $\gammas_1$
intersects exactly one element of $\alphas_1$ with intersection
number one, and for each element of $\alphas_1$ there is one element
of $\gammas_1$ intersecting it (with intersection number one).
Furthermore, we assume that
$$\gammas_1\cap\alphas_0=\gammas_1\cap\alphas_2=\emptyset.$$
The set  $\gammas_2$
corresponds to a basis for $\Ht_1(G,\mathbb{Z})$ which is a set of disjoint, oriented,
and linearly independent
simple closed curves on $\Sigma[\alphas_0\cup\alphas_1]$.
There is a one to one map $i:\gammas_2\ra\alphas_2$ with the property that
for each $\gamma\in\gammas_2$ the curve $i(\gamma)$
has nonempty intersection with $\gamma$.
 In fact if this is not true, by Hall's theorem  there is a subset
 of $\gammas_2$ with $n$ elements, say
 $\{\gamma^{i_1},...,\gamma^{i_n}\}\subset\gammas_2$, such that for
 $$A=\Big\{\alpha\in\alphas_2\ \Big|\ \exists j\in\{1,...,n\}\ \ s.t.\ \
 \alpha\cap\gamma^{i_j}\neq\emptyset\Big\}$$
we have $|A|<n$. Since the sum of the genera of the connected components of
$\Sigma[\alphas]$ is $\ell_0$,
$\Sigma[\alphas-A]$ is a surface whose genus is less than or equal to
$|A|+\ell_0$.
Furthermore, the curves in $\{\gamma^{i_1},...,\gamma^{i_n}\}
\cup{\alphas}_0$ are linearly independent in
$$\Ht_1(\Sigma[\alphas-A],\mathbb{Z}).$$
Thus the genus of $\Sigma[\alphas_2-A]$ is at least $n+\ell_0$ which
is in contradiction with the assumption $|A|<n$.\\

Choose a parallel copy of each curve $\gamma^i$ for $i=1,...,\ell-k+1$,
with the opposite orientation and denote it by $\ba\gamma^i$.
We will assume that $\ba\gamma^i$ is drawn on $\Sig$ very close to $\gamma^i$.
Let $v_i\in\gamma^i$ be points which are not contained in any of the $\alphas$ or $\betas$ curves for any $1\le i\le \ell+k-1$ and denote the corresponding points on $\ba\gamma^i$ by $\ba v_i$.
For any integer $N$, by {\it winding the $\alpha$ curves $N$ times along the $\gamma$ curves}
we mean winding all the $\alpha$-curves which
cut $\gamma^i$ (and hence $\ba\gamma^i$) $N$ times along $\gamma^i$ and $N$ times along
$\ba\gamma^i$, for any of the curves $\gamma^i, i=1,...,\ell-k+1$.
The windings around either of $\gamma^i$ and $\ba\gamma^i$ will be done simultaneously
for all the $\alpha$ curves, so that the new $\alpha$-curves remain disjoint
from each other.\\

Let $Q$ be the $\Q$-vector space generated by the
periodic domains $\Pcal$ such that
$$\mu(\Pcal)=\langle c_1(\spinc),H(\Pcal)\rangle=0.$$
One may write  $Q$ as a direct sum
$$Q=\Big(Q\cap \big\langle A_1,...,A_k,B_1,...,B_l\big\rangle_\Q\Big)\oplus P,$$
for sum subspace $P$ of $Q$ which is generated by the periodic domains
$\{\Pcal_1,...,\Pcal_b\}$.  Thus any periodic domain $\Pcal$ in the vector space
$Q$ is of the form
$$\Pcal=\sum_{i=1}^{b}q_i\Pcal_i+\sum_{i=1}^{k}a_iA_i+\sum_{i=1}^{l}b_iB_i,$$
where of course the coefficients $a_i$ and $b_j$ for $i=1,...,k$ and
 $j=1,...,l$ should satisfy the relation
 $$\mu\big(\sum_{i=1}^ka_iA_i+\sum_{i=1}^lb_iB_i\big)
 =\sum_{i=1}^ka_i(2-2g(A_i))+\sum_{i=1}^{l}b_i(2-2g(B_i))=0.$$
 In the above expression $g(A_i)$ and $g(B_j)$ denote the genus
 of $A_i$ and $B_j$ respectively, for $i=1,...,k$ and $j=1,...,l$.\\

Corresponding to any curve $\gamma^i\in\gammas$ we define
a map $p_{\gamma^i}$ from $P$ to $\Q$. If $\Pcal\in P$ is a
periodic domain and
$$\partial \Pcal=\sum_{i=1}^\ell p_i\alpha_i+\sum_{i=1}^{\ell}
q_i\beta_i=\partial_{\alpha}\Pcal+\partial_{\beta}\Pcal,$$
we may define the functions $p_{\gamma^i}$ by
$$p_{\gamma^i}(\Pcal)=\sum_{j=1}^\ell p_j.\#(\gamma^i.\alpha_j),
\ \ \forall\ i\in\{1,...,\ell-k+1\}.$$
Here $\#(\gamma^i.\alpha_j)$ denotes the intersection
number of $\gamma^i$ with $\alpha_j$. If for some periodic
domain $\Pcal\in P$ we have $p_{\gamma^i}(\Pcal)=0$, for $1\le i\le \ell-k+1$,
we may conclude that
\begin{displaymath}
\begin{split}
&\#(\partial_\alpha \Pcal.\gamma^i)=p_{\gamma^i}(\Pcal)=0,\ \ \
\forall\ \ 1\le i\le \ell-k+1\\
\Rightarrow\ \ &\partial_\alpha\Pcal=\partial\big(\sum_{i=1}^{k}
a_iA_i\big),\ \ \ \text{for some }a_1,...,a_{k}\in\Q \\
\Rightarrow\ \ &\partial\big(\Pcal-\sum_{i=1}^{k}a_iA_i\big)
\in\big\langle \beta_1,...,\beta_\ell\big\rangle_\Q\\
\Rightarrow\ \ &\Pcal=\sum_{i=1}^{k}a_iA_i+\sum_{i=1}^{l}b_iB_i,
\ \ \ \text{for some }b_1,...,b_l\in\Q\\
\end{split}
\end{displaymath}
From the assumption $\Pcal\in P$ we have $\Pcal=0$. Thus the map
\begin{displaymath}
\begin{split}
&e:P\lra \Q^{\ell-k+1}\\
&e(\Pcal):=(p_{\gamma^1}(\Pcal),p_{\gamma^2}(\Pcal),...,
p_{\gamma^{\ell-k+1}}(\Pcal))
\end{split}
\end{displaymath}
is one to one.
By a change of basis in $P$, and changing the order curves in $\gammas$ if necessary,
we can assume that
$$\pi_{i}(e(\mathcal{P}_j))=\delta_{ij}\ \ \ \forall\ \ 1\le i,j \le b,$$
where $\pi_i:\Q^{\ell-k+1}\ra \Q$ is the projection over the $i$-th  factor.\\

We would first like to show that for any positive periodic domain $\Qcal$ in $Q$,
 which is not included in the vector space generated by $A_i$'s and $B_j$'s,
 there is an integer $N=N(\Qcal)$ such that by winding
 $\alpha$-curves $N$ times along the curves in $\gammas$
 (in both positive and negative directions) the new
periodic domain obtained from $\Qcal$ will have some negative coefficient.
Let
$$\Qcal=\sum_{i=1}^b q_i\Pcal_i+\sum_{i=1}^{k}a_iA_i+\sum_{i=1}^{l}b_iB_i$$
be a positive periodic domain in $Q$ such that there is an index  $i$ so that $q_i\neq 0$.
Then we may choose an integer $N$ such that
$$|q_i|N>\max\big\{n_{v_i}(\Qcal),n_{\bar{v}_i}(\Qcal)\big\}.$$
Wind the $\alpha$ curves $N$ times along $\gamma$ curves.
In the new diagram (obtained after the above winding procedure) let $$\Big\{\Pcal'_1,...,\Pcal'_b,A_1',...,A'_{k},B_1',...,B_{l}\Big\}$$
be the new set of periodic domains obtained from
$$\Big\{\Pcal_1,...,\Pcal_b,A_1,...,A_{k},B_1,...,B_{l}\Big\}.$$
 Note that we may compute the coefficients of these new domains at $v_i$ and
 $\ba v_i$ from the following equations
\begin{displaymath}
\begin{split}
&n_{v_i}(A'_j)=n_{v_i}(A_j),\ \ \forall\ j=1,...,k,\ \
n_{v_i}(B'_j)=n_{v_i}(B_j),\ \ \forall\ j=1,...,l,\\
&\ \ \ \&\ \ \
n_{v_i}(\Pcal'_j)=\begin{cases}
      n_{v_i}(\Pcal_j)\ \ \ &\emph{if}\ i\neq j\\
      n_{v_i}(\Pcal_j)+N\ \ \ &\emph{if}\ i=j\\
      \end{cases}.\ \
\end{split}
\end{displaymath}
Similar equations are satisfied for the local coefficients at $\ba v_i$. In fact, we will have
$n_{\ba v_i}(\Pcal_i')=n_{\ba v_i}(\Pcal_i)-N$, while the rest of local coefficients remain
unchanged.
If $q_i<0$ we thus have  $$n_{v_i}(\Qcal')=n_{v_i}(\Qcal)+q_iN<0,$$
and if $q_i>0$ then $n_{\ba{v}_i}(\Qcal')<0$.\\

To finish the proof, first suppose that there is an integer $N$ such that, after winding the $\alpha$
curves $N$ times along the curves in $\gammas$,  any periodic
domain $\Qcal\in Q$ with integer coefficients either
has some negative coefficient or $\lam(\Qcal)=0$.
Then we are clearly done with the proof of the lemma. So, let us
assume otherwise, that for any integer $n$ there exists a periodic domain
$\Qcal_n$ with integer coefficients in $Q$ such that after winding
the $\alpha$ curves $n$
times along the curves in $\gammas$,  the resulting domain
$\Qcal_n'$ is positive and satisfies  $\lam(\Qcal_n')\neq 0$.
Let $\{\Qcal_n\}_{n=1}^\infty$ be the sequence constructed from these
elements of $Q$. As in the proof of lemma~\ref{finiteD},
after passing to a subsequence if necessary, we may assume that
the sequence
$\{\frac{\Qcal_n}{\|\Qcal_n\|}\}_{n=1}^\infty$ is convergent.

Let us assume that
$$\ti{\Qcal}=\lim_{i\ra \infty}\frac{\Qcal_n}{\|\Qcal_n\|}\in Q\otimes_\Q\R.$$
If $\ti{\Qcal}$ is not in the real
vector space generated by $A_i$'s and $B_j$'s,  there is an
integer $N$ with the property  that after winding the $\alpha$-curves  $N$ times
along all  the curves in $\gammas$, the resulting domain $\ti{\Qcal}'$
will have some negative coefficient. So there is an integer $K$ such that
for any
$i>K$, $\Qcal_i'$ has some negative coefficient after winding
the $\alpha$-curves  $N$ times along $\gammas$.
This is in contradiction with the definition of $\Qcal_i$ if
$i>N$. Thus $\ti{\Qcal}$ may be written as
\begin{equation}\label{eq:Qcal}
\ti{\Qcal}=\sum_{i=1}^{k}a_iA_i+\sum_{i=1}^{l}b_iB_i
\end{equation}
for some coefficients $a_i$ and $b_i$ in $\R$.\\

Note that $\ti\Qcal\geq 0$, which implies that for any
$w\in \Sig$, we have
$$\sum_{i=1}^{k}a_in_w(A_i)+\sum_{i=1}^{l}b_in_w(B_i)\geq0.$$
Since $A_i\cap B_j\neq \emptyset$ for $i=1,...,k$ and $j=1,...,l$,
we may pick $w=w_{ij}$ to be a point in this intersection.
But for this choice of $w$, the above inequality reads as $a_i+b_j\geq 0$.
If $b_j$ is the smallest of all $b_1,...,b_l$, the above
consideration implies that $a_i+b_j\geq0$ for all $i=1,...,k$. We may thus write
$$\ti{\Qcal}=\sum_{i=1}^{k}(a_i+b_j)A_i+\sum_{i=1}^{l}(b_i-b_j)B_i,$$
since $\sum_iA_i=\sum_iB_i=\Sig$. However, all the coefficients in the
above expression are non-negative. As a result,
after replacing these new coefficients, we may assume that
the real numbers $a_i$ and $b_j$ in equation~\ref{eq:Qcal} are positive. \\

As in the proof of
lemma~\ref{finiteD}, choose a positive rational periodic
domain $$\Qcal=\sum_{i=1}^{k}a'_iA_i+\sum_{j=1}^{l}b'_jB_j,$$
which is sufficiently close to $\ti{\Qcal}$, and such
that the coefficient of $\Qcal\in Q$   in the domains where $\ti{\Qcal}$
has zero coefficient is zero as well. As before, there are integers $N$
and $M$ such that $N\Qcal$ is a periodic domain with integer coefficients
and $M\ti{\Qcal}-\Qcal>0$. The positivity of the coefficients of $N\Qcal$
imply that $\lam(N\Qcal)=0$. Moreover, there is some positive integer $K>0$
such that for $i>K$ we have
\begin{displaymath}
\begin{split}
&MN\frac{\Qcal_i}{\|\Qcal_i\|}-N\Qcal\geq 0,\ \&\ \frac{M}{\|\Qcal_i\|}<1
\\
\Rightarrow\ \ &N\Qcal_i-N\Qcal\geq 0.
\end{split}
\end{displaymath}
This means that for $i>K$, we have  $\lam(N\Qcal_i)=\lam(\Qcal_i)^N=0$.
Thus we may conclude, using lemma~\ref{lem:no-nilpotents}, that
 $\lam(\Qcal_i)=0$ for $i>K$. This is in clear contradiction with our
 assumption on the integral periodic domains $\Qcal_i$.\\

The above argument shows that there is an integer $N$ with the property
 that after winding the curves in $\alphas$ a total of $N$ times along the
 curves in $\gammas$  we obtain an $\spinc$-admissible Heegaard diagram.
This completes the proof of the lemma.
\end{proof}
\begin{remark}\label{remark:admissibility}
The argument of lemma~\ref{lem:s-admissible} may be extended to show that
for any balanced sutured manifold $(X,\tau)$ and any $\SpinC$ class
$\spinc\in\SpinC(\ovl X)$ there is a Heegaard diagram $(\Sig,\alphas,\betas,\z)$
which is admissible in the following stronger sense. If $\Pcal$ is a
periodic domain with $\Pcal\geq 0$, and $\lam(\Pcal)\neq 0$ in $\Rin_\tau$
then  $$\big \langle c_1(\spinc), H(\Pcal)\big\rangle > 0.$$
When there are genus zero components in $\Rr(\tau)$, the above criteria
is the same as $\spinc$-admissibility condition. However, in certain situations
where all the connected components of $\Rr(\tau)$ have positive genus, using
such Heegaard diagrams may be useful. We face this situation in section~\ref{sec:stabilization}.
\end{remark}
\newpage

\section{The chain complex associated with a balanced sutured manifold}\label{sec:chain-complex}\label{sec:analytic-aspects}
\subsection{Holomorphic disks and boundary degenerations; orientation issues}
Let us assume that $(\Sig,\alphas,\betas,\z)$ is an $\spinc$-admissible Heegaard diagram for
a balanced sutured manifold $(X,\tau)$ and $\spinc\in\SpinC(\ovl X^\tau)$.
We assume that $|\alphas|=|\betas|=\ell$ and that $\z=\{z_1,...,z_\el\}$.
We have already defined $\pi_2(\x,\y)$ for any two intersection points $\x,\y\in\Ta\cap\Tb$.
In discussing the analytic aspects of a Floer theory, we need to consider boundary degenerations
and sphere bubblings as well. We recall the following definitions from \cite{OS-linkinvariant}.
\begin{defn}
Suppose that $\x\in\Ta\cap\Tb$ is an arbitrary  intersection point. A continuous map
$$\psi:\R\times [0,\infty)\lra\mathrm{Sym}^\ell(\Sig)$$ satisfying the boundary conditions
\begin{displaymath}
\begin{split}
&\psi\Big(\R\times\big\{0\big\}\Big)\subset\Ta,\\
\lim_{|s|\ra\infty}&\psi(s,t)=\x,\ \ \&\ \ \
\lim_{t\ra\infty}\psi(s,t)=\x
\end{split}
\end{displaymath}
is called an $\alphas$-{\emph{boundary degeneration}}. The space
of homotopy classes of such maps is denoted by $\pi_2^{\alpha}(\x)$.
The space $\pi_2^{\beta}(\x)$ of $\betas$-{\emph{boundary degenerations}}
is defined similarly.
\end{defn}
If $\{J_t=\Sym^\ell(j_t)\}_{t\in[0,1]}$ is a generic path of almost complex structure,
associated with any $\phi\in\pi_2(\x,\y)$ we may consider the moduli space
$\Mod(\phi)$ of the representatives $$u:[0,1]\times \R \ra \Sym^\ell(\Sig)$$
of $\phi$ which satisfy the time dependent Cauchy-Riemann equation
$$\frac{\partial u}{\partial t}(t,s)+J_t\frac{\partial u}{\partial s}(t,s)
=0,\ \ \ \forall\ (t,s)\in[0,1]\times \R.$$
Similarly, for any $\psi\in\pi_2^\beta(\x)$, $\Nod(\psi)$ consists of the representatives
$u:[0,\infty)\times \R \ra \Sym^\ell(\Sig)$ of $\psi$ which are $J_0$-holomorphic. Also,
for  any $\psi\in\pi_2^\alpha(\x)$, $\Nod(\psi)$ consists of the representatives
$u:[0,\infty)\times \R \ra \Sym^\ell(\Sig)$ of $\psi$ which are $J_1$-holomorphic. \\

The determinant line bundle associated with the linearization of  the
(time dependent) Cauchy-Riemann operator over
the moduli of representatives of any of the above homotopy classes is trivial.
This makes it possible to equip
the corresponding moduli space with an orientation.
Following Ozsv\'ath and Szab\'o's approach in
\cite{OS-3m1}, we may choose a {\emph{coherent system of orientations}}
as follows. \\

As in the previous sections, let us assume that
$$\Sig-\alphas=\coprod_{i=1}^k A_i,\ \ \&\ \ \Sig-\betas=\coprod_{j=1}^l B_j,$$
where $A_i$ and $B_j$ correspond to the components $R_i^-\subset \Rr^-(\tau)$ and
$R_j^+\subset \Rr^+(\tau)$ respectively. Thus, the genus of $A_i$ is $g_i^-$ and
the genus of $B_j$ is $g_j^+$. Without loosing on generality, let us assume
that $l\geq k$.
Let $\x_0,...,\x_m$ be all the intersection points in
$\Ta\cap\Tb$ which correspond to the $\SpinC$ class $\spinc$.
Choose a disk class $\phi_i\in\pi_2(\x_0,\x_i)$ for each $i$, and
complete the classes of the boundary degenerations $A_1,...,A_k$ and
$B_1,...,B_{l-1}$ to a basis for the space of periodic domains
in $\pi_2(\x_0,\x_0)$. Note that
$$B_l=A_1+...+A_k-(B_1+...+B_{l-1})$$
is the only relation satisfied among $A_1,...,A_k$ and $B_1,...,B_l$.
Let us denote this basis by $\psi_1,....,\psi_n$. The choice of an orientation
(i.e. one of the two classes represented by a non-vanishing section) on the determinant
line bundle associated with the classes $\phi_1,...,\phi_m$ and $\psi_1,...,\psi_n$
induces an orientation on the moduli space corresponding to any class
$\phi\in \pi_2(\x_i,\x_j)$, $0\leq i,j\leq m$.
In fact,  $\phi+\phi_i-\phi_j$
is a periodic domain in $\pi_2(\x_0,\x_0)$, and is thus a linear combination of
the classes $\psi_1,...,\psi_n$. As a result, $\phi$ is a juxtaposition of
(possibly several copies of) classes in
$$\Big\{\phi_1,...,\phi_m,\psi_1,...,\psi_n\Big\},$$
and thus inherits a natural orientation in our system of
coherent orientations.\\

Let us study the boundary degenerations and their assigned orientation more
carefully.
Any periodic domain $\psi\in\pi_2(\x,\x)$ such that $\partial \Dcal(\psi)$ is a
union of $\alpha$-curves determines the class of an $\alpha$ boundary degeneration.
Thus, the domain of any $\alpha$ boundary degeneration $\psi\in\pi_2^\alpha(\x)$
is a linear combination of $A_1,...,A_k$:
$$\Dcal(\psi)=a_1A_1+...+a_kA_k.$$
 We may use Lipshitz' index formula to compute
the Maslov index of $\psi$:
$$\mu(\psi)=a_1\chi(A_1)+...+a_k\chi(A_k).$$
If furthermore, $\Dcal(\psi)$ is a positive domain, e.g. if $\psi$ is a holomorphic boundary
degeneration, then all $a_i$ are non-negative.
We may then define the map
\begin{displaymath}
 \begin{split}
 &\la:\coprod_{\x\in\Ta\cap\Tb}\Big(\pi_2^{\beta,+}(\x)
 \coprod \pi_2^{\alpha,+}(\x)\Big)\lra \Ring_\tau, \ \ \ \ \
 \la(\psi):=\prod_{i=1}^\el \la_i^{n_{z_i}(\psi)}.
\end{split}
\end{displaymath}
Here, we use $\pi_2^{\alpha,+}(\x)$ (respectively, $\pi_2^{\alpha,+}(\x)$) to denote the
subset of $\pi_2^{\alpha}(\x)$ (respectively, $\pi_2^{\beta}(\x)$) which consists of
the classes $\psi$ with $\Dcal(\psi)\geq 0$.\\

If  an $\alpha$ boundary degeneration $\psi$ as above is positive and $\la(\psi)\neq 0$,
we may conclude that for $i=1,...,k$, either $a_i=0$ or  the genus of $A_i$ is zero.
Without loosing on generality, assume that the genus of $A_1,...,A_{k_0}$ is zero, and
that the rest of $A_i$ have positive genus. Thus $\Dcal(\psi)\geq 0$ and $\la(\psi)\neq 0$
implies that
$$\Dcal(\psi)=a_1A_1+...+a_{k_0}A_{k_0}.$$
Consequently $\mu(\psi)=2(a_1+...+a_{k_0})$. Similarly, we may assume that the
genera of $B_1,...,B_{l_0}$ are zero, and
that the rest of $B_i$ have positive genus. This would imply that for
any $\psi\in\pi_2^\beta(\x)$
with $\Dcal(\psi)\geq 0$, we will either have $\la(\psi)=0$, or
$$\Dcal(\psi)=b_1B_1+...+b_{l_0}B_{l_0},\ \ \&\ \ \mu(\psi)=2(b_1+...+b_{l_0}).$$

In theorem 5.5 from \cite{OS-linkinvariant}, Ozsv\'ath and Szab\'o prove the following
statement. In fact, the statement of their result is less general, but the proof applies in
the following more general setup as well.

\begin{lem}\label{lem:disk-degeneration}
Let $\psi$ be the class of a boundary degeneration, and that
a coherent choice of orientation is fixed for the Heegaard diagram
$(\Sig,\alphas,\betas,\z)$. If $\Dcal(\psi)\geq 0$, $\la(\psi)\neq 0$, and
$\mu(\psi)\leq 2$ then $\Dcal(\psi)=A_i$ or $\Dcal(\psi)=B_j$ for
some $1\le i\le k_0$ or $1\leq j\leq l_0$ (or $\psi$ is the class of the constant map).
In the first case (i.e. $\Dcal(\psi)=A_i$) we have
\begin{displaymath}
\n(\psi)=
\begin{cases}
0 \ \ \  \ \ \ &\text{if}\ \ k=1\\
\pm1 \ \ \  \ \ \ &\text{if}\ \ k>1.
\end{cases}
\end{displaymath}
where $\n(\psi)=\#\widehat{\Ncal}(\psi)$.
Similarly, for $\Dcal(\psi)=B_j$ we have  $\n(\psi)=0$
if $l=1$ and $\n(\psi)=\pm 1$ if $l>1$.
Here
$\widehat{\Ncal}(\psi)$ is the quotient of ${\Ncal}(\psi)$ under the action of
the subgroup
\begin{displaymath}
\Group=\left\{\left(\begin{array}{cc}a&b\\ 0&\frac{1}{a}\end{array}\right)\ \Big|\
a\in\R^+,\ b\in\R
\right\}
<\mathrm{PSL}_2(\R).
\end{displaymath}
\end{lem}
\begin{proof}
See~\cite{OS-linkinvariant} theorem 5.5. Note that the moduli spaces are now equipped
with an orientation, and we may thus count the points of the moduli spaces with sign,
instead of working modulo $2$. The choice of the plus or minus sign comes from the choice
on the orientation associated with the homotopy classes $A_i$ and $B_j$ of $\alpha$ and $\beta$
boundary degenerations respectively.
\end{proof}
 The argument of Ozsv\'ath and Szab\'o in fact implies that there is a natural choice of
 orientation for $A_1,...,A_{k_0}$ and $B_1,...,B_{l_0}$ which makes the value of $\n(\psi)$
 equal to $+1$. This choice of orientation basically comes from the complex structure on the
 surface $\Sig$, since the moduli space of boundary degenerations associated with
 any of $A_1,...,A_{k_0}$ and $B_1,...,B_{l_0}$ is eventually, possibly after stretching the
necks, is identified with the group $\Group$ via Riemann mapping theorem.
This implies a form of compatibility among these preferred orientations on
$A_1,...,A_{k_0}$ and $B_1,...,B_{l_0}$. More precisely, if $k=k_0$ and $l=l_0$,
choosing the preferred orientation associated with $A_1,...,A_k,B_1,...,B_{l-1}$ induces
an orientation on the moduli space corresponding to $B_l$. This orientation
is the same as the preferred orientation associated with $B_l$.
If $l_0<l$ (or if $k_0<k$), we are free to choose the preferred orientation
associated with the classes $A_1,...,A_{k_0}$ and $B_1,...,B_{l_0}$ without
any compatibility requirement.
We may thus present the following definition. \\

\begin{defn}
A {\emph{coherent system of orientations}} associated with the Heegaard diagram
$(\Sig,\alphas,\betas,\z)$ for the balanced sutured manifold $(X,\tau)$
and the $\SpinC$ class $\spinc\in\SpinC(\ovl X^\tau)$
is an assignment $\Or$ of an orientation to
the determinant line bundle of the linearized Cauchy-Riemann operator
associated with all
classes in $\pi_2(\x,\y)$ (for all $\x,\y\in\Ta\cap\Tb$) with the following properties:\\

$\bullet$ $\Or(\phi_1\star \phi_2)$ is the orientation induced by $\Or(\phi_1)$ and
$\Or(\phi_2)$ via juxtaposition, for any $\x,\y,\z\in\Ta\cap\Tb$ representing $\spinc$, and any
$\phi_1\in\pi_2(\x,\y)$ and $\phi_2\in\pi_2(\y,\z)$.\\

$\bullet$ For any $\x\in\Ta\cap\Tb$ representing $\spinc$, any $R_i^-\subset\Rr^-(\tau)$ with $g_i^-=0$,
and any $R_j^+\subset\Rr^+(\tau)$ with $g_j^+=0$, let us denote by $\psi^-\in\pi_2^\alpha(\x)$ and
$\psi^+\in\pi_2^\beta(\x)$ the classes of boundary degenerations corresponding to $R_i^-$ and $R_j^+$
respectively. Then
$\Or(\psi^-)$ and $\Or(\psi^+)$
are the preferred orientation of Ozsv\'ath and Szab\'o on $\Nod(\psi^-)$ and $\Nod(\psi^+)$
respectively, which give $\n(\psi^-)=1$ and $\n(\psi^+)=1$ (if $k>1$ and $l>1$ respectively).
\end{defn}

The last assumption implies, in particular, that the orientation induced on
the periodic domain determined by $\Sig$ is the natural orientation on it, as defined in
section 3.6 of \cite{OS-3m1}.\\

Let us assume that $\psi\in \pi_2^\alpha(\x_0)$
satisfies $\Dcal(\psi)\in\{A_1,...,A_{k_0}\}$, say
$\Dcal(\psi)=A_1$. Furthermore, assume that a preferred
orientation on $\Nod(\psi)$ is fixed as before.
At the same time $\psi$ may be regarded as a
class in $\pi_2(\x_0,\x_0)$, and a
moduli space $\Mod(\psi)$ may also be associated
with $\psi$. This moduli space is smooth and
two dimensional as well, and gives an open $1$-manifold
$\ov\Mod(\psi)$ after we mod out by the
translation action of $\R$. The choice of
 orientation on $\Nod(\psi)$ induces an
orientation on $\Mod(\psi)$ as well. The reason
is that the determinant line bundle of the
(time dependent) Cauchy-Riemann operator on
both these moduli spaces is pulled back from
the same model, as discussed in subsection 3.6
in \cite{OS-3m1}.\\

According to the discussion
of section~5 from \cite{OS-linkinvariant},
$\ov\Nod(\psi)$ will then appear
as a boundary point of the smooth one dimensional manifold
$\ov\Mod(\psi)$. This induces a second orientation
on $\ov\Nod(\psi)$, as the boundary of
the oriented moduli space $\ov\Mod(\psi)$.
Whether this second orientation
agrees with the orientation of $\ov\Nod(\psi)$
as the quotient of $\Nod(\psi)$
under the action of $\Group$
or not depends on our convention for the
embedding of the translation group
$\R$ (which acts on on $\Mod(\psi)$) in
$\Group$, as will be discussed in more detail below.
The same discussion is valid for $\beta$ boundary degenerations.\\

By the Riemann mapping theorem, the half plane
$$\Hbb^+=\R\times[0,+\infty) \subset \C$$
is conformal to the unit disk, or to the strip
$[0,1]\times \R \subset \C$.
We may thus think of $\Hbb^+$ as the domain of the class
$\psi$, when considered as an element in $\pi_2(\x,\x)$. We may then fix
a real number $r\in\R$ and interpret $\psi$ as a class
with
$$\psi\Big([r,+\infty)\times\{0\}\Big)\subset \Ta,\ \ \&\ \
\psi\Big((-\infty,r]\times\{0\}\Big)\subset \Tb.$$
Furthermore, we have to assume that $\psi(r,0)$ and $\psi(\infty)$
are both the intersection point $\x\in\Ta\cap\Tb$.
The group $\Group$ consists of the maps $\rho_{a,b}$ for $a>0$ and
$b\in\R$ which are defined by
$$\rho_{a,b}(z):=az+b.$$
With this notation fixed, the re-parametrization group of the domain
of $\psi\in\pi_2(\x,\x)$ is then identified as
$$\R_r=\left\{ \rho_{a,b} \Big|\
a\in\R^+\ \&\  b=r(1-a)\right\}<\Group.$$
If we identify $a$ as the exponential of a real number, the subgroup
$\R_r$ is identified with $\R$. This induces an orientation on the
one dimensional subgroup $\R_r$ of $\Group$.\\

As $r$ approaches $-\infty$, the limit of the subgroups
$\R_r$ determines the embedding of the translation group
in the automorphism group of the domain of $\alpha$ boundary
degenerations. In the this case, assuming $r<<0$ we may write
$$a=1-\frac{c}{r},\ \ c\in(r,\infty),\ \ \Rightarrow\ b=c.$$
As $c$ grows large, $a$ grows large as well. Thus the above
parametrization of $\R_r$  by the interval $(r,+\infty)$ is
orientation preserving. With $r$ converging to $-\infty$,
the sequence $\{\rho_{1-(c/r),c}\}_r$ converges to
$$\rho_{1,c}:\Hbb^+\ra \Hbb^+,\ \ \rho_{1,c}(z)=z+c.$$
The limit of $\R_r$, as $r\ra -\infty$, is thus the translation
subgroup
$$\R_\alpha=\left\{ \rho_{1,c} \Big|\
c\in\R\right\}<\Group,$$
and the above parametrization of $\R_\alpha$ is orientation preserving.\\

On the other hand, when $r$ approaches $+\infty$,
the limit of the subgroups $\R_r$ determines the embedding of the translation group
in the automorphism group of the domain of $\beta$ boundary
degenerations. In the this later case,
assuming $r>>0$ we may write
$$a=1-\frac{c}{r},\ \ c\in(-\infty,r),\ \ \Rightarrow\ b=c.$$
This time, as $c$ grows large, $a$ becomes small. Thus the above
parametrization of $\R_r$  by the interval $(-\infty,r)$ is
orientation reversing. With $r$ growing large,
the sequence $\{\rho_{1-(c/r),c}\}_r$ converges to
$$\rho_{1,c}:\Hbb^+\ra \Hbb^+,\ \ \rho_{1,c}(z)=z+c.$$
The limit of $\R_r$, as $r\ra +\infty$, is thus the translation
subgroup
$$\R_\beta=\left\{ \rho_{1,c} \Big|\
c\in\R\right\}<\Group,$$
but this time,
the above parametrization of $\R_\beta$ is orientation reversing.\\

With the above conventions for the orientations of $\Group$,
$\R_r$, $\R_\alpha$ and $\R_\beta$ fixed, we have thus proved the following lemma.

\begin{lem}\label{lem:orientation}
Let $\phi\in\pi_2^\alpha(\x)$ and $\psi\in\pi_2^\beta(\x)$ be the
classes of $\alpha$ and $\beta$ boundary degenerations respectively.
Furthermore, assume that $\mu(\phi)=\mu(\psi)=2$, and that
$\Nod(\phi)$ and $\Nod(\psi)$ are smooth manifolds. Then the
orientation induced on $\ov\Nod(\phi)$ agrees with the boundary
orientation induced from $\ov\Mod(\phi)$, while the
orientation induced on $\ov\Nod(\psi)$ is the opposite of the boundary
orientation induced from $\ov\Mod(\psi)$.
\end{lem}

\subsection{Energy bounds and relative gradings}
Recall that for a Riemannian manifold $(M,g)$ and a domain $\Omega\subset \C$
the energy of a smooth map $u:\Omega\ra X$ is defined by
$$E(u)=\frac{1}{2}\int_{\Omega}\|du\|_g^2.$$
Suppose that $(\Sig,\alphas,\betas,\z)$ is a Heegaard diagram for
a balanced sutured manifold $(X,\tau)$. Let
$$\Sig-\alphas-\betas=\coprod_{i=1}^m D_i$$
be the connected components in the complement of the curves, and
$\eta$ denotes a K\"ahler form on $\Sig$. We denote the area of
$D_i$ with respect to $\eta$ by $\Ar(D_i)$, and for a domain
$\Dcal=\sum_{i=1}^ma_i D_i$ we define
$$\Ar(\Dcal)=\sum_{i=1}^m a_i\Ar(D_i).$$
The following lemma is basically lemma~3.5 from \cite{OS-3m1} and
theorem~6.3 from \cite{Juh}.

\begin{lem}\label{lem:energy-bound}
There is a constant $C$ which depends only on the Heegaard
diagram $(\Sig,\alphas,\betas,\z)$ and the K\"ahler form $\eta$
such that for any pseudo-holomorphic Whitney disk
$$u:(\D,\partial \D)\lra (\Sym^\ell(\Sig),\Ta\cup\Tb)$$
we have
$$E(u)\leq C.\Ar\left(\Dcal(u)\right).$$
\end{lem}
The existence of energy bounds is needed in Gromov compactness arguments.\\

Finally, note that
$$\pi_1\left(\Sym^\ell(\Sig)\right)=\Ht_1\left(\Sym^\ell(\Sig);\Z\right)$$
provided that $\ell>1$. Throughout the construction, we will assume that
the requirement $\ell>1$ is satisfied, by stabilizing Heegaard diagrams if
necessary.

\begin{defn}
For $\spinc\in \SpinC(\ovl X)$ let
$$\mathfrak{d}(\spinc)=\mathrm{gcd}_{h\in\Ht_2(\ovl X,\Z)}\big\langle c_1(\spinc), h\big\rangle.$$
If $H=(\Sig,\alphas,\betas,\z)$ is an $\spinc$-admissible Heegaard diagram for
$(X,\tau)$ for any $\x,\y\in\Ta\cap\Tb$ with $\relspinc(\x),\relspinc(\y)\in\spinc$,
and for $\phi\in\pi_2(\x,\y)$,
we define the relative grading of $\x$ and $\y$ by
$$\mathrm{gr}(\x,\y)=\mu(\phi)\ \ \ \left(\text{mod }\ \mathfrak{d}(\spinc)\right).$$
Thus, $\mathrm{gr}(\x,\y)\in\Z_{\mathfrak{d}(\spinc)}=\frac{\Z}{\mathfrak{d}(\spinc)\Z}$.
\end{defn}

The relative grading is independent of the choice of $\phi$. It induces a relative grading
on the module $\CFT(X,\tau,\spinc;H)$. For this purpose, we should determine the
grading associated with the generators $\la_1,...,\la_\el\in G(\Ring)$.
Each $\la_i$ corresponds to the class
\begin{displaymath}
[\gamma_i]\in \Ker\left((\imath_X)_*:\Ht_1(X;\Z)\ra \Ht_1(\ovl X;\Z)\right),
\end{displaymath}
where $\imath_X:X\ra \ovl X$ is the inclusion map. It is thus the boundary of an
integral $2$-chain $A_i=A_{[\gamma_i]}$ in $\ovl X$, which is well defined up to addition of $2$-cycles.
The evaluation  $d_i=-\langle c_1(\spinc), A_i\rangle$ is thus well-defined as an
element of  $\frac{\Z}{\mathfrak{d}(\spinc)\Z}$. We may then define the grading on $G(\Ring)$ by
setting
\begin{displaymath}
\mathrm{gr}\left(\prod_{i=1}^\el \la_i^{n_i}\right)
:=\sum_{i=1}^\el d_in_i\in \frac{\Z}{\mathfrak{d}(\spinc)\Z},\ \
\ \ \forall \ \ \prod_{i=1}^\el \la_i^{n_i}\in G(\Ring).
\end{displaymath}
If $\phi\in\pi_2(\x,\x)$ is a positive disk, it determines a periodic domain, and
$$\mu(\phi)=\Big\langle c_1(\spinc),H(\phi)\Big\rangle
=\mathrm{\gr}\left(\la_\z(\phi)\right) \ \ \ \left(\text{mod }\ \mathfrak{d}(\spinc)\right).$$

The $\Ring$-module $\CFT(X,\tau,\spinc;H)$ is thus equipped with a relative homological
grading by the elements in $\frac{\Z}{\mathfrak{d}(\spinc)\Z}$.
The differential of the corresponding Ozsv\'ath-Szab\'o complex
$\CFT(X,\tau,\spinc;H)$ which will be defined in the following subsection lowers this
relative grading by one. In particular, a relative grading is induced on the homology groups
corresponding to any test ring $\Ringg$ for $\Ring_\tau$.\\

\subsection{The construction of the chain complex}
Let $(X,\tau)$ be a balanced sutured manifold.
As discussed in section~\ref{sec:quasi-isomorphism}
we associate a coefficient ring $\Ring=\Ring_{\tau}$ with $\tau$,
which is a $\F$-algebra.
Let us denote by $\overline{X}$ the three-manifold
(with positive and negative boundary components)
obtained from $X$ by filling out the sutures in $\tau$.
Let $\spinc\in\RelSpinC(\overline{X})$ be a  $\SpinC$ structure on $\overline{X}$.
Consider an $\spinc$-admissible
Heegaard diagram $(\Sigma,\alphas,\betas,\z)$ for $(X,\tau)$.
Associated with this Heegaard diagram,
let
$$\CFT(\Sig,\alphas,\betas,\z;\spinc)
=\Big\langle \x\in \Ta\cap\Tb\ |\ \relspinc_\z(\x)\in\spinc\Big\rangle_{\Ring}$$
be a free $\Ring$-module generated by the
intersection points in $\Ta\cap\Tb$ which represent the $\SpinC$ class $\spinc$.
Note that the "exact sequence"
\begin{displaymath}
\begin{diagram}
0&\rTo &\big\langle\PD[\gamma_i]\big\rangle_{i=1}^\el
&\rTo&\RelSpinC(X,\tau)&\rTo{[.]=s_{\{1,...,\el\}}}&
\SpinC(\ovl X)&\rTo&0
\end{diagram}
\end{displaymath}
implies that the assignment of  relative $\SpinC$ structures gives
a filtration on the $\A$-module $\CFT(\Sig,\alphas,\betas,\z;\spinc)$, which is compatible
with the filtration $\chi:G(\Ring)\ra \Hbb$.
This  module may be decomposed using the filtration by
relative $\SpinC$ structures:
\begin{equation}\label{eq:spin-decomposition}
\begin{split}
\CFT(\Sig,\alphas,\betas,\z;\spinc)&=
\bigoplus_{{\relspinc\in\spinc\subset \RelSpinC(X,\tau)}}
\CFT(\Sig,\alphas,\betas,\z;\relspinc)\\
\CFT(\Sig,\alphas,\betas,\z;\relspinc)&=
\Big\langle \la\x\ |\
\x\in\Ta\cap\Tb,\ \la\in G(\Ring),\ \&\
\relspinc_\z(\x)+\chi(\la)=\relspinc\Big\rangle_{\F}.
\end{split}
\end{equation}
Here $G(\Ring)$ denotes the set of generator of the form
$\la=\prod_{i=1}^\el \la_i^{a_i}$ for
$\Ring$, as a  module over $\F$.\\

Furthermore, fix a coherent system $\Or$ of orientations on the determinant line bundles
of the linearization of Cauchy-Riemann operators associated with the classes
of the Whitney disks (corresponding to $\spinc$). We will drop $\Or$ from the notation,
unless an issue related to the orientation should be discussed.\\

Define an $\Ring$-module homomorphism by the following equation
\begin{displaymath}
\begin{split}
&\partial:\CFT(\Sig,\alphas,\betas,\z;\spinc)
\lra\CFT(\Sig,\alphas,\betas,\z;\spinc)\\
&\partial (\x):=\sum_{y\in\Ta\cap\Tb}
\sum_{\{\phi\in\pi_2^+(\x,\y)|\mu(\phi)=1\}}
\big(\m(\phi)\la(\phi)\big).\y.
\end{split}
\end{displaymath}
Here $\m(\phi)=\#\widehat{\mathcal{M}}(\phi)$ is the algebraic count
(i.e. with the signs determined by the orientation) of the
points in $\ov\Mod(\phi)$ for
any positive Whitney disk class $\phi\in\pi_2(\x,\y)$
such that $\la(\phi)\neq 0$. For other disk classes,
the contribution $\m(\phi)\la(\phi)$
is trivial by definition.\\

It is important to note that for any $\phi\in\pi_2^+(\x,\y)$
with $\mu(\phi)=1$ and $\la(\phi)\neq 0$,
the moduli space $\ov\Mod(\phi)$ is smooth, zero dimensional, oriented,
and compact. Smoothness and zero dimensionality of the moduli space
follows from the generic choice of the path of complex structures on
$\Sym^\ell(\Sig)$ (see the general discussion of \cite{OS-3m1}, section 3).
The compactness is however more critical. If $u_i$ is a sequence of
pseudo-holomorphic representatives of $\phi$, the amount of energy
$E(u_i)$ remains bounded by lemma~\ref{lem:energy-bound}. We may thus use the Gromov
compactness theorem to describe the possible limits of this sequence.
In fact, any possible Gromov limit of the sequence is the juxtaposition of some
pseudo-holomorphic representative $u$ of a class $\phi'\in\pi_2^+(\x,\y)$
with boundary degenerations and sphere bubblings. Let us assume that
$v_1,...,v_p$ are the classes of degenerations and bubbles. Then
the domains of $u$ and $v_i$, $i=1,...,p$ are positive and
$$\la(\phi)=\la(u)\la(v_1)...\la(v_p)\neq 0.$$
This implies that the domain of each $v_i$ is a linear combination of
the domains
$A_1,...,A_{k_0}$ or $B_1,...,B_{l_0}$ (with non-negative coefficients).
Here $A_1,...,A_{k_0}$ are the zero genus components in $\Sig-\alphas$ and
$B_1,...,B_{l_0}$ are the zero genus components in $\Sig-\betas$.
The Maslov index of each $v_i$
is thus a positive even number. Since the $\Mod(\phi')\neq \emptyset$,
$\mu(\phi')$ is non-negative and $p$ is thus forced to be $0$.
However, this means that the Gromov limit of $u_i$ is in $\ov\Mod(\phi)$,
i.e. $\ov\Mod(\phi)$ is compact, and thus finite.
In other words,  for any class
$\phi\in\pi_2^+(\x,\y)$ with $\mu(\phi)=1$, either $\phi$ does not contribute
to the coefficient of $\y$ in $\partial (\x)$ (e.g. $\la(\phi)=0$),
or $\m(\phi)$ is finite.\\

The Heegaard diagram $(\Sig,\alphas,\betas,\z;\spinc)$
is  $\spinc$-admissible, so by lemma
~\ref{lem:s-admissible}, for any intersection point $\y\in\Ta\cap\Tb$ there are only finitely many
$\phi\in\pi_2(\x,\y)$ such that $\mu(\phi)=1$, $\Dcal(\phi)\ge 0$, and $\la(\phi)\neq 0$.
Thus there are only finitely many classes $\phi\in\pi_2^+(\x,\y)$ with $\mu(\phi)=1$
and $\la(\phi)\neq 0$
which admit holomorphic representative.
This shows that the terms which contribute to the coefficient of $\y$ in $\partial (\x)$
are finite, and  that the map $\partial$ is thus well-defined.\\

The map $\partial$ is, by definition, a homomorphism of $\Ring$-modules.
It is obvious from the definition and the discussion of
subsection~\ref{subsec:algebra-of-SM} that $\partial$ preserves the decomposition of
equation~\ref{eq:spin-decomposition}, and we thus obtain a set of $\F$-module homomorphisms
\begin{displaymath}
\partial:\CFT(\Sig,\alphas,\betas,\z;\relspinc)\lra \CFT(\Sig,\alphas,\betas,\z;\relspinc),
\end{displaymath}
for any relative $\SpinC$ class $\relspinc\in\RelSpinC(X,\tau)$.
\begin{thm}\label{differential}
The filtered $\Ring$-module
$\mathrm{\CF}(\Sig,\alphas,\betas,\z;\spinc)$ is a filtered $\tab$ chain complex,
where $\Ring=\Ring_\tau$ is the coefficient ring associated with $\tau$ and
the filtration by the elements of the $\Z$-module $\Hbb=\Ht^2(X,\partial X;\Z)$
is given by the assignment of the relative $\SpinC$ classes in
$\RelSpinC(X,\tau)$ to the generators in $\Ta\cap\Tb$ using the map $\relspinc_\z$.
\end{thm}
Before we start proving the above theorem we re-phrase lemma~\ref{lem:disk-degeneration}
in the presence of a coherent system of orientation.

\begin{lem}\label{lem:disk-degeneration-2}
With the notation of lemma~\ref{lem:disk-degeneration} fixed,
let  $\Or$
be a coherent system of orientations associated with the Heegaard diagram.
Let $\psi$ be the class of a boundary degeneration.
If $\Dcal(\psi)\geq 0$, $\la(\psi)\neq 0$, and
$\mu(\psi)\leq 2$ then $\Dcal(\psi)=A_i$ or $\Dcal(\psi)=B_j$ for
some $1\le i\le k_0$ or $1\leq j\leq l_0$ (or $\psi$ is the class of the constant map).
In the first case (i.e. $\Dcal(\psi)=A_i$) we have
\begin{displaymath}
\n(\psi)=
\begin{cases}
0 \ \ \  \ \ \ &\text{if}\ \ k=1\\
1 \ \ \  \ \ \ &\text{if}\ \ k>1.
\end{cases}
\end{displaymath}
Similarly, for $\Dcal(\psi)=B_j$ we have  $\n(\psi)=0$
if $l=1$ and $\n(\psi)=1$ if $l>1$.
\end{lem}

Now we can prove theorem ~\ref{differential} using the above lemma.
The proof is similar to  the proof of lemma 4.3 in \cite{OS-linkinvariant}.

\begin{proof}({\bf of Theorem}~\ref{differential}.)
Clearly, for any Whitney disk $\phi\in\pi_2(\x,\y)$,
which has a holomorphic representative,
we have $\Dcal(\phi)\geq 0$ and
$\la(\phi)$ is thus a well-defined element of $\Ring$. Thus we only need to prove
$\partial\circ\partial=0$. \\

Let $\x$ and $\y$ be two intersection points in $\Ta\cap\Tb$.
Fix a class $\phi\in\pi_2(\x,\y)$ such that $\mu(\phi)=2$.
Consider the ends of the moduli space $\widehat{\sM}(\phi)$.
This space has three types of ends, which are in correspondence
with the broken flow-lines. More precisely, these
are (respectively) the ends corresponding to a holomorphic
Whitney disk $\phi_1$ connecting $\x$ to an intersection point
$\w$ juxtaposed with a holomorphic Whitney disk $\phi_2$
connecting $\w$ to $\y$ such that $\mu(\phi_1)=\mu(\phi_2)=1$,
the ends corresponding to a sphere bubbling off i.e. a holomorphic
Whitney disk $\phi'$ connecting $\x$ to $\y$ juxtaposed with a
 holomorphic sphere $S$ in $\mathrm{Sym}^\ell(\Sig)$, and the ends
 corresponding to a boundary bubbling i.e. a holomorphic Whitney
 disk $\phi'$ connecting $\x$ to $\y$ juxtaposed with a
 holomorphic boundary degeneration.\\

 If $\x\neq\y$ the space $\widehat{\sM}(\phi)$ does not have
 any boundary of the second and the third types, since
 any holomorphic boundary degeneration or holomorphic sphere
 with the property that its associated monomial in $\Ring$ is
 non-trivial will have Maslov index at least 2.
 Thus the remaining
 Whitney disk should have Maslov index less than or equal to
 zero. This implies that the moduli space associated
 with the Whitney disk is empty, or consists of a constant
 function (which can not happen by the assumption $\x\neq\y$).
 When $\x\neq\y$ the Gromov ends of this moduli space thus consist of
\begin{displaymath}
\coprod_{\w\in\Ta\cap\Tb}
\coprod_{\substack{\phi_1\in\pi_2(\x,\w)\\ \phi_2\in\pi_2(\w,\y)\\
\phi_1*\phi_2=\phi}}
\Big(\widehat{\sM}(\phi_1)\times\widehat{\sM}(\phi_2)\Big)
\end{displaymath}
For any fixed  $\la\in G(\Ring)$
the coefficient of $\la\y$ in $\partial^2\x$ (assuming $\x\neq \y$)
is equal to
\begin{displaymath}
\sum_{\substack{\phi\in\pi_2(\x,\y)\\\mu(\phi)=2\\
\la(\phi)=\la}}\sum_{\w\in\Ta\cap\Tb}
\sum_{\substack{\phi_1\in\pi_2(\x,\w)\\ \phi_2\in\pi_2(\w,\y)\\
\phi_1*\phi_2=\phi}}\Big(\m(\phi_1).\m(\phi_2)\Big).
\end{displaymath}
For each $\phi\in\pi_2(\x,\y)$ the amount of the two interior
sums in the above formula is the total number, counted with the sign
determined by the coherent system $\Or$ of orientations, of the ends of
the moduli space $\widehat{\sM}(\phi)$,
which is zero.
Consequently the total sum in the above formula is trivial,
and the coefficient of
$\la\y$ in $\partial\big(\partial(\x)\big)$ is thus equal to $0$. \\

Let us now assume that $\x=\y$. Let us denote the class of the
generator of holomorphic spheres by $S$. The domain $\Dcal(S)$ associated
with $S$ is the surface $\Sig$:
$$\Dcal(S)=A_1+...+A_k=B_1+...+B_l,\ \ \Rightarrow
\la(S)=\la(A_1)...\la(A_k)=\la(B_1)...\la(B_l).$$
Thus $\la(S)=0$ unless $k=k_0=l_0=l$. In this later case,
the Maslov index of $S$ is $2k$, which is greater than $2$
unless $k=1$. Combining with lemma~\ref{lem:disk-degeneration-2},
we may thus conclude that in all possible cases, the
total contribution to $\partial^2(\x)$ from sphere bubblings
is trivial.\\

We may thus assume, without loosing on generality,
that the ends of the moduli space
$\widehat{\sM}(\phi)$ do not contain any sphere bubblings. If
the ends of this moduli space contain a boundary disk degeneration,
then the degeneration would consist of
 the juxtaposition of a constant function and a holomorphic
 boundary degeneration with Maslov index $2$. If we denote
 the boundary degeneration by $\psi$,
 lemma~\ref{lem:disk-degeneration-2} implies that $\Dcal(\psi)=A_i$
 or $\Dcal(\psi)=B_j$.\\

In the above situation, if $\Dcal(\phi)=\Dcal(\psi)=A_i$ or
 $B_j$, the boundary disk degeneration among the ends
 of $\widehat{\sM}(\psi)$ are described by lemma~\ref{lem:disk-degeneration-2}.
 Suppose first that $k,l>1$.
 Let $\Dcal(\phi)=B_j$,  and let $\psi$ be the corresponding
 boundary disk degeneration with the same domain.
 Then the ends of $\sM(\phi)$ consist of
\begin{displaymath}
\widehat{\Ncal}(\psi)~\bigcup~\Big(\coprod_{\w\in\Ta\cap\Tb}
\coprod_{\substack{\phi_1\in\pi_2(\x,\w)\\ \phi_2\in\pi_2(\w,\x)\\ \phi_1\star
\phi_2=\phi}}\big(\widehat{\sM}(\phi_1)\times\widehat{\sM}(\phi_2)\big)\Big).
\end{displaymath}
According to lemma~\ref{lem:orientation}, the orientation of $\ov\Nod(\psi)$
agrees with the orientation induced from $\ov\Mod(\psi)$.
Thus the total number of the ends for this moduli space is equal to
\begin{displaymath}
\n(\psi)+\sum_{\w\in\Ta\cap\Tb}\sum_{\substack{\phi_1\in
\pi_2(\x,\w)\\ \phi_2\in\pi_2(\w,\x)\\ \phi_1\star\phi_2=\phi}}
\Big(\m(\phi_1).\m(\phi_2)\Big)=0.
\end{displaymath}
By lemma~\ref{lem:disk-degeneration-2}, we have $\n(\psi)=1$,
thus the total value of the second sum is equal to $-1$.\\
Note that $\la(\psi)=\la(R_j^+)=\la_j^+$, since the domain
associated with $\psi$ is  $B_j$.
Thus, such degenerations contribute to the
coefficient of $\la(B_j)\x$, i.e. the contribution of $\psi$ to
$\partial^2(\x)$ is $-\la_j^+.\x$.
Similarly, for a $\alpha$ boundary degeneration $\psi$ with $\Dcal(\psi)=A_i$,
we obtain the equality
\begin{displaymath}
\sum_{\w\in\Ta\cap\Tb}\sum_{\substack{\phi_1\in
\pi_2(\x,\w)\\ \phi_2\in\pi_2(\w,\x)\\ \phi_1\star\phi_2=\phi}}
\Big(\m(\phi_1).\m(\phi_2)\Big)=1.
\end{displaymath}
Thus the contribution of $\psi$ to $\partial^2(\x)$ is
$\la_i^-.\x$.\\

The coefficient of $\x$ in $\partial\big(\partial(\x)\big)$
is equal to
\begin{displaymath}
\begin{split}
&\sum_{\substack{\phi\in\pi_2(\x,\x)\\ \Dcal(\phi)=A_i, 1\le i\le
k}}\la(A_i)\Big(\sum_{\w\in\Ta\cap\Tb}\sum_{\substack{\phi_1\in\pi_2(\x,\w)\\
\phi_2\in\pi_2(\w,\y)\\ \phi_1\star\phi_2=\phi}}\big(\m(\phi_1).\m(\phi_2)\big)\Big)+\\
&\sum_{\substack{\phi\in\pi_2(\x,\x)\\ \Dcal(\phi)=B_j, 1\le j\le
l}}\la(B_j)\Big(\sum_{\w\in\Ta\cap\Tb}\sum_{\substack{\phi_1\in\pi_2(\x,\w)\\
\phi_2\in\pi_2(\w,\y)\\ \phi_1\star\phi_2=\phi}}\big(\m(\phi_1).\m(\phi_2)\big)\Big)+\\
&\sum_{\substack{\phi\in\pi_2(\x,\x)\\ \Dcal(\phi)\neq A_i ~\text{or} ~B_j}}
\la(\phi)\Big(\sum_{\w\in\Ta\cap\Tb}\sum_{\substack{\phi_1\in\pi_2(\x,\w)\\
\phi_2\in\pi_2(\w,\y)\\ \phi_1\star\phi_2=\phi}}\big(\m(\phi_1).\m(\phi_2)\big)\Big).
\end{split}
\end{displaymath}
Our argument shows  that in the above the sum, the sums in
the first and the second line combine to give the following
expression:
\begin{displaymath}
\sum_{1\leq i\leq k}\la(A_i)-\sum_{1\leq i\leq l}\la(B_i)=\sum_{1\leq i\leq k}\la(R_i^+)-
\sum_{1\leq i\leq l}\la(R_i^-)=0.
\end{displaymath}
Thus the sum of the contributions from the first two lines in the above expression is zero.
The last line is a sum of zero terms, by an argument similar to the case $\x\neq \y$,
so it is trivial. Consequently, the coefficient of $\x$
in $\partial\big(\partial(\x)\big)$ is zero.\\

When $k=1$ or $l=1$ one should be cautious. If $k=l=1$, the contributions from
both $\alpha$ and $\beta$ boundary degenerations is zero by lemma~\ref{lem:disk-degeneration-2}. However, if
$k=1$ and $l>1$, since the Heegaard diagram is balanced, we may conclude that $l>l_0$.
However,
$$\la(A_1)=\la(\Sig)=\prod_{i=1}^l \la(B_i)=0\ \ \Rightarrow\ \ \sum_{i=1}^l\la_i^+=\la_1^-=0.$$
The rest of the argument in this case if completely identical with the case $k,l>1$.
This completes the proof of the theorem.
\end{proof}

In the following  section we  will prove the following theorem.
\begin{thm}\label{thm:invariance}
The filtered $\tab$ chain homotopy type of the filtered
$\tab$ chain complex $\CFT(\Sig,\alphas,\betas,\z;\spinc)$
is an invariant of the balanced sutured manifold and the $\SpinC$ class $\spinc\in\SpinC(\overline{X})$.
In particular, for any filtered test ring $\Ringg$ for $(\Ring,\Hbb)$ and
for any $\relspinc\in\RelSpinC(X,\tau)$, the chain homotopy type of
$$\CFT(\Sig,\alphas,\betas,\z;\relspinc;\Ringg)\subset
\CFT(\Sig,\alphas,\betas,\z;\spinc)\otimes_\Ring\Ringg$$
is also an invariant of $(X,\tau,\relspinc)$.
\end{thm}
\begin{defn}
We may thus denote the filtered $\tab$ chain homotopy type of the filtered
$\tab$ chain complex $\CFT(\Sig,\alphas,\betas,\z;\spinc)$ and its invariant decomposition
into chain complexes  $\CFT(\Sig,\alphas,\betas,\z;\relspinc)$ by
$$\CFT(X,\tau;\spinc)=\bigoplus_{\relspinc\in\spinc\subset\RelSpinC(X,\tau)}
\CFT(X,\tau;\relspinc).$$
\end{defn}
\subsection{Additional algebraic structure}
From the definitions, it is clear that the multiplication by a generator $\la\in G(\Ring)$ gives a map
$$m_\la:\CFT(X,\tau;\relspinc)\lra
\CFT(X,\tau;\relspinc+\chi(\la)).$$
This map shifts the homological grading by $\mathrm{gr}(\la)$, and generalizes the
$U$-action in the original construction of Ozsv\'ath and Szab\'o.\\

Let us assume, once again,  that $(\Sig,\alphas,\betas,\z)$ is an $\spinc$-admissible
Heegaard diagram for the sutured manifold $(X,\tau)$, where $\spinc\in\SpinC(\ovl X)$ is a
fixed $\SpinC$ structure.
Let $\Omega(\Ta,\Tb)$ denote the space of paths connecting $\Ta$ to $\Tb$ in
$\Sym^\ell(\Sig)$.
Any intersection point $\x\in\Ta\cap\Tb$, viewed as a constant path, is a point
in $\Omega(\Ta,\Tb)$, and for any $\x,\y\in\Ta\cap\Tb$ and any Whitney disk
$u$ representing a class $\phi\in\pi_2(\x,\y)$, $u$ may be viewed as a path
connecting $\x$ to $\y$ in $\Omega(\Ta,\Tb)$. The homotopy class of this path
depends only on $\phi$.
As in section 4 of \cite{OS-3m1} for any one-cocycle
\begin{displaymath}
\zeta\in Z^1\left(\Omega(\Ta,\Tb),\Z\right)
\end{displaymath}
the evaluation $\zeta(\phi)$ is  well-defined. Correspondingly, we may  define the map
\begin{displaymath}
\begin{split}
A_\zeta&:\CFT(X,\tau;\spinc)\lra \CFT(X,\tau;\spinc)\\
A_\zeta(\x)&:=\sum_{\substack{\y\in \Ta\cap\Tb\\ \relspinc(\y)\in\spinc}}
\sum_{\substack{\phi\in\pi_2^+(\x,\y)\\ \mu(\phi)=1}}
\left(\zeta(\phi).\la_\z(\phi).\m(\phi)\right).\y,\ \ \forall\ \x\in\Ta\cap\Tb,\
\ s.t.\ \  \relspinc(\x)\in\spinc.
\end{split}
\end{displaymath}
The map $A_\zeta$ is then extended as a homomorphism of $\Ring$-modules
to $\CFT(X,\tau;\spinc)$. It respects the decomposition according to relative
$\SpinC$ structures in $\spinc\subset \RelSpinC(X,\tau)$.
As in lemmas 4.18 and 4.19 from \cite{OS-3m1} one may prove that the map $A_\zeta$ satisfies
\begin{equation}\label{eq:properties-of-action}
\begin{split}
&(i)\ \ \ \ \ \ \partial\circ A_\zeta+ A_\zeta\circ \partial =0,\ \ \&\\
&(ii)\ \ \ \ A_\zeta=\partial \circ H_\zeta- H_\zeta\circ \partial,\ \ \
\text{if }\ \zeta\ \text{is a coboundary},
\end{split}
\end{equation}
for some $\Ring$-module homomorphism $H_\zeta$ which respects the
filtration in $\Hbb$.
As a result, the following proposition may be proved in this generalized setup.
\begin{prop}\label{prop:homological-action}
There is a natural action of $\Ht^1(\Omega^1(\Ta,\Tb);\Z)$ on $\CFT(X,\tau;\spinc)$
lowering degree by one, which is well-defined up to filtered chain homotopy equivalence.
Furthermore, this induces an action of the exterior algebra
\begin{displaymath}
\wedge^*\left(\Ht_1(\ovl X;\Z)/\mathrm{Tors}\right)
\subset \wedge^*\left(\Ht^1\left(\Omega(\Ta,\Tb),\Z\right)\right)
\end{displaymath}
on the module $\CFT(X,\tau;\spinc)$, which is well-defined up to
chain homotopy equivalence.
\end{prop}
\begin{proof}
One should simply repeat the proof of proposition 4.17 from \cite{OS-3m1}.
From the properties stated in equation~\ref{eq:properties-of-action}
and the isomorphisms
\begin{displaymath}
\Ht^1\left(\Omega(\Ta,\Tb);\Z\right)
\cong \mathrm{Hom}\left(\pi_1\left(\Omega(\Ta,\Tb)\right),\Z\right)
\cong \pi_2\left(\Sym^\ell(\Sig)\right)\oplus \mathrm{Hom}\left(\Ht^1(\ovl X,\Z),\Z\right),
\end{displaymath}
the proof of the above proposition is reduced to showing
$A_\zeta\circ A_\zeta=0$.
For this purpose, let $f:\Omega(\Ta,\Tb)\ra S^1$ denote a representative
of $\zeta$. For a generic point $p\in S^1$ we set $V_p=f^{-1}(p)$ and observe that
for any generator $\x\in\Ta\cap\Tb$ representing the $\SpinC$ class $\spinc$
\begin{displaymath}
\begin{split}
&A_\zeta(\x)=\sum_{\substack{\y\in \Ta\cap\Tb\\ \relspinc(\y)\in\spinc}}
\sum_{\substack{\phi\in\pi_2^+(\x,\y)\\ \mu(\phi)=1}}
\left(a(\zeta,\phi).\la_\z(\phi)\right).\y, \\
\text{where }\ \ &
a(\zeta,\phi)=\#\Big\{u\in\Mod(\phi)\ \big|\ u\left([0,1]\times \{0\}\right)\in V_p\Big\}.
\end{split}
\end{displaymath}
Let us now consider a positive homotopy class $\phi\in\pi_2^+(\x,\y)$ with $\mu(\phi)=2$.
Associated with $\phi$, and for generic points $p,q\in S^1$, we consider the one-dimensional moduli space
\begin{displaymath}
\begin{split}
\Xi_{p,q}(\phi):=\left\{(s,u)\in[0,\infty)\times \Mod(\phi)\ \big|\
\begin{array}{c}
u\left([0,1]\times \{s\}\right)\subset V_p\\
u\left([0,1]\times \{-s\}\right)\subset V_q\\
\end{array}\right\}.
\end{split}
\end{displaymath}
This one-manifold does not have any boundary at $s=0$.
Furthermore, if we set $I_0=[0,1]\times\{0\}$,
the boundary at infinity (i.e. the structure of
the moduli space as $s\ra \infty$) is modeled on
\begin{displaymath}
\coprod_{\substack{\phi_1\star\phi_2=\phi\\ \mu(\phi_1)=\mu(\phi_2)=1}}
\left(\Big\{u_1\in\Mod(\phi_1)\ \big|\ u_1(I_0)\subset V_p\Big\}
\times\Big\{u_2\in\Mod(\phi_2)\ \big|\ u_2(I_0)\subset V_q\Big\}\right).
\end{displaymath}
Other possible boundary points correspond to boundary disk degenerations and
sphere bubblings. If we furthermore assume that $\la_\z(\phi)\neq 0$,
any boundary disk degeneration or sphere bubbling will reduce the Maslov index
at least by $2$. Thus the moduli space corresponding to such degenerations
would be empty, if we choose a generic path of almost complex structures.\\

The number of points in the boundary of $\Xi_{p,q}(\phi)$, counted
with sign, would vanish. On the other hand, this total count corresponds to
the contribution of the pairs $(\phi_1,\phi_2)$ with $\phi=\phi_1\star\phi_2$
and $\mu(\phi_1)=\mu(\phi_2)=1$ to the coefficient of
$\la_\z(\phi).\y$ in $A_\zeta^2(\x)$. Thus $A_\zeta^2=0$
for all $\zeta\in \Ht^1\left(\Omega(\Ta,\Tb),\Z\right)$.
Thus  the action descends to an action of the exterior
algebra
$$\wedge^*\left(\Ht^1\left(\Omega(\Ta,\Tb),\Z\right)\right).$$
This completes the proof of the proposition.
\end{proof}

\newpage

\section{Invariance of the filtered chain homotopy type}
\subsection{Pseudo-holomorphic $m$-gons}\label{subsec:m-gons}
Let us assume that the Heegaard diagram
$H=(\Sig,\alphas^1,\alphas^2,...,\alphas^m,\z)$
is given, so that $\Sig$ is a
closed Riemann surface of genus $g$, each $\alphas^i$
is a set of $\ell$ disjoint simple closed curves on $\Sig$, and
$\z=\{z_1,...,z_\el\}$ is a set of marked points in
$$\Sig-\alphas^1-\alphas^2-...-\alphas^m.$$
We assume that for all $i=1,..,m$, every connected component
in $\Sig-\alphas^i$ contains at least one element of $\z$.
The Heegaard diagram $(\Sig,\alphas^i,\alphas^j,\z)$
determines a balanced sutured manifold $(X_{ij},\tau_{ij})$.
Let $\overline{X_{ij}}$ denote the
three-manifold obtained from $X_{ij}$ by filling out
the sutures in $\tau_{ij}$, and fix
the $\SpinC$ classes $\spinc_{ij}\in\SpinC(\overline{X_{ij}})$.
Assume that for any pair of indices $i<j$,
$(\Sig,\alphas^i,\alphas^j,\z)$ is an  $\spinc_{ij}$-admissible
Heegaard diagram for the balanced sutured manifold $(X_{ij},\tau_{ij})$.
Furthermore, let $\Or_{ij}$ be a coherent system of orientations
on $(\Sig,\alphas^i,\alphas^j,\z)$ associated with $\spinc_{ij}$.
Finally, suppose that
$$\CFT(\Sig,\alphas^i,\alphas^j,\z;\spinc_{ij})=
\bigoplus_{\relspinc_{ij}\in\spinc_{ij}}
\CFT(\Sig,\alphas^i,\alphas^j,\z;\relspinc_{ij})$$
is the corresponding chain complex, and its
decomposition into relative $\SpinC$ classes.
Let us assume that
$$\Sig-\alphas^i=\coprod_{j=1}^{k_i}A^i_j, \ \ \ i=1,2,...,m,$$
 are the connected components in the complements of the
 curves in $\alphas^i$. We will denote the genus of $A^i_j$ by
 $g^i_j\in\Z^{\geq 0}$. We will also denote
 $$\sum_{p=1}^{k_i} \la(A_p^i)\in \Big\langle \la_1,...,\la_\el\Big\rangle_\F$$
 by $\la(\alphas^i)$.
For any subset $I$ of the set of indices $\{1,...,m\}$
introduce the $\F$-algebra
\begin{displaymath}
\Ring_I=\frac{\Big\langle \la_1,...,\la_\el\Big\rangle_\F}
{\Big\langle \la(\alphas^i)=\la(\alphas^j)\
\big|\ \ \forall \ i,j\in I\Big\rangle\oplus
\Big\langle \la(A^i_j)\ \big|\ i\in I,\ g_j^i>0\Big\rangle}.
\end{displaymath}
If for two subsets $I,J\subset \{1,...,m\}$ we have
$I\subset J$, then $\Ring_J$ would be a quotient
of $\Ring_I$, and we have a natural homomorphism
$$\rho_{IJ}:\Ring_I\lra \Ring_J.$$
This homomorphism may be used to give $\Ring_J$ the
structure of an $\Ring_I$-module. As a result, from
any $\Ring_I$ chain complex
$(C,d)$, we obtain a natural $\Ring_J$ chain complex
$C\otimes_{\Ring_I}\Ring_J$.
In particular, for any index set $I$ which contains
$i,j$, we may consider the $\Ring_I$ chain complex
$$C_{ij}(I)=\CFT(\Sig,\alphas^i,\alphas^j,\z;
\spinc_{ij})\otimes_{\Ring_{ij}}\Ring_I.$$
We will denote $C_{ij}\big(\{1,...,m\}\big)$ by
$C_{ij}$ for simplicity.\\

Associated with each set of curves $\alphas^i$ is
a torus $\Tai\subset \mathrm{Sym}^{\ell}(\Sig)$.
A Whitney $m$-gon is a continuous map
$u$ from the standard $m$-gon $\DD_m$ into $\mathrm{Sym}^{\ell}(\Sig)$
which maps the $i$-th edge of the $m$-gon to $\Tai$. If we fix
$$\x_i\in\Tai\cap\mathbb{T}_{\alpha^{i+1}}, \ \ i=1,...,m-1,
\ \ \&\ \ \x_m\in\mathbb{T}_{\alpha^{m}}\cap\mathbb{T}_{\alpha^{1}},$$
we may let $\pi_2(\x_1,...,\x_{m})$ denote the set of Whitney
$m$-gons which map the vertex $v_i$ between the $i$-th edge and the $(i+1)$-th
edge to $\x_i$ (for $i=1,...,m-1$), and the vertex $v_m$ between
the $m$-th edge and the first edge to $\x_m$.\\

Let us fix a generic continuous family $\{J_p\}_{p\in\DD_m}$
of almost complex structures on $\Sym^\ell(\Sig)$
determined by a family $\{j_p\}_{p\in \DD_m}$ of complex
structures on $\Sig$. Furthermore, we will assume that under
a fixed identification of a neighborhood of the
$i$-th vertex $v_i$ of $\DD_m$ with $[0,1]\times (0,\infty)$
the family is translation invariant, i.e.
$$j_{(s,t)}=j_{(s,t+R)},\ \ \ \forall\ (s,t)\in[0,1]\times (0,\infty),\ \ R\in\R^+.$$
We will drop this generic family $\{J_p\}_{p\in\DD_m}$ from our notation.
For $\phi\in\pi_2(\x_1,...,\x_m)$ we let $\Mod(\phi)$ denote the
set of pseudo-holomorphic representatives of $\phi$.\\

Fix a subset $I=\{i_1<i_2<...<i_p\}\subset\{1,...,m\}$.
This subset determines a sub-diagram
$$H_I=(\Sig,\alphas^{i_1},...,\alphas^{i_p},\z)$$
of $H$. Correspondingly, we may consider the $p$-gons
associated with $H_I$. We will say that two $p$-gons
$\phi\in\pi_2(\x_{i_1},...,\x_{i_p})$ and $\phi'\in
\pi_2(\y_{i_1},...,\y_{i_p})$ are equivalent if
and only if there exists Whitney disk classes $\psi_{i_j}\in
\pi_2(\x_{i_j},\y_{i_j})$ for $j=1,...,p$ such that
$\phi$ is obtained from $\phi'$ by juxtaposition of
the disk $\psi_{i_j}$ at the vertices $\y_{i_j}$ for
$j=1,...,p$. The set of equivalence classes of such
$p$-gons will be denoted by $\RelSpinC(H,I)$. It is
important to note that $\RelSpinC(H,\{i,j\})$ determines
a subset of the set of  $\SpinC$ structures on  the
three-manifold $\overline{X_{ij}}$, which are realized
by the Heegaard diagram.\\

\begin{defn}
Suppose that we have a pair of index
sets $I,J\subset \{1,...,m\}$ such that
$I=\{i_1<i_2<...<i_p\}$ and
$J=\{i_r=j_1<j_2<...<j_q=i_{r+1}\}$.
We will call the pair $I,J$ {\emph{attachable}},
and define
$$I\star J:=\Big\{i_1<...<i_r=j_1<j_2<...<j_q=i_{r+1}<i_{r+2}<...<i_p\Big\}.$$
We will denote $r$ by $r(I,J)$ for future reference.
Suppose that $I$ and $J$ are attachable index sets as above, and that we are given a $p$-gon
$\phi$ and a $q$-gon $\phi'$
$$\phi\in\pi_2(\x_{i_1},...,\x_{i_p}),\ \
\&\ \ \psi\in\pi_2(\y_{j_1},...,\y_{j_p}),$$
where $\x_{i_s}\in\mathbb{T}_{\alpha^{i_s}}
\cap\mathbb{T}_{\alpha^{i_{s+1}}}$ and $\y_{j_s}\in
\mathbb{T}_{\alpha^{j_s}}\cap\mathbb{T}_{\alpha^{j_{s+1}}}$.
Furthermore, assume that $\x_{i_r}=\y_{j_q}$.
Then we may {\emph{juxtapose}} $\phi$ and $\psi$ to obtain
the class of some $(p+q-2)$-gon, which will be
denoted by $\phi\star\psi$.
\end{defn}

Let us now restrict ourselves to the polygons whose
vertices correspond to the fixed set
$\Ss=\{\spinc_{ij}\}_{i<j}$ of $\SpinC$ structures.
Let us denote the subset of $\RelSpinC(H,I)$ which consists of
polygons such that the $\SpinC$ structures
associated with the vertices are in $\Ss$
by $\RelSpinC(H,I;\Ss)$.
Then, the above construction gives a well-defined
operation between the equivalence classes
of polygons in $\RelSpinC(H,I;\Ss)$. More precisely,
if $I$ and
$J$ are a pair of attachable index sets,
we will have a map
$$\big(.\star.\big):\RelSpinC(H,I;\Ss)\times
\RelSpinC(H,J;\Ss)\lra \RelSpinC(H,I\star J;\Ss),$$
defined by the above juxtaposition process.\\

\begin{defn}\label{def:SpinC}
With the above notation fixed, a {\emph{coherent system of
$\SpinC$ structures on polygons}}
for the Heegaard diagram $H=(\Sig,\alphas^1,...,\alphas^m,\z)$,
and compatible with
$\Ss$ is a choice of classes
$$\Tt=\Big\{\phi_I\in\RelSpinC(H,I;\Ss)\ \big|\ I\subset \{1,...,m\},\ |I|\geq 3\Big\}$$
 such that the following is satisfied.
If $I$ and
$J$ are attachable index sets,
then we have $$\phi_{I\star J}=\phi_I\star\phi_J.$$
\end{defn}

\begin{lem}\label{lem:polygon-compatibility}
Let us assume that a coherent system $\{\phi_I\}_I$ of
$\SpinC$ structures is fixed
for the Heegaard diagram $H$. If $K=I\star J$ and
a polygon $\psi_K$ in
the same class as $$\phi_K=\phi_I\star\phi_J\in\SpinC(H,K;\Ss)$$
is decomposed as $\psi_K=\psi_I\star\psi_J$,
then the the class of $\psi_I$ in $\RelSpinC(H,I;\Ss)$ is
equal to the class of $\phi_I$ and
the class of $\psi_J$ in $\RelSpinC(H,J;\Ss)$ is equal
to the class of $\phi_J$.
\end{lem}
\begin{proof}
Fix the above notation and let $K=I\star J$. After  addition of disk
classes we may assume that the corners of $\psi_K$ are the same as
the corners of $\phi_K$ (i.e. both are chosen from $\{\x_{ij}\}_{i<j}$).
This means that
\begin{displaymath}
\begin{split}
&\phi_I\in\pi_2\Big(\x_{i_1i_2},...,\x_{i_{p-1}i_p},\x_{i_1i_p}\Big),\\
&\psi_I\in\pi_2\Big(\x_{i_1i_2},...,\x_{i_{r-1}i_r},
\y,\x_{i_{r+1}i_{r+2}},...,\x_{i_{p-1}i_p},\x_{i_1i_p}\Big),\\
&\phi_J\in\pi_2\Big(\x_{j_1j_2},...,\x_{j_{q-1}j_q},
\x_{j_1j_q}\Big),\ \ \&\ \
\psi_J\in\pi_2\Big(\x_{j_1j_2},...,\x_{j_{q-1}j_q},\y\Big)
\end{split}
\end{displaymath}
and we have $\phi_I\star\phi_J=\psi_I\star\psi_J$.
We thus have the following relation among the
associated domains:
$$\Dcal(\phi_I)-\Dcal(\psi_I)=\Dcal(\psi_J)-\Dcal(\phi_J)=\Dcal.$$
The coefficients of the domains in the expression appearing on the left hand side of the above
equality on both sides of any curve in $\alphas^i$, with $i\notin I$, are equal.
Similarly, the coefficients of the domains in the expression appearing as the middle term
 in the above equality on the two sides of any curve
in $\alphas^j$, with $j\notin J$, are equal. This
implies that $\partial(\Dcal)$ is included in
$$\coprod_{i\in I\cap J}\alphas^i=\alphas^{i_r}
\coprod\alphas^{i_{r+1}}.$$
Thus, $\Dcal$ is the domain associated with a disk
in $\pi_2(\x_{i_ri_{r+1}},\y)$, and
the $\SpinC$ class of $\psi_I$ is the same as that
of $\phi_I$. Similarly, the $\SpinC$ class of
$\psi_J$ is the same as that of $\phi_J$.
This completes the proof of the lemma.
\end{proof}

The above lemma implies that a coherent system of
$\SpinC$ structures on polygons for $H$ is completely determined
by the choice of triangle classes
$$\Big\{\phi_{ijk}\in \SpinC(H,\{i,j,k\};\Ss)\ \big|\ 1\leq i<j<k\leq m\Big\},$$
which satisfy the following compatibility relation
\begin{equation}\label{eq:SpinC-compatibility}
\phi_{ikl}\star\phi_{ijk}=\phi_{ijl}\star\phi_{jkl}\ \ \forall\ 1\leq i<j<k<l\leq m.
\end{equation}
Furthermore, the above lemma implies that for $\phi_{ijk},\phi_{ikl}$ and $\phi_{ijl}$
as above, there exists at most one class $\phi_{jkl}$ such that equation~\ref{eq:SpinC-compatibility}
is satisfied.
This observation implies that a coherent system of $\SpinC$ classes
of polygons for the Heegaard diagram $H$ is determined by the
family of triangle classes
$$\Big\{\phi_{1ij}\ \big|\ 1<i<j\leq m\Big\}.$$
However, this family should have the property that for any triple $1<i<j<k\leq m$
of indices, there is a triangle class $\psi$ such that
\begin{equation}\label{eq:SpinC-compatibility-2}
\phi_{1jk}\star\phi_{1ij}=\phi_{1ik}\star\psi.
\end{equation}
If this is the case, we will write
$$\Tt=\{\phi_I\}_I=\Big\langle \phi_{1ij}\ |\ 1<i<j\leq m \Big\rangle. $$

Let us fix a system $\Tt$ of compatible
$\SpinC$ structures for the Heegaard diagram
$H$ as above, which is generated by the triangle classes $\phi_{1ij}$.
The set of periodic domains for polygons in
$\Tt$ is generated by periodic domains for each pair
$(\alphas^i,\alphas^j)$. To be more precise, let us denote
by $\Pp_{ij}$ the set of periodic domains for the
Heegaard diagram $(\Sig,\alphas^i,\alphas^j)$.
Then any periodic domain which appears as the difference
of two $q$-gons with the same set of vertices
$$\y_{j}\in\mathbb{T}_{\alpha^{i_j}}
\cap\mathbb{T}_{\alpha^{i_{j+1}}},\ \ \ j=1,...,q,\
i_{q+1}:=i_1,\ \&\ i_1<i_2<...<i_q, $$
and representing the same $\SpinC$ class
may be written as a sum of periodic domains in
$\Pp_{i_1i_2},\Pp_{i_2i_3},...,\Pp_{i_{q-1}i_q},$ and $\Pp_{i_1i_q}$.
\begin{defn}
Let the Heegaard diagram $H=(\Sig,\alphas^1,\alphas^2,...,\alphas^m,\z)$
and $\Ss$, $\Tt$, and $\Pp_{ij}$ be as above.
The  Heegaard diagram $H$ is called {\emph{$\Ss$-admissible}}
if for any index set $I=\{i_1<...<i_q\}$, and any periodic domain
 $$\Pcal=\Pcal_{i_1i_2}+\Pcal_{i_2i_3}+...+\Pcal_{i_{q-1}i_q}+\Pcal_{i_1i_q}$$
 with $\Pcal_{ij}\in \Pp_{ij}$, the following is true. If
\begin{displaymath}
\begin{split}
&\sum_{j=1}^q{\Big\langle c_1(\spinc_{i_ji_{j+1}}),
H(\Pcal_{i_ji_{j+1}})\Big\rangle}=0
\end{split}
\end{displaymath}
then either the coefficients of the domain
$\Pcal$ at some point $w$ is negative, or
$\la(\Pcal)=0$ in $\Ring_I$.
\end{defn}
The existence of $\Ss$-admissible Heegaard diagrams,
and the possibility of modifying $H$ to an admissible Heegaard
 diagram using finger moves, follows with an argument
completely similar to the arguments of section~\ref{sec:admissibility}.
Furthermore, the $\Ss$-admissibility of the Heegaard diagram
$H$ implies that for any index set $$I=\{i_1<...<i_q\}\subset \{1,...,m\},$$
any integer $N$,
and any set of corners
$$\y_{j}\in\mathbb{T}_{\alpha^{i_j}}
\cap\mathbb{T}_{\alpha^{i_{j+1}}},\ \ \ j=1,...,q,\
i_{q+1}:=i_1,\ \&\ i_1<i_2<...<i_q, $$
such that $\relspinc_\z(\y_j)\in\spinc_{i_ji_{j+1}}$,
there are at most finitely many classes
$\phi\in\pi_2(\y_1,...,\y_q)$ satisfying the following three conditions.\\

$\bullet$ $\phi=\phi_I\in\SpinC(H,I;\Ss)$.

$\bullet$ $\mu(\phi)=N$.

$\bullet$ $\Dcal(\phi)\geq 0$ and $\la_\z(\phi;I)\neq 0$, where $\la_\z(\phi;I)$
is defined by
$$\la_\z(\phi;I):=\prod_{i=1}^\el \la_i^{n_{z_i}(\phi)}\in \Ring_I.$$

The construction of Ozsv\'ath and Szab\'o in subsection 8.2 from \cite{OS-3m1}
may be extended to this more general context without any major modification.
Namely, for any index set $I\subset \{1,...,m\},$ and any polygon class
$\phi\in\phi_I$, the determinant line bundle of the Cauchy-Riemann
operator over $\Mod(\phi)$ is trivial, and one may thus choose an
orientation, i.e. one of the two classes of nowhere vanishing
sections of this determinant line bundle, associated with $\phi$.

\begin{defn}
A {\emph{coherent system of orientations}} associated with the Heegaard diagram
$H$ and the coherent system $\Tt$ of $\SpinC$ classes of polygons of $H$ is a choice
of orientation $\Or_I(\phi)$ for any polygon class $\phi$ with
$\phi=\phi_I \in \SpinC(H,I;\Ss)$, such that the following are satisfied.\\

$\bullet$ For any $1\leq i<j\leq m$, $\Or_{ij}$ is a coherent system of
orientations associated with the $\SpinC$ class $\spinc_{ij}$ for the
Heegaard diagram $(\Sig,\alphas^i,\alphas^j,\z)$.\\

$\bullet$ For any pair $I,J$ of attachable index sets and any attachable
polygon classes $\phi$ and $\psi$, with
$$\phi=\phi_I\in \SpinC(H,I;\Ss),\ \ \&\ \ \psi\in\SpinC(H,J;\Ss),$$
we have $$(-1)^{r(I,J)|J|}\Or_I(\phi)\wedge \Or_J(\psi)=\Or_{I\star J}(\phi\star \psi).$$
\end{defn}

Lemma~\ref{lem:polygon-compatibility} implies that in order for us to obtain
a coherent system of orientations associated with
the Heegaard diagram $H$ and the coherent system of $\SpinC$ classes of polygons
$\Tt$, it suffices to determine $\Or_{ij}$ and $\Or_{1ij}$ for any pair of
indices $1\leq i<j\leq m$. This observation implies that the following lemma,
which was proved in \cite{OS-3m1} as lemma 8.7, is valid in our setup.
Although the Heegaard diagrams are more general, the proof carries over without
any major modification.

\begin{lem}\label{lem:Choice-of-Orientation}
Suppose that the Heegaard diagram $H$, and the coherent system of $\SpinC$ classes of
polygons $\Tt$ are as above. Then for any choice of coherent systems of orientations
$\Or_{1i}$  corresponding to the $\SpinC$ classes $\spinc_{1i}$ (with $1<i\leq m$),
and any choice of $\Or_{1ij}(\phi_{1ij})$ for $1<i<j\leq m$, there always exists
a coherent system of orientations $\Or=\{\Or_I\}_I$ such that $\Or_{1i}$ is
the initial choice of the coherent system of orientations corresponding to
$\spinc_{1i}$ and $\Or_{1ij}(\phi_{1ij})$ is the prescribed orientation.
\end{lem}
\begin{proof}
For an index set $I=\{i_1<...<i_q\}$ let $i(I)=i_1$ and $j(I)=i_q$ denote the smallest and
largest element of $I$ respectively.
Let us assume that $\phi$ is a $q$-gon class in the same
$\SpinC$ class as $\phi_I$.\\

If $1\in I$, then we may assume $|I|\geq 3$, since otherwise, we already have
a choice of orientation. In this case, we may write, in a unique way,
\begin{displaymath}
\begin{split}
&\phi=\phi_I\star \phi_1\star ...\star \phi_q,\ \ \ \ \ \ \phi_j\in\pi_2(\y_j,\x_{i_ji_{j+1}})\\
&\phi_I=\phi_{1i_2i_3}\star\phi_{1i_3i_4}\star ...\star \phi_{1i_{q-1}i_{q}}.
\end{split}
\end{displaymath}
Thus $\Or_I(\phi)$ is determined if we determine all the maps $\Or_{ij}$ for
$1<i<j\leq m$ in a compatible way. Note that $\Or_{1ij}(\phi_{1ij})$ is already defined.
If otherwise $1\notin I$, we may write, again in a unique way
\begin{displaymath}
\begin{split}
&\phi=\phi_I\star \phi_1\star ...\star \phi_q,\ \ \ \ \ \ \phi_j\in\pi_2(\y_j,\x_{i_ji_{j+1}})\\
&\phi_{1i(I)j(I)}\star \phi_I=
\phi_{1i_1i_2}\star\phi_{1i_2i_3}\star ...\star \phi_{1i_{q-1}i_{q}}.
\end{split}
\end{displaymath}
Thus, in order to determine the orientation $\Or_{I}(\phi)$, it suffices to
determine all maps $\Or_{ij}$ for $1<i<j\leq m$. In order to determine
the $\Or_{ij}$ from $\Or_{1i}$, $\Or_{1j}$ and $\Or_{1ij}(\phi_{1ij})$,
one may then use the argument of lemma 8.7 from \cite{OS-3m1}.
\end{proof}

\begin{remark}
Note that the choice of $\Or_{1i}$ for $1<i\leq m$ determines the orientation
for all boundary degenerations in a unique way. In fact, suppose that
$\psi$ is the class of
some $\alphas^i$ boundary degeneration corresponding to the corner
$\y\in\mathbb{T}_{\alpha^i}\cap\mathbb{T}_{\alpha^j}$, say for some $j>i$,
and that $\phi$ is a Whitney disk in $\pi_2(\x_{ij},\y)$.
Furthermore, let $\psi'$ denote the class in $\pi_2^{\alpha^i}(\x_{1i})$
which has the same domain as $\psi$. We may then write
$$\phi_{1ij}\star \phi\star \psi= \phi_{1ij}\star\psi'\star \phi,$$
implying that $\Or_{ij}(\psi)$ is uniquely determined  by $\Or_{1i}(\psi')$,
and is equal to it as the class of an $\alpha^i$ boundary degeneration.
\end{remark}

Let $H$ be an $\Ss$-admissible Heegaard diagram,
and $\Tt$ be a system of compatible $\SpinC$ structures
as before. Correspondingly, assume that
$$\Or=\Big\{\Or_I\ |\ I\subset \{1,...,m\},\ |I|\geq 2\Big\}$$
is a coherent system of orientations associated with $\Tt$.
Associated with any subset $I=\{i_1,...,i_q\}\subset \{1,...,m\}$
of indices, we may define a holomorphic polygon map
\begin{displaymath}
\begin{split}
&f_{I}:\bigotimes_{j=1}^{q-1}
\CFT(\Sig,\alphas^{i_j},\alphas^{i_{j+1}},\z;\spinc_{i_ji_{j+1}})
\otimes_{\Ring_{\{i_j,i_{j+1}\}}}\Ring_I\\
&\ \ \ \ \ \ \ \ \ \ \ \ \ \
\lra\CFT(\Sig,\alphas^{i_1},\alphas^{i_{q}},\z;\spinc_{i_1i_{q}})
\otimes_{\Ring_{\{i_1,i_{q}\}}}\Ring_I.\\
\end{split}
\end{displaymath}
In other words, if $\{i<j\}\lhd I$ denotes that  $i$ and $j$
are consecutive elements in $I$, and $i(I),j(I)$ denote the smallest and
largest elements of $I$ respectively, we will have
a map
\begin{displaymath}
\begin{split}
&f_I:\bigotimes_{\{i<j\}\lhd I}C_{ij}(I)=\bigotimes_{j=1}^{q-1}
C_{i_ji_{j+1}}(I)\lra C_{i_1i_q}(I)=C_{i(I),j(I)}(I)\\
&f_I\big(\y_1\otimes\y_2\otimes ...\otimes\y_{q-1}\big):=
\sum_{\substack{\y_q\in\mathbb{T}_{\beta^{1}}\cap
\mathbb{T}_{\beta^{q}}\\
[\relspinc(\y_q)]=\spinct_{1q}}}\sum_{\substack{\phi\in
\pi_2(\y_1,\y_2,...,\y_{q})\\ \mu(\phi)=3-q\\ \phi\in (\phi_I)}}
\big(\m(\phi)\la_\z(\phi;I)\big).\y_{q},
\end{split}
\end{displaymath}
where $\betas^j=\alphas^{i_j}$ and $\spinct_{1q}=\spinc_{i_1i_q}$.\\

Since $H$ is  admissible, it follows that only
finitely many terms would contribute to the above sum, and
$f_I$ is thus well-defined.\\

These maps satisfy a generalized
associativity property,
which may be stated in our setup as follows
(we will only state the associativity corresponding to the
full index set $\{1,...,m\}$).
\begin{thm}\label{thm:general-associativity}
With the above notation fixed, if we set $[m]=\{1,...,m\}$
as the full index set, the map
\begin{equation}\label{eq:general-associativity}
\begin{split}
F_{[m]}&:C_{12}\big([m]\big)\otimes C_{23}\big([m]\big)\otimes...
\otimes C_{m-1,m}\big([m]\big) \lra C_{1m}\big([m]\big), \\
F_{[m]}&:=\sum_{1\leq i< j\leq m}(-1)^{ij}f_{\{1,2,...,i,j,j+1,...,m\}}
\circ f_{\{i,i+1,...,j\}}
\end{split}
\end{equation}
is trivial.
\end{thm}
\begin{proof}
Let us denote the set ${\{1,2,...,i,j,j+1,...,m\}}$
of indices by $I(i,j)$, and $\{i,i+1,...,j\}$  by $J(i,j)$.
We have to show that for any set $\y_1,...,\y_{m}$ of
intersection points with
$\y_i\in\Tai\cap\mathbb{T}_{\alpha^{i+1}}$, and such that
$\relspinc(\y_i)\in\spinc_{i,(i+1)}$, the coefficient
of $\y_m$ in
\begin{displaymath}
 \begin{split}
  \sum_{1\leq i< j\leq m}(-1)^{ij}f_{I(i,j)}\Big(\y_1\otimes...
  \otimes\y_{i-1}\otimes f_{J(i,j)}\big(\y_i\otimes...\otimes
  \y_{j-1}\big)\otimes\y_{j}\otimes...\otimes\y_{m-1}\Big)
 \end{split}
\end{displaymath}
is zero.
Let us consider a Whitney polygon class $\psi\in\pi_2(\y_1,...,\y_m)$
with Maslov index $4-m$ and in the same class as $\psi_{[m]}$,
and consider the ends of $\sM(\psi)$. The ends of this moduli space
do not contain any boundary disk degenerations or sphere bubblings.
The reason is that the Maslov index of the holomorphic boundary disk
degenerations and holomorphic spheres are greater than or equal
to $2$ if the corresponding element of the coefficient ring is non-trivial.
This would imply that the remaining component should have
Maslov index at most $2-m$. As a result, the moduli space
associated with the remaining part would be empty. \\
Thus all degenerations
of this moduli space (for dimensional reasons) are degenerations
along an arc which connects two different edges of the $m$-gon. The
ends corresponding to a degeneration along an arc connecting the
$i$-th edge to the $j$-th edge, with $i<j$ correspond
to a degeneration of $\psi$ into the
juxtaposition of a holomorphic Whitney $(j-i+1)$-gon connecting
$\y_i,...,\y_{j-1}$ to an intersection point
$\x\in\Tai\cap\Taj$ with Maslov index $2-j+i$, with
a holomorphic $(m-j+i+1)$-gon  connecting
$\y_1,...,\y_{i-1},\x,\y_j,...\y_{m-1}$ to $\y_m$ with Maslov
index $2-m+j-i$. Thus,  the ends of $\Mod(\psi)$ will have the following form.
\begin{displaymath}
\begin{split}
\partial \Mod(\psi)=\coprod_{\substack{1\leq i<j\leq m\\
\x\in\Tai\cap\Taj}}\ \ \ \coprod_{\substack{\phi_{ij}
\in\pi_2^{2-j+i}(\y_i,...,\y_{j-1},\x)\\
\psi_{ij}\in\pi_2^{2-m+j-i}(\y_1,...,\y_{i-1},\x,\y_j,...,\y_m)\\ \psi_{ij}\star
\phi_{ij}=\psi}}\Big(\Mod(\psi_{ij})\times\Mod(\phi_{ij})\Big).
\end{split}
\end{displaymath}
In the above decomposition, we are dropping the condition
that the polygons represent the  $\SpinC$ class
determined by $\Tt$. The sign difference between the orientation we  assign to the component
$\Mod(\psi_{ij})\times\Mod(\phi_{ij})$,
and its orientation as a
 boundary component of $\partial \Mod(\psi)$ is computed as
\begin{displaymath}
\epsilon\Big(\Mod(\psi_{ij})\times\Mod(\phi_{ij})\Big)=
(-1)^{r(I(i,j),J(i,j))|J(i,j)|}=(-1)^{i(j-i+1)}=(-1)^{ij}.
\end{displaymath}
Note that the total number of points in the moduli space on the right hand side
of the above equation, when counted with the above induced signs, will be zero. We
should of course mod out by possible automorphisms of the domain, when necessary.
Fix a generator $\la\in G(\Ring_{[m]})$.
The coefficient of $\la.\y_m$ in
 \begin{displaymath}
 \begin{split}
  \sum_{1\leq i< j\leq m}(-1)^{ij}\Big(\y_1\otimes...\otimes\y_{i-1}\otimes f_{J(i,j)}
\big(\y_i\otimes...\otimes\y_{j-1}\big)\otimes\y_{j}\otimes...\otimes\y_{m-1}\Big)
 \end{split}
\end{displaymath}
is equal to
\begin{displaymath}
\begin{split}
&\sum_{\substack{1\leq i<j\leq m\\ \x\in\Tai\cap\Taj}}
\ \ \ \sum_{\substack{\phi_{ij}\in\pi_2^{2-j+i}(\y_i,...,\y_{j-1},\x)\\
\psi_{ij}\in\pi_2^{2-m+j-i}(\y_1,...,\y_{i-1},\x,\y_j,...,\y_m)\\
\la(\phi_{ij})\la(\psi_{ij})=\la}}(-1)^{ij}\Big(\m(\psi_{ij})\m(\phi_{ij})\Big)\\
&\ \ \ \ \ \ \ \ =\sum_{\substack{\psi\in\pi_2^{4-m}(\y_1,...,\y_m)\\
\la(\psi)=\la}}\#\Big(\partial \big(\Mod(\psi)\big)\Big)=0.
\end{split}
\end{displaymath}
The above computation thus completes the proof of the theorem.
\end{proof}
\begin{remark}
The maps $$f_I:\bigotimes_{\{i<j\}\lhd I}C_{ij}(I)\ra C_{i(I)j(I)}(I)$$
will sometimes refine to the maps respecting the relative $\SpinC$ structures.
We will face this situation in the upcoming sections several times. Each time
we will give a separate argument for such an splitting, to avoid the complexity of
a general treatment.
 \end{remark}
 In order to prove the above associativity, we do not
 need to use the full system $\Pp$ of
 compatible $\SpinC$ structures. In fact, a subsystem
 containing the classes of polygons
 associated with the index sets $I(i,j)$ and $J(i,j)$
 suffices for this purpose. In other words,
 we only make use of the $\SpinC$ classes in the subset
 $$\Tt_1=\Big\{\phi_{I(i,j)}\ |\ 1\leq i<j\leq m\Big\}
 \cup\Big\{\phi_{J(i,j)}\ |\ 1\leq i<j\leq m\Big\}\subset \Tt$$
 for defining the maps appearing  on the left-hand-side
 of equation~\ref{eq:general-associativity}.
\begin{defn}
The set
\begin{displaymath}
\begin{split}
\Tt_1=\Big\{&\phi_{I(i,j)}\in\SpinC(H,I(i,j))\ |\ 1\leq i<j\leq m\Big\}\\
&\bigcup\Big\{\phi_{J(i,j)}\in\SpinC(H,J(i,j))\ |\ 1\leq i<j\leq m\Big\}
\end{split}
\end{displaymath}
of polygon classes is called a {\emph{system of first
degenerations}} for $\phi_{[m]}\in\SpinC(H,[m];\Ss)$
if
$$\phi_{I(i,j)}\star\phi_{J(i,j)}=\phi,\ \ \ \forall\ 1\leq i<j\leq m.$$
\end{defn}
Thus, instead of $\Tt$, we may fix a system of first
degenerations for a class $\phi_{[m]}\in\SpinC(H,[m];\Ss)$,
together with a compatible system of coherent orientations
associated with them.
Then theorem~\ref{thm:general-associativity} would still remain true.\\

\subsection{Special Heegaard diagrams corresponding to handle slides}\label{subsec:special-HD}
Let us assume that $(\Sig,\alphas,\betas,\z)$ is a Heegaard diagram, which
 corresponds to a balanced sutured manifold $(X,\tau)$. Let us assume that
$$\alphas=\big\{\alpha_1,...,\alpha_\ell\big\},\ \ \&\ \ \betas=\big\{\beta_1,...,\beta_\ell\big\},$$
and that $\beta_i$ is the image of $\alpha_i$ under a small Hamiltonian isotopy for $i=2,...,\ell$
so that $\beta_i$ is disjoint from $\alpha_j$ for $j\neq 1,i$ and cuts $\alpha_i$ in a pair of canceling
intersection points. The area bounded between $\alpha_i$ and $\beta_i$ is thus of the form
$\Pcal_i=D_i^+-D_i^-$, such that $D_i^+$ and $D_i^-$ are two of connected components in
$$\Sig-\alphas-\betas,$$
and $\partial \Pcal_i=\alpha_i-\beta_i$. Furthermore, assume that $\beta_1$ is obtained from
$\alpha_1$ by first moving it by a small Hamiltonian isotopy, and then doing a handle slide along
$\alpha_2$. Thus, the only curve in $\alphas\cup\betas$ that intersects $\beta_1$
is $\alpha_1$, which cuts $\beta_1$ in a pair of intersection points. These two intersection points
are connected by a bi-gon, which we will denote by $D_1^+$. There is a domain with small
area which is bounded between $\alpha_1,\beta_1,\alpha_2$ and $\beta_2$, denoted by
$D_1^-$, so that $\Pcal_1=D_1^+-D_1^--D_2^-$ is a periodic domain satisfying
$$\partial \Pcal_1=\alpha_1+\alpha_2-\beta_1.$$
We will assume that none of  the marked points $\z=\{z_1,..,z_\el\}$ are in
any of  $D_1^+,...,D_\ell^+$ or $D_1^-,...,D_\ell^-$. Let us assume
$$\Sig-\alphas-\betas=\Big(\coprod_{i=1}^\ell D_i^+\Big)\bigcup
\Big(\coprod_{i=1}^\ell D_i^-\Big)\bigcup \Big(\coprod_{i=1}^m E_i\Big),$$
and that $\z^i=\{z^i_1,...,z^i_{j_i}\}$ are the marked points in $E_i$ for
$i=1,...,m$. Thus, $\z=\z^1\cup ...\cup\z^m$. In the ring $\Ring_\tau$, let
$\la(z^i_j)$ denote the element associated with $z^i_j\in\z$. Furthermore,
define
$$\mu_i=\prod_{j=1}^{j_i}\la(z^i_j)\in \Ring_\tau,\ \ \ i=1,...,m.
$$
If $\Sig-\alphas=\coprod_{i=1}^{m_a} A_i$ and $\Sig-\betas=
\coprod_{i=1}^{m_b} B_i$, we will have $m_a=m_b=m$, and
after renaming the indices if necessary, we may assume
$\z\cap A_i=\z\cap B_i$, and $\la(A_i)=\la(B_i)=\mu_i$.
Let us denote by $\A_\mu$ the sub-ring of $\A_\tau$ generated by
$\mu_1,...,\mu_m$.\\

Any pair of curves $(\alpha_i,\beta_i)$ intersect in a pair $x_i^+,x_i^-$
of points, so that the bi-gon $D_i^+$ connects $x_i^+$ to $x_i^-$. Any
map $\epsilon:\{1,...,\ell\}\ra \{+,-\}$ thus corresponds to an intersection
point
$$\x^\epsilon=\big\{x_1^{\epsilon(1)},x_2^{\epsilon(2)},...,x_\ell^{\epsilon(\ell)}\big\}\in\Ta\cap\Tb.$$
These all correspond to the same $\SpinC$ class in $\SpinC(\ovl{X}^{\tau})$, which will be
denoted by $\spinc_0$. For $\epsilon:\{1,...,\ell\}\ra \{+,-\}$
let $|\epsilon|$ denote the number of elements in $\epsilon^{-1}\{+\}$.
We may refine the homological grading of the generators of $\CFT(\Sig,\alphas,\betas,\z;\spinc_0)$
into a relative $\Z$-grading by setting
$$\mathrm{gr}(\epsilon,\delta)=|\epsilon|-|\delta|,\ \ \ \forall \
\epsilon,\delta:\{1,...,\ell\}\lra \{+,-\}.$$
We will show below that this gives a well-defined relative grading in an appropriate sense.
\\

\begin{figure}[ht!]
\begin{center}
\includegraphics[totalheight=0.3\textheight]{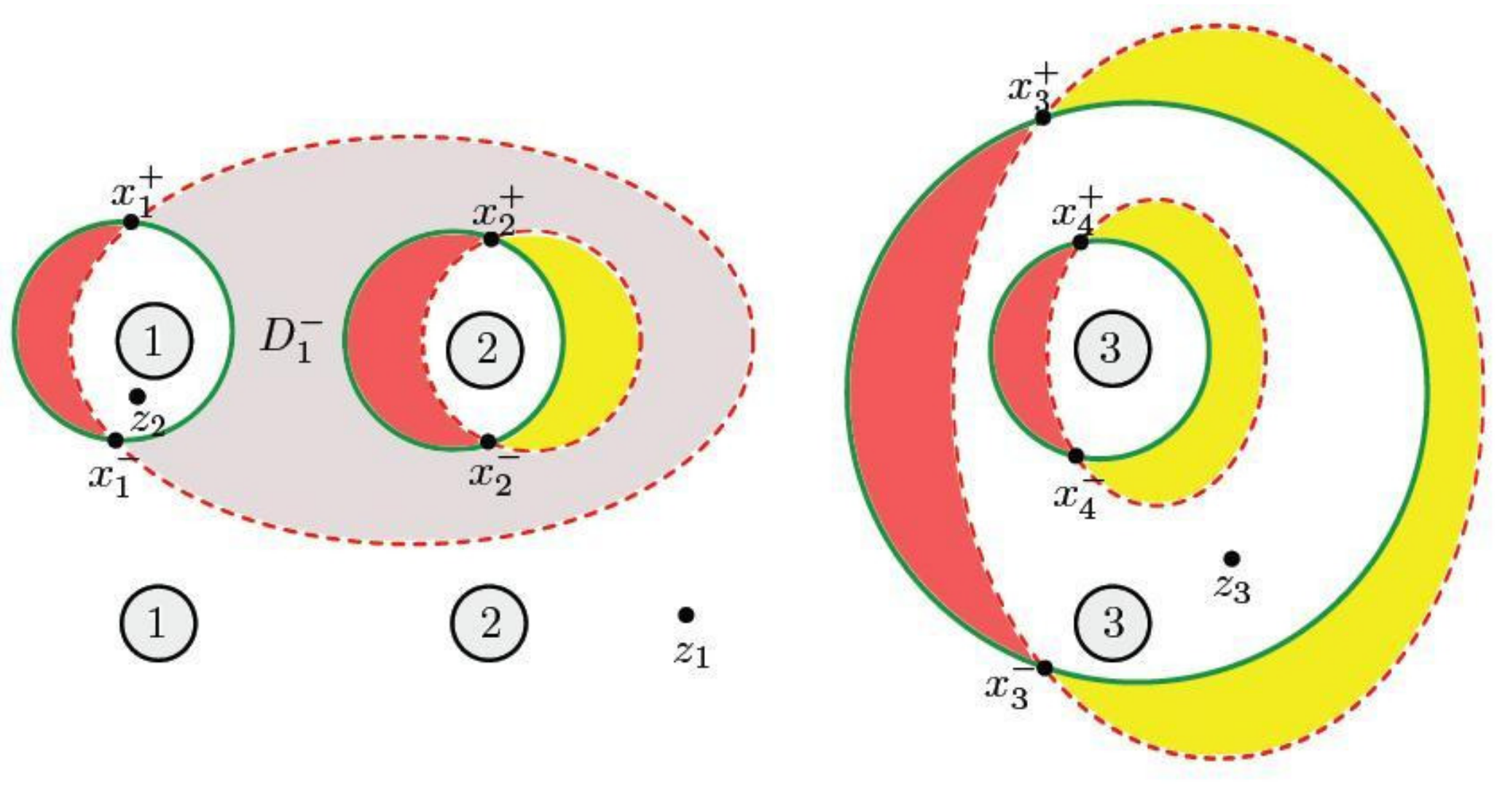}
\caption{The green curves denote the elements of $\alphas$ and the dashed red curves denote the
elements of $\betas$. The domains $D_i^+$ are shaded red, while
the domains $D_i^-$ are shaded yellow. The domain
$D_1^-$  is shaded gray.}\label{fig:Special-HD}
\end{center}
\end{figure}


The periodic domains corresponding to the above Heegaard diagram
are generated, as a free abelian group, by
$\Pcal_1,...,\Pcal_\ell, A_1,...,A_m$.
Note that $\la(\Pcal_i)=1$. If $$\Pcal=q_1\Pcal_1+...
+q_\ell\Pcal_\ell+a_1A_1+...+a_mA_m\geq 0$$
is a positive periodic domain with $\langle c_1(\spinc), H(\Pcal)\rangle=0$ for some
$\spinc\in\SpinC(\ovl X^\tau)$, we will have $a_1,...,a_m\geq 0$ and
$$0=a_1(2-2g(A_1))+a_2(2-2g(A_2))+...+a_m(2-2g(A_m)).$$
Here $g(A_i)$ denotes the genus of $A_i$.
The reason is that the evaluation of the $\SpinC$ classes over the above
periodic domains is independent of $\spinc$. If moreover we know that
$\la(\Pcal)\neq 0$, we may conclude that  if $a_i\neq 0$ then $g(A_i)=0$.
Let us assume that $A_1,...,A_k$ are the components of genus zero,
and the rest of $A_i$ have positive genus. This implies that
$$a_1,...,a_m\geq 0,\ \ 0=a_1+...+a_k,\ \ \& \ \ a_{k+1}=...=a_m=0.$$
Thus all $a_i$ are zero, and $\Pcal=q_1\Pcal_1+...+q_\ell\Pcal_\ell$.
Since $\Pcal_2,...,\Pcal_\ell$ are disjoint, and all $\Pcal_i$ have both
positive and negative coefficients, one can easily conclude that
$\Pcal$ has both positive and negative coefficients. Thus the
constructed Heegaard diagram is admissible for all $\SpinC$
classes.\\

One may choose an orientation $\Or$ for this Heegaard diagram.
For any choice of orientation, one may observe that $D^+_i$ and
$D^-_i$ are both domains of Whitney disks, for $i=2,...,\ell$, and
the number of points in
$$\ov\Mod(D^+_i)\bigcup \ov\Mod(D^-_i)$$
is zero (in fact, the orientation of the two moduli spaces are the opposite
of one another, in some sense). Thus there is a filtered chain complex
$\CFT(\Sig,\alphas,\betas,\z;\spinc_0)$
associated with the above Heegaard diagram, which is freely generated by
the set $\{\x^\epsilon\}_\epsilon$. This set consists of $2^\ell$ generators.
\\

The complex $\CFT(\Sig,\alphas,\betas,\z;\spinc)$ is thus trivial for
$\spinc\neq \spinc_0$ and is equal to
$$\CFT(\Sig,\alphas,\betas,\w;\spinc_0)\otimes_{\Ring_\mu}\Ring_\tau $$
for $\spinc=\spinc_0$ where $\w=\{z^1_1,z^2_1,...,z^m_1\}$. For the rest
of the computation, we may thus assume that $\w=\z$, i.e. that there is a
single marked point in each connected component of $\Sig-\alphas$.
Each $\mu_i$ (or under the assumption $\w=\z$, each $\la_i$)
corresponds to some component $A_i$. Thus $\mu_i=0$ if the genus
of $A_i$ is positive. With our previous notation, this means that
$\mu_{k+1}=...=\mu_{m}=0$. We set the degree associated with
$\mu_i$, $i=1,...,k$, equal to $-2=\langle c_1(\spinc_0), H(A_i)\rangle$.
This gives a grading on the complex
$\CFT(\Sig,\alphas,\betas,\w;\spinc_0)$.
\\

Let us assume that $\epsilon,\delta:\{1,...,\ell\}\ra \{-,+\}$ are a pair of
indices. After re-naming the elements of $\{1,...,\ell\}$ we may assume that
\begin{displaymath}\begin{cases}
\epsilon(i)=+\ \&\ \delta(i)=-\ \ &\text{if } 1\leq i\leq \ell_1\\
\epsilon(i)=-\ \&\ \delta(i)=+\ \ &\text{if } \ell_1< i\leq \ell_2\\
\epsilon(i)=\delta(i)\ \ &\text{if }\ell_2<i\leq \ell,
  \end{cases}
\end{displaymath}
for some $1\leq \ell_1\leq \ell_2\leq \ell$.
Then $\Dcal=D^+_1+...+D^+_{\ell_1}-D^+_{\ell_1+1}-...-D^+_{\ell_2}$ is the domain
of a disk connecting $\x^\epsilon$ to $\x^\delta$. If $\Pcal$ is a periodic domain
and $\Dcal+\Pcal$ is the domain of a positive disk $\phi$ with 
$\la(\phi)\neq 0$, the same argument as before implies that
$$\Pcal=a_1A_1+...+a_kA_k+q_1\Pcal_1+...+q_\ell \Pcal_\ell.$$
Furthermore, the assumption $\Dcal(\phi)\geq 0$ implies that all $a_i$ are
non-negative. In order to prove that the above grading assignment
is well-defined, one only needs to check the following easy equality
\begin{equation}\label{eq:index}
\mu(\phi)=2\left(\sum_{i=1}^k a_i\right)+2\ell_1-\ell_2.
\end{equation}

If the coefficient ring $\Rinn_\tau$ is used instead of
$\Ring_\tau$, the corresponding quotient $\Rinn_\mu$ of
$\Ring_\mu$ will be equal to $\Z$.
One would then quickly conclude from the above presentation of the domain $\Pcal$ that
if $\la(\phi)\neq 0$ (as an element in the quotient $\Rinn_\tau$), then  $a_1=...=a_k=0$
and $\ell_2=\ell_1$.
If $\mu(\phi)=1$ then $\ell_1=\ell_2=1$. We can then carry out the rest of the argument for
any choice of indices $\epsilon$ and $\delta$, if the coefficient ring is replaced
with $\Rinn_\tau$.\\

With coefficients in $\Ring$, however, in order to complete
our investigation we need to assume $\epsilon(i)=\{+\}$ for $i=1,...,\ell$. The corresponding
generator is often called the {\emph{top generator}}. In this case the equality
$\ell_1=\ell_2$ is automatic. Replacing $\mu(\phi)=1$ in equation~\ref{eq:index}
we obtain $a_1=...=a_k=0$ and $\ell_1=1$. From here we will have (from positivity of the domain)
 that $q_{2}=...=q_\ell=0$.
This means that if the top intersection point $\x^\epsilon$ is connected to
an intersection point $\x^\delta$ by a  positive domain $\phi$ of index $1$,
such that $\la(\phi)\neq 0$,
$\delta$ differs from $\epsilon$ only over one element of $\{1,...,\ell\}$, where
$\epsilon$ gives $+$ and $\delta$ gives $-$. Let us assume that this element is
$i\in\{1,...,\ell\}$.\\

 If $i\neq 1$, the possible domains one may obtain as $\Dcal+\Pcal$ are
$D^+_i$ and $D_i^-$, and the total contribution of $\x^\delta$ to
$\partial (\x^\epsilon)$ is zero. However, if $i=1$, the possible domains are
$D^+_1, D^-_1+D^-_2$ and $D^-_1+D^+_2$. Again, the total contribution of these
three domains is zero by the argument of \cite{OS-3m1} (lemma~9.4).
The above discussion implies that the top generator $\x^\epsilon$
is closed and represents a non-trivial element of the homology groups
corresponding to either of chain complexes
$$\CFT(\Sig,\alphas,\betas,\w;\spinc_0)\ \ \&\ \
\CFT(\Sig,\alphas,\betas,\z;\spinc_0)=
\CFT(\Sig,\alphas,\betas,\z;\spinc_0)\otimes_{\Ring_\mu}\Ring_\tau.$$
Moreover, the same argument implies that
all the generators of the form $\x^\delta$
are closed, when the coefficient ring $\Rinn_\tau$ is used instead of $\Ring_\tau$,
giving rise to an identification of the chain complexes:
$$\CFT(\Sig,\alphas,\betas,\z;\spinc_0;\Rinn)=
\ov{\mathrm{HF}}(\#^\ell S^1\times S^2,\spinct_0)\otimes_{\Z}\Rinn_\tau.$$
The above equality means that the differential on the right hand side of the
above equality is trivial. Any module isomorphism of the right hand side which
respects the filtration by relative $\SpinC$ structures is thus a
filtered chain homotopy equivalence.
The top generator $\x^\epsilon$ of $\CFT(\Sig,\alphas,\betas,\z;\spinc_0)$ is usually denoted by $\Theta$, or
$\Theta_{\alpha\beta}$.\\

The above example illustrates how the arguments of Ozsv\'ath and Szab\'o for the
study of the special Heegaard diagrams, i.e. Heegaard diagrams where most of $\beta_i$ are
Hamiltonian isotopes of curves in $\alphas$, may be generalized to the present situation.
The above type of Heegaard diagrams appear in the arguments for the invariance under
handle-slide.
We will face similar Heegaard diagrams again. In particular this happens when we study the exact triangles.
Each time, a separate argument should be presented for computing the contribution
of holomorphic disks and polygons. However 
the argument is always a straight forward modification of the corresponding
argument for Heegaard diagrams arising from closed three-manifolds.\\

\subsection{The triangle map and the invariance}
Fix a  Heegaard triple
$$H=(\Sig,\alphas,\betas,\gammas,\z)$$ and  assume that
$\z=\{z_1,...,z_\el\}$.  We will denote the balanced sutured manifold
associated
with $(\Sig,\alphas,\betas,\z)$ by
$(X,\tau)=(X_{\alpha\beta},\tau_{\alpha\beta})$,
 and the corresponding
coefficient ring by $\Ring=\Ring_\tau$.
Similarly, let $(X_{\alpha\gamma},\tau_{\alpha\gamma})$
and $(X_{\beta\gamma},\tau_{\beta\gamma})$ be the
balanced sutured manifolds associated with the Heegaard diagrams
$(\Sig,\alphas,\gammas,\z)$, and
$(\Sig,\betas,\gammas,\z)$ respectively.
Suppose that
$$\Sig-\alphas=\coprod_{i=1}^k A_i,\ \ \Sig-\betas=
\coprod_{i=1}^lB_i,\ \ \&\ \ \Sig-\gammas=\coprod_{i=1}^mC_i,$$
where $A_i,B_i$ and $C_i$ are the connected components
of the curve complements. We will furthermore assume that $m=l$, and that
these components are labeled  so that for
each $i=1,...,l$ we have
$C_i\cap \z=B_i\cap \z$, and $g(C_i)=g(B_i)$. This implies that
$\la(\betas)=\la(\gammas)$ in $\langle \la_1,...,\la_\el\rangle_\F$, and
more importantly, $\la(C_i)=\la(B_i)$ for $i=1,...,l$.
\\

Assume that the coefficient rings $\Ring_{\beta\gamma}$ and
$\Ring_{\alpha\gamma}$ are
associated with the Heegaard
diagrams $(\Sig,\betas,\gammas,\z)$ and
$(\Sig,\alphas,\gammas,\z)$ respectively.
Then the above observation implies that
$$\Ring_{\beta\gamma}=\frac{\Big\langle\la_1,...,\la_\el\Big\rangle_\F}
{\Big\langle \la(B_i)\ |\ g(B_i)=g(C_i)>0\Big\rangle}
$$
is naturally mapped by a quotient homomorphism
$$\rho_{\beta\gamma}:\Ring_{\beta\gamma}\lra \Ring_{\alpha\beta}=\Ring_\tau$$
to $\Ring$.
We may thus
consider the $\Ring$-module
$$\CFT(\Sig,\betas,\gammas,\z)\otimes_{\Ring_{\beta\gamma}}\Ring,$$
which will have the structure of a filtered $\tab$ chain complex,
with $\Hbb=\Ht^2(X,\partial X;\Z)$.\\

Any triangle class
$\psi_H\in \RelSpinC\big(H,\{\alphas,\betas,\gammas\}\big)$
determines a set of three $\SpinC$ structures
\begin{displaymath}
 \begin{split}
 & \spinc_{\alpha\beta}
 \in\SpinC\Big(\overline{X_{\alpha\beta}}(\tau_{\alpha\beta})\Big),\ \ \
\spinc_{\alpha\gamma}
\in\SpinC\Big(\overline{X_{\alpha\gamma}}(\tau_{\alpha\gamma})\Big),\ \ \&\ \
  \spinc_{\beta\gamma}
 \in\SpinC\Big(\overline{X_{\beta\gamma}}(\tau_{\beta\gamma})\Big).
\end{split}
\end{displaymath}
These three $\SpinC$ classes, together with the triangle class
of $\psi_H$ give a coherent system of  $\SpinC$ structures for
$H$ which will be denoted by $\Tt$.  We will assume that $\Tt$,
or equivalently the triangle class $\psi_H$, is fixed, and
will drop them from the notation when there is no confusion.
In particular, by the admissibility of a Heegaard diagram
we would mean  $\Tt$-admissibility.\\

Any choice of coherent systems of orientations $\Or_{\alpha\beta}$ and
$\Or_{\alpha\gamma}$ associated with the $\SpinC$ classes $\spinc_{\alpha\beta}$
and $\spinc_{\alpha\gamma}$ may be completed to a  coherent system of
orientations for $\Tt$ by lemma~\ref{lem:Choice-of-Orientation}. Furthermore,
we are free to choose the orientation associated with a fixed representative
of $\psi_H$. Let us fix such a coherent system $\Or$ of orientations.
Once again, we will drop this choice of orientation from the notation. \\

Assuming that the
 Heegaard triple $H$ is admissible,
the triangle map corresponding to $H$ (and the triangle
class $\psi_H$) is
defined via the construction of subsection~\ref{subsec:m-gons}.
\begin{displaymath}
\begin{split}
&f_{\alpha\beta\gamma}:\CFT(\Sig,\alphas,\betas,\z;
\spinc_{\alpha\beta})\otimes
\Big(\CFT(\Sig,\betas,\gammas,\z;\spinc_{\beta\gamma})
\otimes_{\Ring_{\beta\gamma}}
\Ring_\tau\Big)\lra\CFT(\Sig,\alphas,\gammas,\z;
\spinc_{\alpha\gamma})\\
&f_{\alpha\beta\gamma}(\x\otimes\q)=\sum_{\y\in\Ta\cap\Tg}
\sum_{\substack{\psi\in\pi_2^0(\x,\q,\y)\\
(\psi)=(\psi_H)}}
\big(\m(\psi){\la_\z(\psi)}\big).\y
\end{split}
\end{displaymath}
As usual, $\la_\z$ is the map
\begin{displaymath}
\begin{split}
&\la_\z:\coprod_{\x\in\Ta\cap\Tb}\coprod_{\y\in\Tb\cap\Tc}
\coprod_{\w\in\Tc\cap\Ta}\pi_2^+(\x,\y,\w)\lra G(\Ring)\\
&\la_\z(\psi):=\prod_{i=1}^\el \la_i^{n_{z_i}(\psi)}\in G(\Ring).
\end{split}
\end{displaymath}
The admissibility of the Heegaard diagram implies that
$f_{\alpha\beta\gamma}$ is well-defined.\\

\begin{lem}
The map $f_{\alpha\beta\gamma}$ is an $\ta$ chain
map. More precisely
$$f_{\alpha\beta\gamma}\Big(\partial(\x)\otimes\q\Big)+
f_{\alpha\beta\gamma}\Big(\x\otimes\partial(\q)\Big)=
\partial \Big(f_{\alpha\beta\gamma}\big(\x\otimes\q\big)\Big)$$
for all $\x\in\Ta\cap\Tb$ and $\q\in\Tb\cap\Tc$ corresponding to the
$\SpinC$ classes $\spinc_{\alpha\beta}$ and $\spinc_{\beta\gamma}$ respectively.
\end{lem}
\begin{proof}
 The equality
$$f_{\alpha\beta\gamma}\Big(\partial(\x)\otimes\q\Big)
+f_{\alpha\beta\gamma}\Big(\x\otimes\partial(\q)\Big)
=\partial \Big(f_{\alpha\beta\gamma}\big(\x\otimes\q\big)\Big)$$
in the above lemma is nothing but the following special
case (i.e. the case $m=3$) of theorem~\ref{thm:general-associativity}:
$$f_{\alpha\beta\gamma}\Big(f_{\alpha\beta}(\x)\otimes\q\Big)
+f_{\alpha\beta\gamma}\Big(\x\otimes f_{\beta\gamma}(\q)\Big)-
f_{\alpha\gamma} \Big(f_{\alpha\beta\gamma}\big(\x\otimes\q\big)\Big)=0.$$
\end{proof}

As in \cite{OS-3m1}, holomorphic triangle maps satisfy
an associativity law, which comes from considering  Heegaard quadruples.
Let $K=(\Sig,\alphas,\betas,\gammas,\deltas,\z)$ be an
admissible  Heegaard quadruple.
This means that we have a coherent system $\Tt$ of
$\SpinC$ classes of polygons, which consists of a square class
$$\psi_K\in \RelSpinC\Big(K,\{\alphas,\betas,\gammas
,\deltas\}\Big) $$ and
triangle classes
\begin{displaymath}
 \begin{split}
  &\psi_\alpha\in \RelSpinC\Big(K,\{\betas,\gammas,
  \deltas\}\Big),\ \ \ \ \psi_\beta\in \RelSpinC\Big(
  K,\{\alphas,\gammas,\deltas\}\Big)\\
&\psi_\gamma\in \RelSpinC\Big(K,\{\alphas,\betas,\deltas\}
\Big),\ \ \&\ \ \psi_\delta\in \RelSpinC\Big(K,\{\alphas,
\betas,\gammas\}\Big).
 \end{split}
\end{displaymath}
We implicitly assume that the set of corners of these
representatives of the triangle
classes is a fixed set of $6$ intersection points between
the pairs of tori from $\{\Ta,\Tb,\Tc,\Td\}$.
These classes have to satisfy the following compatibility criteria
$$\psi_K=\psi_\alpha\star\psi_\gamma=\psi_\beta\star\psi_\delta.$$
The triangle classes also determine $\SpinC$ structures on the Heegaard
diagrams determined by any pair of curve collections.
These $\SpinC$ structures will be
denoted by $\spinc_{\alpha\beta}\in\RelSpinC(\overline{X_{\alpha\beta}}
(\tau_{\alpha\beta}))$, etc.. Moreover,
we will assume that
$$\Sig-\alphas=\coprod_{i=1}^k A_i,\ \ \Sig-\betas=\coprod_{i=1}^lB_i,
\ \ \Sig-\gammas=\coprod_{i=1}^lC_i,\ \ \&\ \ \Sig-\deltas=\coprod_{i=1}^lD_i,$$
are labeled so that $B_i\cap\z=C_i\cap\z=D_i\cap\z$ for
$i=1,...,l$. Furthermore, we will assume that $g(B_i)=g(C_i)=g(D_i)$ for $i=1,...,l$.
Then we will have $\la(\betas)=\la(\gammas)=\la(\deltas)$
in $\langle \la_1,...,\la_\el\rangle$, and $\la(B_i)=\la(C_i)=\la(D_i)$ for  $i=1,...,l$.\\

One may also choose a coherent system of orientations associated
with $\Tt$. In fact, we are free to choose
$\Or_{\alpha\beta},\Or_{\alpha\gamma}, \Or_{\alpha\delta}$, and the
orientation of the triangle classes $\psi_\beta, \psi_\gamma$ and
$\psi_\delta$. Once again, we keep such a coherent system of orientations
implicit in our notation.\\

We may thus consider the following filtered
$\tab$  chain complexes, which are relevant for the associativity:
\begin{displaymath}
\begin{split}
C_{\alpha\beta}=\CFT(\Sig,\alphas,\betas,\z;
\spinc_{\alpha\beta}),\ \ \ \ \ \ \
&C_{\beta\gamma}=\CFT(\Sig,\betas,\gammas,\z;\spinc_{\beta\gamma})
\otimes\Ring_\tau\\
C_{\alpha\gamma}=\CFT(\Sig,\alphas,\gammas,\z;\spinc_{\alpha\gamma})
,\ \ \ \ \ \ \ &C_{\beta\delta}=\CFT(\Sig,\betas,
\deltas,\z;\spinc_{\beta\delta})\otimes\Ring_\tau\\
C_{\alpha\delta}=\CFT(\Sig,\alphas,\deltas,\z;\spinc_{\alpha\delta}),
\ \ \& \ \ &C_{\gamma\delta}=\CFT(\Sig,\gammas,
\deltas,\z;\spinc_{\gamma\delta})\otimes\Ring_\tau.
\end{split}
\end{displaymath}
Following the construction of subsection~\ref{subsec:m-gons},
we  define a rectangle map as in \cite{OS-3m1}:
\begin{displaymath}
\begin{split}
&h_{\alpha\beta\gamma\delta}:C_{\alpha\beta}
\otimes C_{\beta\gamma}\otimes C_{\gamma\delta}\lra C_{\alpha\delta}\\
&h_{\alpha\beta\gamma\delta}(\x\otimes\p\otimes\q)=\sum_{\y\in\Ta\cap\Td}
\sum_{\substack{\psi\in\pi_2(\x,\p,\q,\y)\\
\mu(\psi)=-1}}\big(\m(\psi)\la(\psi)\big)\y.
\end{split}
\end{displaymath}
\begin{lem}\label{lem:associativity}
The rectangle map $h_{\alpha\beta\gamma\delta}$
gives a chain homotopy between the chain maps
$f_{\alpha\gamma\delta}(f_{\alpha\beta\gamma}
(.\otimes.)\otimes.)$ and $f_{\alpha\beta\delta}
(.\otimes f_{\beta\gamma\delta}(.\otimes.))$
in the sense that
\begin{displaymath}
\begin{split}
f_{\alpha\gamma\delta}&\Big(f_{\alpha\beta\gamma}
\big(\x\otimes\p\big)\otimes\q\Big)-f_{\alpha\beta\delta}\Big(\x\otimes
f_{\beta\gamma\delta}\big(\p\otimes\q\big)\Big)\\
&=\partial\Big(h_{\alpha\beta\gamma\delta}
\big(\x\otimes\p\otimes\q\big)\Big)+
h_{\alpha\beta\gamma\delta}\Big(\partial
\big(\x\otimes\p\otimes\q\big)\Big)
\end{split}
\end{displaymath}
for any $\x\in\Ta\cap\Tb,\p\in\Tb\cap\Tc$, and $\q\in\Tc\cap\Td$.
\end{lem}
\begin{proof}
Once again, this is a special case of theorem~
\ref{thm:general-associativity}, where we put $m=4$ and
use the above data.
\end{proof}
\begin{proof}{\bf Theorem}~\ref{thm:invariance}
The proof of the independence  from the choice of
the path of almost complex structures, as well
as the proof of the
isotopy 
invariance of the
filtered $\tab$ chain homotopy type of $\CFT(\Sig,\alphas,\betas,\z;\spinc)$
is the same as the proof of the special case discussed in \cite{OS-3m1}. We only need to
keep track of the marked points, and that the constructed chain
homotopy equivalence respects the decomposition into
relative $\SpinC$ classes in $\RelSpinC(X,\tau)$.
The same is almost true for the handle slides
supported away from the marked points. We will present the
proof in this case, to give an illustration of
the procedure, which involves the use of holomorphic triangles
and squares introduced above.\\

Fix a $\SpinC$ class $\spinc\in\SpinC(\overline{X})$ and let
$\relspinc\in\spinc\subset \RelSpinC(X,\tau)$ be a fixed relative
$\SpinC$ class in $\spinc$.
To prove the handle slide invariance consider
the  Heegaard quadruple
$(\Sigma,\alphas,\betas,\gammas,\deltas,\z)$
where $\gammas$ and $\deltas$ are obtained from $\betas$
as follows. Let $\betas=\{\beta_1,...,\beta_\ell\}$.
Then we let $\delta_i$ to be a small Hamiltonian isotope of $\beta_i$
for $i=1,...,\ell$ which cuts it in a pair of transverse canceling
intersection points. Similarly, for $i=2,...,\ell$, we let $\gamma_i$ be
a small Hamiltonian isotope of $\beta_i$ which cuts either of
the curves $\beta_i$ and $\delta_i$ in a pair of transverse canceling intersection points.
Finally, we let $\gamma_1$ be the simple closed curve obtained by first moving
$\beta_1$ by a small Hamiltonian isotopy, and then taking its handle slide over
$\beta_2$. We may assume that $\gamma_1$ cuts either of $\beta_1$ and
$\delta_1$ in a pair of canceling intersection points, while it is disjoint from
 the rest of the curves $\beta_i,\gamma_i$ and $\delta_i$.
 We let
 $$\gammas=\big\{\gamma_1,...,\gamma_\ell\big\},\ \
 \&\ \ \deltas=\big\{\delta_1,...,\delta_\ell\big\}.$$

Consider the (admissible)  Heegaard diagram
$(\Sig,\betas,\gammas,\z)$, which is a standard
Heegaard diagram of the type studied in subsection~\ref{subsec:special-HD}. Note that all
marked points which are in the same connected
component of $\Sig-\betas$ or $\Sig-\gammas$ are in the same
connected component of $\Sig-\betas-\gammas$.
Let $\Theta_{\beta\gamma}$ be the top generator of the complex
$$\CFT(\Sig,\betas,\gammas,\z;\spinc_0)
$$ corresponding to its canonical $\SpinC$ structure.
This generator is represented by the
intersection point in $\Tb\cap\Tc$ which contains
positive intersection points
between the corresponding curves $\beta_i$ and $\gamma_i$.
Similarly, associated with the  Heegaard diagram $(\Sig,\betas,\deltas,\z)$
the top generator of  the homology is denoted by a
$\Theta_{\beta\delta}$, and is represented by the
positive intersection points
 between the corresponding curves $\beta_i$ and $\delta_i$.
 Finally, $\Theta_{\gamma\delta}$ is defined in a similar way.
 We may consider $\Theta_{\beta\gamma},\Theta_{\gamma\delta}$
 and $\Theta_{\beta\delta}$ as generators of the
 complexes $C_{\beta\gamma},C_{\gamma\delta}$
 and $C_{\beta\delta}$ respectively. Here, we assume
 $\spinc_{\alpha \bullet}=\spinc$ for $\bullet \in\{\beta,\gamma,\delta\}$,
 and that
 $$\spinc_{\beta\gamma}=\spinc_{\beta\delta}=\spinc_{\gamma\delta}=\spinc_0$$
is the canonical $\SpinC$ structure on
$\ovl{X_{\beta\gamma}}=\ovl{X_{\gamma\delta}}=\ovl{X_{\beta\delta}}$.
Note that $\Theta_{\beta\gamma},\Theta_{\gamma\delta}$
and $\Theta_{\beta\delta}$  are connected to each other by a natural
triangle class $\Delta_\alpha$ of small area. Moreover,
for any fixed $\x\in\Ta\cap\Tb$ with  $\relspinc(\x)\in\spinc$,
we have a generator $I(\x)\in\Ta\cap\Td$, determined by
the closest intersection points in $\Ta\cap\Td$  to $\x$.
Similarly, there is a generator $J(\x)$ in $\Ta\cap\Tc$
determined as the closest intersection points between
$\alphas$ and $\gammas$ to $\x$. There is a triangle
class $\Delta_\gamma$ connecting
$\Theta_{\beta\delta}$, $\x$ and $I(\x)$ with very small
area. Similarly, there is a triangle class $\Delta_\delta$ connecting
$\Theta_{\beta\gamma}$, $\x$ and $J(\x)$ with very small
area. Finally, there is a triangle class $\Delta_\beta$ which connects
$I(\x),J(\x)$ and $\Theta_{\gamma\delta}$.
Let $\square$ be the square class
$\Delta_\gamma\star\Delta_\alpha$. Then
$\square$ may also be degenerated as
$\square=\Delta_\delta\star\Delta_\beta$. The data
$$\Pp=\big\{\square,\Delta_\alpha,\Delta_\beta,
\Delta_\gamma,\Delta_\delta\big\}$$
thus gives a coherent system $\Tt$ of  $\SpinC$ classes of polygons for the
Heegaard quadruple,
which will be implicit for the rest of the construction.\\

\begin{lem}
The Heegaard quadruple $H=(\Sig,\alphas,\betas,\gammas,\deltas,\z)$ is
$\Tt$-admissible, provided that $H=(\Sig,\alphas,\betas,\z)$ is
$\spinc$-admissible.
\end{lem}
\begin{proof}
We will prove the lemma for the class of $\square$. The rest of the
admissibility claims are similar, and in fact simpler.
Let us denote the small periodic domains constructed as the domain bounded between
$\beta_i$ and $\delta_i$ by $\Qcal_i$ for $i=1,...,\ell$. Thus $\Qcal_i$ is the difference of
two bi-gons. Similarly, let $\Qcal_{i+\ell}$ be the domain bounded between $\beta_i$ and $\gamma_i$
for $i=2,...,\ell$, and $\Qcal_{\ell+1}$ be the domain bounded between $\beta_1,\gamma_1$ and $\beta_2$,
as in the previous subsection. We will thus have
\begin{displaymath}
\begin{cases}
\partial \Qcal_i=\beta_i-\delta_i,\ \ \ &\text{for }i=1,...,\ell,\\
\partial \Qcal_{i+\ell}=\beta_i-\gamma_i,\ \ \ &\text{for }i=2,...,\ell,\ \&\\
\partial \Qcal_{\ell+1}=\beta_1+\beta_2-\gamma_1.&
\end{cases}
\end{displaymath}
Finally, let $A_1,...,A_k,B_1,...,B_l, \Pcal_1,...,\Pcal_m$ be the periodic domains
corresponding to $(\Sig,\alphas,\betas,\z)$.
As before $\Sig-\alphas=\coprod A_i$ and $\Sig-\betas=\coprod B_i$.
It may be checked
then that the space of periodic domains for the Heegaard diagrams
$(\Sig,\betas,\gammas,\z)$, $(\Sig,\betas,\deltas,\z)$ and
$(\Sig,\gammas,\deltas,\z)$  is generated by the following
periodic domains respectively
\begin{displaymath}
\begin{split}
&\Big\langle B_1,....,B_l,\Qcal_{\ell+1},...,\Qcal_{2\ell}\Big\rangle,\ \
\Big\langle B_1,....,B_l,\Qcal_{1},...,\Qcal_{\ell}\Big\rangle,\ \&\\
&\Big\langle \ov{B}_1,....,\ov{B}_l,\Qcal_1+\Qcal_2-\Qcal_{\ell+1},
\Qcal_2-\Qcal_{2+\ell},...,\Qcal_\ell-\Qcal_{2\ell}\Big\rangle.\\
\end{split}
\end{displaymath}
Here, $\ov{B}_i$ is the domain obtained from $B_i$ by adding an
appropriated combination of $\Qcal_j,\ j=1,...,\ell$, so that
its boundary is supported on the curves in $\deltas$.\\

Let us now assume
that we have a periodic domain $\Pcal$ with
\begin{displaymath}
\begin{split}
&\Pcal=\Pcal_{\alpha\beta}+\Pcal_{\beta\gamma}+\Pcal_{\gamma\delta}+\Pcal_{\alpha\delta}\geq 0,\ \
\la(\Pcal)\neq 0,\ \ \&\\
&\Big\langle c_1(\spinc),H(\Pcal_{\alpha\beta})\Big\rangle+
\Big\langle c_1(\spinct_0),H(\Pcal_{\beta\gamma})\Big\rangle+
\Big\langle c_1(\spinct_0),H(\Pcal_{\gamma\delta})\Big\rangle+
\Big\langle c_1(\spinc),H(\Pcal_{\alpha\delta})\Big\rangle=0
\end{split}
\end{displaymath}
$\Pcal$ may then be written as
$$\Pcal=\sum_{i=1}^ka_iA_i+\sum_{i=1}^lb_iB_i+\sum_{i=1}^mp_i\Pcal_i+\sum_{i=1}^{2\ell}q_i\Qcal_i.$$
With the above notation fixed, computing the evaluation of $\SpinC$ classes over the periodic
domains (i.e. re-writing the last equation above) we obtain
\begin{equation}\label{eq:Maslov-index-admissibility}
\begin{split}
0&=\sum_{i=1}^ka_i(2-2g(A_i))+\sum_{i=1}^lb_i(2-2g(B_i))
+\sum_{i=1}^mp_i\big\langle c_1(\spinc),H(\Pcal_i)\big\rangle,
\end{split}
\end{equation}
since the Maslov index of all $\Qcal_i$ are zero for all $\SpinC$ structures,
according to Lipshitz' index formula \cite{Robert-cylindrical}.
Let us set
$$\Qcal=\sum_{i=1}^ka_iA_i+\sum_{i=1}^lb_iB_i+\sum_{i=1}^mp_i\Pcal_i.$$
Then $\Qcal$ is a periodic domain for the Heegaard diagram
$(\Sig,\alphas,\betas,\z)$, with $\Pcal-\Qcal$ only consisting of the domains
with very small area.
The assumption $\Pcal\geq 0$ thus implies that $\Qcal\geq 0$.
Furthermore $\la(\Qcal)=\la(\Pcal)$, since no marked point
lives in the small domains.
The equation~\ref{eq:Maslov-index-admissibility} implies that
$\langle c_1(\spinc),H(\Qcal)\rangle=0$. The $\spinc$-admissibility of
of the Heegaard diagram $(\Sig,\alphas,\betas,\z)$ thus implies that
$\Qcal=0$. As a result,
$$\Pcal=\sum_{i=1}^{2\ell}q_i\Qcal_i\geq 0.$$
It is then an easy combinatorial exercise to check from this last equality
that all $q_i$ need to vanish. We have thus shown that $\Pcal=0$. This
completes the proof of the admissibility claim.
\end{proof}
Finally, the last step towards defining the holomorphic triangle map and
the holomorphic square map using the Heegaard diagram $H$ is choosing the
orientation. Note that the choice of orientation over the Heegaard diagrams
$(\Sig,\alphas,\betas,\z)$, $(\Sig,\alphas,\gammas,\z)$, and
$(\Sig,\alphas,\deltas,\z)$ may be done without any restriction, and we may thus
choose the orientation on $(\Sig,\alphas,\betas,\z)$, and the induced orientations
on the other Heegaard diagrams as our preferred choice of orientation.
Orienting the triangles and the square in $\Tt$ will then provide us
with a coherent system $\Or$ of orientations for the Heegaard diagram $H$.\\

 We may thus define the triangle and the square maps associated
with this Heegaard diagram and $\Tt$.
The argument of Ozsv\'ath and Szab\'o from \cite{OS-3m1} (lemma 9.7)
applies here to give
\begin{displaymath}
f_{\beta\gamma\delta}(\Theta_{\beta\gamma}\otimes
\Theta_{\gamma\delta})=\Theta_{\beta\delta}.
\end{displaymath}

We may thus define a map
 $$F=F_{\alpha\beta\gamma}:\CFT(\Sig,\alphas,\betas,\z;\spinc)
 \lra \CFT(\Sig,\alphas,\gammas,\z;\spinc)$$
 by setting $F(\x):=f_{\alpha\beta\gamma}\big(\x\otimes
 \Theta_{\beta\gamma}\big)$. Since $\Theta_{\beta\gamma}$ is
 closed, $F$ is a chain map. More importantly, $F$
 respects the decomposition into relative $\SpinC$
 structures,
and the image of $\CFT(\Sig,\alphas,\betas,\z;\relspinc)$,
for the fixed relative $\SpinC$ structure
$$\relspinc\in\spinc\subset \RelSpinC(X_{\alpha\beta},
\tau_{\alpha\beta})=\RelSpinC(X,\tau)$$
is in $\CFT(\Sig,\alphas,\gammas,\z;\relspinc)$.
Let us denote by $G$ the similar filtered $\tab$ chain map
 $$G=F_{\alpha\gamma\delta}:\CFT(\Sig,\alphas,
 \gammas,\z;\spinc)\lra \CFT(\Sig,\alphas,
 \deltas,\z;\spinc)$$
 defined by $G(\y):=f_{\alpha\gamma\delta}\big(\y
 \otimes\Theta_{\gamma\delta}\big)$.
Also, define the map
$$H=H_{\alpha\beta\gamma\delta}:\CFT(\Sig,\alphas,
\betas,\z;\spinc)\lra\CFT(\Sig,\alphas,\deltas,\z;\spinc)$$
by $H(\x):=h_{\alpha\beta\gamma\delta}\big(\x\otimes
\Theta_{\beta\gamma}\otimes\Theta_{\gamma\delta}\big)$.
Checking that all the above maps respect the relative
$\SpinC$ structures is straight forward.
Using lemma~\ref{lem:associativity}, and the fact that
$f_{\beta\gamma\delta}(\Theta_{\beta\gamma}\otimes
\Theta_{\gamma\delta})=\Theta_{\beta\delta}$ we have
\begin{displaymath}
\begin{split}
G\big(F(\x)\big))-
f_{\alpha\beta\delta}\Big(\x\otimes \Theta_{\beta\delta}\Big)
=\partial\Big(H\big(\x\big)\Big)+
H\Big(\partial\big(\x\big)\Big).
\end{split}
\end{displaymath}
Small triangles which contribute to
$f_{\alpha\beta\delta}\Big(\x\otimes
\Theta_{\beta\delta}\Big)$ may be used to show that
in terms of an appropriate energy filtration we have
$$f_{\alpha\beta\delta}\Big(\x\otimes
\Theta_{\beta\delta}\Big)=I(x)+\epsilon(\x)$$
where  $\epsilon(\x)$ consists of a
combination of generators with smaller energy than $\x$.
This implies that there is a filtered $\tab$ chain equivalence
$$K:\CFT(\Sig,\alphas,\deltas,\z;\spinc)\lra
\CFT(\Sig,\alphas,\betas,\z;\spinc),$$
respecting the decomposition according to relative
$\SpinC$ structures, such that
$K\Big(f_{\alpha\beta\delta}\big(\x\otimes
\Theta_{\beta\delta}\big)\Big)=\x$. Thus
setting $G'=K\circ G$ and $H'=K\circ H$ we have
$$G'\circ F-Id=H'\circ \partial+\partial \circ H',$$
and $G'\circ F$ is chain homotopic to the identity.
The other composition is similarly
chain homotopic to the identity. This completes the
proof of the handle slide invariance.\\

The invariance under isotopy and stabilization-destabilization is
completely similar to the proofs presented in sections 7 and 10
of \cite{OS-3m1}.
Thus the filtered $\tab$ chain homotopy type
of $\CFT(\Sig,\alphas,\betas,\z;\spinc)$ is
an invariant of $(X,\tau,\spinc)$, and
will be denoted by
$$\CFT(X,\tau;\spinc)=
\bigoplus_{\relspinc\in\spinc}\CFT(X,\tau;\relspinc).$$
\end{proof}

Theorem~\ref{thm:invariance}, together with remark~\ref{remark:weak-admissibility}
imply that for computing $\CFT(X,\tau;\spinc;\Ringg)$, given a test ring
$\Ringg$ for $\Ring_\tau$, one may use any  Heegaard diagram
 which is weakly $\spinc$-admissible in the sense of remark~\ref{remark:weak-admissibility}.
 In particular, we can easily prove the following proposition.

\begin{prop}\label{prop:non-taut-sutured-manifolds}
The irreducible balanced sutured manifold $(X,\tau)$ is taut if and only if the filtered
$(\Ringg_\tau,\Hbb_\tau)$ chain homotopy type of the
complex $\CFT(X,\tau;\spinc;\Ringg_\tau)$ is not trivial for
some $\SpinC$ structure $\spinc\in\SpinC(\ovl X)$.
\end{prop}
\begin{proof}
Suppose that $(X,\tau)$ is an irreducible balanced sutured manifold which is not taut.
As in the proof of proposition 9.18 from \cite{Juh}, there is a weakly admissible
Heegaard diagram $(\Sig,\alphas,\betas,\z)$ for $(X,\tau)$ such that
$\Ta\cap\Tb$ is empty. This Heegaard diagram is $\spinc$-admissible for
the test ring $\rho_\tau:\Ring_\tau\ra \Ringg_\tau$ and for any
$\spinc\in\SpinC(\ovl X)$ by remark~\ref{remark:weak-admissibility},
and may thus be used to compute
$$\CFT(X,\tau;\spinc;\Ringg_\tau)\cong 0,\ \ \ \forall\ \spinc\in\SpinC(\ovl X),$$
where $\cong$ denotes the equivalence of filtered chain homotopy types.\\

Conversely, if $(X,\tau)$ is taut, theorem 1.4 from \cite{Juh-surface}
implies that $\HFT(X,\tau;\Z)\neq 0$. Since $\Z$ is a test ring for
$\Ringg_\tau$ this implies, in particular, that the filtered chain homotopy
type of $\CFT(X,\tau;\Ringg_\tau)$ is non-trivial.
\end{proof}
In fact, the proof of proposition~\ref{prop:non-taut-sutured-manifolds} implies the following corollary.
\begin{cor}
For an irreducible balanced sutured manifold $(X,\tau)$, the filtered $(\Ringg_\tau,\Hbb_\tau)$
chain homotopy type of $\CFT(X,\tau;\Ringg_\tau)$ is trivial if and only if
$\mathrm{SFH}(X,\tau)=0$.
\end{cor}

\newpage

\section{Stabilization of sutured manifolds}\label{sec:stabilization}
\subsection{Simple stabilization of a balanced sutured manifold}
Let us fix a balanced sutured manifold $(X,\tau=\{\gamma_1,...,\gamma_\el\})$  and  let
$$\Rr^-(\tau)=\bigcup_{i=1}^k R_i^-\ , \ \&\ \Rr^+(\tau)=\bigcup_{j=1}^lR_j^+,$$
as before.
\begin{defn} We say that
 a sutured manifold $(X,\ov{\tau})$ is obtained by a {\emph{simple stabilization}} of $(X,\tau)$ if $\ov{\tau}=\tau\cup\{\gamma_{\kappa+1},\gamma_{\kappa+2}\}$ and  $\gamma_{\kappa+1}$ and $\gamma_{\kappa+2}$ are oriented simple closed curves  so that $-\gamma_{\kappa+1}$ and $\gamma_{\kappa+2}$ are both
 parallel to an oriented  suture $\gamma_i\in \tau$,
 where $\gamma_i$ is in the common boundary of two genus zero components of $\Rr(\tau)$.
  Moreover, $\gamma_i$ and $\gamma_{\kappa+1}$ bound an annulus in $\partial X-\ov\tau$.
\end{defn}

Without loss of generality we assume that $i=\el$, and that
$\gamma_i=\gamma_{\kappa}\in\partial R_k^-\cap \partial R_{l}^+$ and $\gamma_{\kappa+1},\gamma_{\kappa+2}\subset R_l^+$.
So the genera of both $R_k^-$ and $R_l^+$ are zero. Note that
\begin{displaymath}
\begin{split}
&\Rr^-(\ov{\tau})=\Rr^-(\tau)\coprod A_{\kappa+1,\kappa+2},\ \ \&\\
&\Rr^+(\ov{\tau})=(\Rr^+(\tau)-R_l^+)\coprod A_{\kappa,\kappa+1}\coprod R_{l+1}^+\\
\text{where }\ \ \ \
&R_l^+-\left(\gamma_{\kappa+1}\cup\gamma_{\kappa+2}\right)
=A_{\kappa+1,\kappa+2}\coprod A_{\kappa,\kappa+1}\coprod R_{l+1}^+
\end{split}
\end{displaymath}
 and $A_{\kappa+1,\kappa+2}$ and $A_{\kappa,\kappa+1}$ are the annulus components
 with the boundary sets $\{\gamma_{\kappa+1},\gamma_{\kappa+2}\}$
 and $\{\gamma_\kappa,\gamma_{\kappa+1}\}$ respectively. Thus we have
\begin{displaymath}
\begin{split}
\la^+(\ov{\tau})&=\la^+(\tau)+\la_{\kappa}\la_{\kappa+1}+\la_{l+1}^{+}-\la_l^+,\ \&\\ \la^-(\ov{\tau})&=\la^-(\tau)+\la_{\kappa+1}\la_{\kappa+2}
\end{split}
\end{displaymath}
where $\la_{l+1}^{+}=\la(R_{l+1}^+)$.
The algebra associated to the boundary of $(X,\ov{\tau})$ is defined by
\begin{displaymath}
\Ring_{\ov{\tau}}=\frac{\Big\langle\la_1,...,\la_{\kappa}
,\la_{\kappa+1},\la_{\kappa+2}\Big\rangle_{\Z}}
{\Big\langle\la^{+}(\ov{\tau})-\la^{-}(\ov{\tau})\Big\rangle_{\Z[\kappa+2]}
+\Big\langle\la_i^+|g_i^+>0\Big\rangle_{\Z[\kappa+2]}
+\Big\langle\la_j^-|g_j^->0\Big\rangle_{\Z[\kappa+2]}}
\end{displaymath}
Note that the homomorphism
\begin{displaymath}
\begin{split}
&f:\Z[\la_1,...,\la_{\kappa+2}]\longrightarrow\Z[\la_1,...,\la_{\kappa+1}]\\
&f(\la_i)=
\begin{cases}
\la_i\ \ \ \ &\text{if}\ \ i\neq\kappa+2\\
\la_{\kappa}\ \ \ \ &\text{if}\ \ i=\kappa+2
\end{cases}
\end{split}
\end{displaymath}
induces a homomorphism $\ov{f}:\Ring_{\ov{\tau}}\longrightarrow\Ring_{\tau}[\la_{\kappa+1}]$,
where
$$\Ring_{\tau}[\la_{\kappa+1}]=\Ring_{\tau}
\otimes_{\Z[\la_1,...,\la_{\kappa}]}\Z[\la_1,...,\la_{\kappa+1}].$$
Consequently, $(\Ring_{\tau}[\la_{\kappa+1}],\ov{f})$ is a test ring
for $\Ring_{\ov{\tau}}$.\\

Let $(\Sig,\alphas,\betas,\z)$ be a Heegaard diagram for $(X,\tau)$.
A Heegaard diagram $(\Sig,\ov{\alphas},\ov{\betas},\ov{\z})$ for $(X,\ov{\tau})$
may then be constructed from $(\Sig,\alphas,\betas,\z)$ as follows. We set
$$\ov{\alphas}=\alphas\cup\big\{\alpha_{\ell+1}\big\}, \ \&\
\ov{\betas}=\betas\cup\big\{\beta_{\ell+1}\big\},\ \&\
\ov{\z}=\z\cup\big\{z_{\kappa+1},z_{\kappa+2}\big\}, $$
 where the additional curves $\alpha_{\ell+1}$ and $\beta_{\ell+1}$
 are isotopic simple closed curves on $\Sig$ in the the domain of $\Sig-\alphas\cup\betas$ containing the
 marked point $z_\el$ with the following properties.
 We assume that $\#\alpha_{\ell+1}\cap\beta_{\ell+1}=2$ and that
 $\alpha_{\ell+1}$ and $\beta_{\ell+1}$ bound the disks $A_{k+1}$ and $B_{l+1}$ respectively.
 Furthermore, we assume  that
 $$z_{\kappa}\in B_{l+1}-A_{k+1},\ \ z_{\kappa+1}\in A_{k+1}\cap B_{l+1},\ \
 \&\ \ z_{\kappa+2}\in A_{k+1}-B_{l+1}.$$
 The picture around the marked point $z_{\el}$ is illustrated in figure~\ref{fig:stabilization}.

\begin{figure}[ht]
\def\svgwidth{8cm}
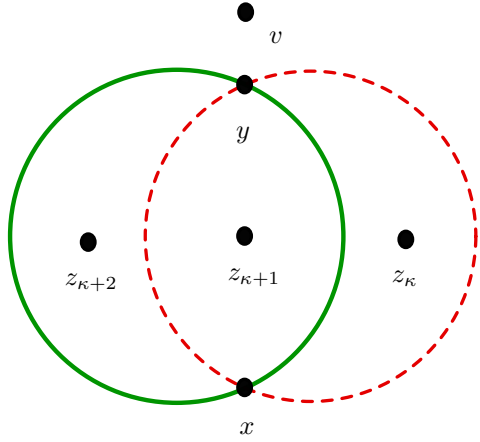
\caption{For simple stabilization, a pair of intersecting null-homotopic simple closed curves
$\alpha_{\ell+1}$ and $\beta_{\ell+1}$ are added to the Heegaard diagram close to the
marked point $z_{\el}$. The locations of the new marked points $z_\el,z_{\el+1}$ and
$z_{\el+2}$ are illustrated in the figure. The marked point $v$ is used as the connected
sum point of the current diagram (on a Riemann sphere) with the old Heegaard diagram.
 }\label{fig:stabilization}
\end{figure}

 Note that $H(A_{k+1})=[S_{\kappa+1,\kappa+2}]$ and $H(B_{l+1})=[S_{\kappa,\kappa+1}]$ in
 $H_2(\ovl{X}_{\ov{\tau}}=\ovl X^{\ov\tau})$ where
 $S_{\kappa,\kappa+1}$ and $S_{\kappa+1,\kappa+2}$ are sphere
 boundary components of $\ovl{X}_{\ov{\tau}}$ corresponding
 to $A_{\kappa,\kappa+1}$ and $A_{\kappa+1,\kappa+2}$ in $\Rr(\ov{\tau})$ respectively.
In the above situation,we say the Heegaard diagram $(\Sig,\ov{\alphas},\ov{\betas},\ov{\z})$
is obtained  from $(\Sig,\alphas,\betas,\z)$ by a simple stabilization.\\

We may define a natural map
$$\is:\SpinC(X,\ov{\tau})\longrightarrow\SpinC(X,\tau)$$
as follows. Fix $\ov{\relspinc}\in\SpinC(X,\ov{\tau})$ and let $\ov{v}$
be a nowhere vanishing vector field on $X$ representing $\ov{\relspinc}$ such that
$\ov{v}|_{\partial X}=v_{\ov{\tau}}$.
Consider a neighborhood $N$ of
$$A=\ovl{A}_{\kappa,\kappa+1}\cup \ovl{A}_{\kappa+1,\kappa+2}\subset\partial X$$
together with a diffeomorphism
\begin{displaymath}
\begin{split}
&\psi:N\longrightarrow S^1\times I\times I,\ \ \ \text{s.t.}\\
&\psi(A)=S^1\times\{0\}\times I,\ \ \&\ \
\psi_{*}\left(\ov{v}|_N\right)|_{S^1\times I\times\{0,1\}}
=\frac{\partial}{\partial s},
\end{split}
\end{displaymath}
 where $I=[0,1]$ is the unit interval and $s$ denotes the standard parameter
 on the third component of the product $S^1\times I\times I$.
 The vector field
 $\psi_{*}(\ov{v})$ may be changed through an isotopy
 to a new vector field $\psi_*(v)$ on $S^1\times I\times I$
 with the property $\psi_*(v)|_{S^1\times\{0\}\times I}=\frac{\partial}{\partial s}$,
 where the vector field remains  fixed through the isotopy on
 $$\left(S^1\times\big\{1\big\}\times I\right)\bigcup
 \left(S^1\times I\times\big\{0\big\}\right)\bigcup
 \left(S^1\times I\times\big\{1\big\}\right).$$
 The vector field $v$ on $N$ may be glued to $\ov v|_{X-N}$ to give a vector field on $X$, still denote by $v$,
which represents an element $\relspinc$ in $\SpinC(X,\tau)$.
Define $\is(\ov{\relspinc}):=\relspinc$. It is easy to see that $\is$ is well-defined and bijective, and that
the following diagram is commutative
\begin{displaymath}
\begin{diagram}
\RelSpinC(X,\ov\tau)&\rTo{\is}&\RelSpinC(X,\tau)\\
\dTo{[.]}&&\dTo{[.]}\\
\SpinC(\ovl X_{\ov\tau})&\rTo{i_{\ovl X}^*}&\SpinC(\ovl X)
\end{diagram},
\end{displaymath}
 where $i_{\ovl X}:\ovl X\ra \ovl X_{\ov\tau}$ denotes the
 inclusion map.\\

\begin{prop}\label{prop:stabilization}
Let $(X,\ov{\tau})$ be a sutured manifold obtained by a simple stabilization
on $(X,\tau)$ and $(\Sig,\ov{\alphas},\ov{\betas},\ov{\z})$ be the Heegaard
diagram for $(X,\ov{\tau})$ obtained by the simple stabilization of the
Heegaard diagram $(\Sig,\alphas,\betas,\z)$ for $(X,\tau)$ as above.
Then for any $\SpinC$ class $\spinc\in\SpinC(\ovl X_{\ov\tau})$
the filtered chain homotopy type of the complex   $$\mathrm{CF}\left(\Sig,\ov{\alphas},\ov{\betas},\ov{\z},\spinc;\Ring_{\tau}[\la_{\kappa+1}]\right)
=\mathrm{CF}(\Sig,\ov{\alphas},\ov{\betas},\ov{\z},\spinc)
\otimes_{\Ring_{\ov{\tau}}}\Ring_{\tau}[\la_{\kappa+1}]$$
is the same as the filtered chain homotopy type of the mapping cone of
$$\mathrm{CF}\left(\Sig,\alphas,\betas,\z,i_{\ovl X}^*(\spinc);\Ring_\tau[\la_{\kappa+1}]\right)
\xrightarrow{\la_{\kappa+1}-\la}\mathrm{CF}\left(\Sig,\alphas,\betas,\z,
i_{\ovl X}^*(\spinc);\Ring_\tau[\la_{\kappa+1}]\right).$$
where we define
\begin{displaymath}
\begin{split}
\mathrm{CF}\left(\Sig,\alphas,\betas,\z,i_{\ovl X}^*(\spinc);\Ring_\tau[\la_{\kappa+1}]\right)
&:=\mathrm{CF}\left(\Sig,\alphas,\betas,\z,i_{\ovl X}^*(\spinc)\right)
\otimes_{\Ring_\tau}\Ring_\tau[\la_{\kappa+1}],\\
\& \ \ \ \ \  \la&:=\prod_{\substack{\gamma_\el\neq \gamma_i\in\partial R_{l}^+}}
\la_{i}\in \Ring_\tau[\la_{\el+1}].
\end{split}
\end{displaymath}
\end{prop}

\subsection{The analytic input}
Before we start the proving proposition~\ref{prop:stabilization}, we need to rephrase the statements
of theorem 5.1, lemma 6.3 and lemma 6.4 of \cite{OS-linkinvariant} for balanced sutured manifolds
and the corresponding Heegaard diagrams. \\

Let $(X,\tau)$ be a sutured manifold with the Heegaard diagram $(\Sig,\alphas,\betas,\z)$
and consider a point $v$ on $\Sig$. Let $\phi\in\pi_2(\x,\y)$ be the homotopy class of a
Whitney disk connecting intersection points $\x$ and $\y$, and assume that
 $n_v(\phi)=k\in\Z^{\geq 0}$. We may define a map
\begin{align}
\nonumber
\rho^{v}&:\sM(\phi)\longrightarrow\Sym^k(\D)\\
\nonumber
\rho^v(u)&=u^{-1}(v\times\Sym^{\ell-1}(\Sig))
\end{align}
Correspondingly,  we may define the moduli spaces $\sM(\phi,t)$ and
$\sM(\phi,\Delta)$ by
\begin{displaymath}
\begin{split}
\sM(\phi,t)&=\Big\{u\in\sM(\phi)\ \big|\ (t,0)\in\rho^v(u)\Big\},\ \&\\
\sM(\phi,\Delta)&=\Big\{u\in\sM(\phi)\ \big|\ \rho^{v}(u)=\Delta\Big\}
\end{split}
\end{displaymath}
where $t\in [0,1]$ and $\Delta\in\Sym^{k}(\D)$.\\


Let $(X_1,\tau_1)$ and $(X_2,\tau_2)$ be sutured manifolds with corresponding Heegaard diagrams
$(\Sig_1,\alphas_1,\betas_1,\z_1)$ and $(\Sig_2,\alphas_2,\betas_2,\z_2)$. We can form their
connected sum along the points $w$ and $v$ on $\Sig_1$ and $\Sig_2$ to obtain a new
sutured manifold $(X,\tau)$ with the corresponding Heegaard diagram $(\Sig,\alphas,\betas,\z)$ where
$$\Sig=\Sig_1\#\Sig_2,\ \ \  \alphas=\alphas_1\cup\alphas_2,\ \ \
\betas=\betas_1\cup\betas_2\ \ \  \&\ \ \ \z=\z_1\cup\z_2.$$
Note that $\Ta\cap\Tb=(\T_{\alpha_1}\cap\T_{\beta_1})\times(\T_{\alpha_2}\cap\T_{\beta_2})$.
Consider intersection points $\x_1,\y_1\in \T_{\alpha_1}\cap\T_{\beta_1}$ and $\x_2,\y_2\in\T_{\alpha_2}\cap\T_{\beta_2}$. Any homology class
$$\phi\in \pi_2(\x_1\times\x_2,\y_1\times\y_2)$$ can be uniquely decomposed as
$\phi=\phi_1\#\phi_2$ where
$$\phi_1\in\pi_2(\x_1,\y_1),\ \
\phi_2\in\pi_2(\x_2,\y_2),\ \ \&\ \  n_w(\phi_1)=n_v(\phi_2).$$
Conversely, any pair of homology classes $\phi_1\in\pi_2(\x_1,\y_1)$ and $\phi_2\in\pi_2(\x_2,\y_2)$
such that $n_w(\phi_1)=n_v(\phi_2)$ can be combined to a give
homology class $$\phi=\phi_1\#\phi_2\in \pi_2(\x_1\times\x_2,\y_1\times\y_2).$$

\begin{thm}\label{thm:gluing}
Let $(X_1,\tau_1)$ and $(X_2,\tau_2)$ be balanced sutured manifolds with the corresponding
Heegaard diagrams $(\Sig_1,\alphas_1,\betas_1,\z_1)$ and
$(\Sig_2,\alphas_2,\betas_2,\z_2)$ respectively.
Consider the balanced sutured manifold $(X,\tau)$ obtained by taking the
connected sum of the two Heegaard diagrams along $w$ and $v$ as described
above. Furthermore, assume that  $w$ (respectively, $v$) is in a genus zero
connected component of either of $\Sig_1-\alphas_1$, and $\Sig_1-\betas_1$ (respectively,
$\Sig_2-\alphas_2$,$\Sig_2-\betas_2$). For any homotopy class $$\phi=\phi_1\#\phi_2\in\pi_2(\x_1\times\x_2,\y_1\times\y_2)$$ we then have
$$\mu(\phi)=\mu(\phi_1)+\mu(\phi_2)-2k$$
where $\x_1,\y_1\in \T_{\alpha_1}\cap\T_{\beta_1}$, $\x_2,\y_2\in\T_{\alpha_2}\cap\T_{\beta_2}$ and $k=n_{w}(\phi_1)=n_v(\phi_2)$. \\
Suppose furthermore that $\mu(\phi_1)=1$, $\mu(\phi_2)=2k$ and one of the following is true:\\

$\bullet$ At least one component of $\Rr(\tau_2)$ has nonzero genus and $\la(\phi)\neq 0$.

$\bullet$ All the components of $\Rr(\tau_2)$ are genus zero, and
$$\ell_2=|\alphas_2|=|\betas_2|>g(\Sig_2).$$

If the fibered product
$$\sM(\phi_1)\times_{\mathrm{Sym}^k(\D)}\sM(\phi_2)=
\Big\{u_1\times u_2\in\sM(\phi_1)\times\sM(\phi_2)\ \big|\ \rho^w(u_1)=\rho^v(u_2)\Big\}$$
of $\sM(\phi_1)$ and $\sM(\phi_2)$ is a smooth manifold, then by taking the length of the
connected sum tube
sufficiently large there is an identification of this moduli space with $\sM(\phi)$.\\
\end{thm}
\begin{proof}
The proof is similar to the proof of theorem 5.1 in \cite{OS-linkinvariant}.
As in that proof we use Lipshitz cylindrical formulation.
However, we keep the same notation for the moduli spaces and the corresponding maps
for the sake of simplicity.\\

The formula for Maslov index follows from excision principle for the linearized
$\dbar$ operator, using cylindrical formulation\cite{Robert-cylindrical}.
For the second part of the theorem, if all components of $\Rr(\tau)$ be genus zero and
$\ell_2>g(\Sig_2)$, the proof of  theorem 5.1 from \cite{OS-linkinvariant} applies word by word.
In the other case, the proof requires some  modification, as follows.
We drop the details and only highlight the differences. For more details, we refer the reader
to \cite{OS-linkinvariant}.
\\

 Suppose that $\Rr(\tau)$ has a component with nonzero genus.
 Consider a sequence of almost complex structures $\{J(t)\}_{t\in[1,\infty)}$
 on $\Sig_1\#\Sig_2$, where $J(t)$ denotes the complex structure determined
 by a pair of generic complex structures $j_1$ and $j_2$ on $\Sig_1$ and
 $\Sig_2$, and by setting the neck-length equal to $t$.
 Let us assume that $\ti{\sM}_{J(t)}(\phi)\neq\emptyset$ as  $t\longrightarrow\infty$.
 Consider a sequence of pseudo-holomorphic curves
 $\{{u}_t\}_{t\in\Z^{+}}$ such that ${u}_t\in{\sM}_{J(t)}(\phi)$.
 Under the assumptions $\mu(\phi_1)=1$ and $\mu(\phi_2)=2k$, and using
 Gromov compactness theorem, a subsequence of this sequence is weakly convergent to
 a pseudo holomorphic representative ${u}_1$ of $\phi_1$
 and a broken flow-line representative of $\phi_2$.
 This broken flow-line can not contain any sphere bubblings,
 since otherwise our assumption on $\Rr(\tau)$ implies that
 $\la(\phi_2)=0$, and thus $\la(\phi)=0$. Hence we may continue
 the argument of Ozsv\'ath and Szab\'o from  here there, and conclude that
 there is a component ${u}_2$ of this broken flow line such that
 $\ti{u}_1$ and $\ti{u}_2$ represents a pre-glued flow line, i.e. that
$${\rho}^{w}({u}_1)={\rho}^{v}({u}_2).$$
Let $\phi_2'$ be the homotopy class represented by ${u}_2$. If $\phi_2'\neq\phi_2$
then the above Gromov limit contains boundary degenerations or other flow lines.
The assumption $\la(\phi_2)\neq 0$ then implies that $\mu(\phi_2')<\mu(\phi_2)=2k$.
Let us consider the  map
$${\rho}^v:{\sM}(\phi_2')\longrightarrow\Sym^{k}(\D).$$
For any point  $\Delta\in\Sym^{k}(\D)$ the moduli space $({\rho}^v)^{-1}(\Delta)$
will have the expected dimension equal to $\mu(\phi')-2k<0$.
Thus for a generic choice of $\Delta\in\Sym^{k}(\D)$, this moduli space is empty.
This observation implies that $\phi_2'=\phi_2$.\\

Thus the Gromov limit of a sequence of holomorphic representative of $\phi$,
as we stretch the neck, is a pre-glued flow line representing
$\phi_1$ and $\phi_2$. Conversely, given a pre-glued flow line, one can obtain a
pseudo-holomorphic representative of $\phi$ in ${\sM}(\phi)$ by the
gluing theorem of Lipshitz \cite{Robert-cylindrical}, as in
the proof of theorem 5.1 from \cite{OS-linkinvariant}. This completes the proof
of theorem~\ref{thm:gluing}.
\end{proof}
\begin{lem}\label{lem:degeneration}
Let $(X,\tau)$ be a balanced sutured manifold represented by  the Heegaard diagram
$(\Sig,\alphas,\betas,\z)$. Let $\phi\in\pi_2(\x,\y)$ be the homotopy class of a Whitney
disk connecting the intersection point $\x$ to $\y$. Assume furthermore
that $\Dcal(\phi)\geq 0$ and  that $\la(\phi)\neq 0$.
If $\mu(\phi)=2$ then $\sM(\phi,t)$ is generically a zero dimensional moduli space.
Furthermore, there is a number $\epsilon>0$ such that for all $t\le\epsilon$ the only
non-empty such moduli spaces $\sM(\phi,t)$ are the moduli spaces corresponding to
$\phi\in\pi_2(\x,\x)$ where $\phi$ is obtained by splicing a boundary degeneration
with Maslov index 2 corresponding to one of the genus zero components of $\Rr^+(\tau)$
and a constant flow line. For any such moduli space we have
\begin{displaymath}
\#\sM(\phi,t)=\begin{cases}
0\ \ \ &\text{if }\ l=1\\
1\ \ \ &\text{if }\ l>0
\end{cases}.
\end{displaymath}
\end{lem}
\begin{proof}
Given lemma~\ref{lem:disk-degeneration-2}, the proof is exactly
the same as the proof of lemma 6.3 in \cite{OS-linkinvariant}.
\end{proof}
\begin{lem}\label{lem:sphere}
Consider the Heegaard diagram $(S,\alpha,\beta,\z)$, where $S=S^2$ is the Riemann sphere,
and $\alpha$ and $\beta$ are simple closed curves on $S$ intersecting transversely
in two points $\{x,y\}$, and $\z=\{z_1,z_2,z_3,z_4\}$ where there is one marked point
in each one of the four connected components of $S-\alpha-\beta$.
Fix a generic point $\Delta\in\Sym^{k}(\D)$ for some positive integer $k$. Then we have
$$\sum_{\substack{\phi\in\pi_2^{2k}(a,a)\\
n_{z_4}(\phi)=0}}\#\sM(\phi,\Delta)=1,$$
for $a\in\{x,y\}$.
\end{lem}
\begin{proof}
Using lemma~\ref{lem:disk-degeneration-2}, the proof is exactly the same
as the proof of lemma 6.4 in \cite{OS-linkinvariant}.
\end{proof}

\subsection{Proof of the stabilization formula}
In this subsection we prove proposition~\ref{prop:stabilization}.

\begin{proof}{\bf (of proposition~\ref{prop:stabilization}).}
Let $\spinc\in\SpinC(\ovl{X}_{\hat{\tau}})$ be a $\SpinC$ structure on $\ovl{X}_{\hat{\tau}}$.
Consider a Heegaard diagram $(\Sig,\alphas,\betas,\z)$ for $(X,\tau)$, which is
$i_{\ovl X}^*(\spinc)$-admissible.
Furthermore, assume that  for any positive periodic domain $\Pcal$ we
have the following implication:
$$\Big\langle c_1\big(i_{\ovl X}^*(\spinc)\big),H(\Pcal)\Big\rangle\le 0\ \ \Rightarrow\ \ \la(\Pcal)=0.$$
The existence of such a Heegaard diagram is guaranteed by remark~\ref{remark:admissibility}.
Let $\ov H=(\Sig,\ov{\alphas},\ov{\betas},\ov{\z})$ be the Heegaard diagram for $(X,\ov{\tau})$
obtained by a simple stabilization on $(\Sig,\alphas,\betas,\z)$. We claim that this Heegaard diagram is
$\spinc$-admissible.
Suppose that $\Pcal$ is a positive periodic domain corresponding to $\ov H$
such that $\langle c_1(\spinc),H(\Pcal)\rangle=0$.
Then there are integers $a$ and $b$, and a positive periodic domain $\Pcal_0$ for $H$ such that $$\Pcal=\Pcal_0+aA_{k+1}+bB_{l+1}\ \&\
n_{z_{\kappa}}(\Pcal_0)=n_{z_{\kappa+1}}(\Pcal_0)=n_{z_{\kappa+2}}(\Pcal_0).$$
Thus $\Pcal_0$ may be viewed as a periodic domain associated with the Heegaard diagram
$(\Sig,\alphas,\betas,\z)$. If $n_{z_{\kappa}}(\Pcal_0)=d$ then
\begin{align}
\nonumber
\big\langle c_1(\spinc), (i_{\ovl X})_*H(\Pcal_0)\big\rangle
&=\big\langle c_1(i_{\ovl X}^*(\spinc)),H(\Pcal_0)\big\rangle+2d.
\nonumber
\end{align}
 From here we may conclude
 \begin{displaymath}
 \begin{split}
 &\big\langle c_1(\spinc), (i_{\ovl X})_*H(\Pcal_0)\big\rangle=-2(a+b)\\
 \Rightarrow\ \ &\big\langle c_1(i_{\ovl X}^*(\spinc)),H(\Pcal_0)\big\rangle
 =-2(a+b+d)=-n_{z_{\kappa+1}}(\Pcal_0)\leq 0.
 \end{split}
 \end{displaymath}
Since the Heegaard diagram $(\Sig,\alphas,\betas,\z)$ is $i_{\ovl X}^*(\spinc)$-admissible
in the stronger sense of remark~\ref{remark:admissibility}, we conclude that
$\la(\Pcal_0)=0$ in $\Rin_\tau$. The condition that $z_{\kappa}$ is in the genus zero components
of $\Sig-\alphas$ and $\Sig-\betas$ then implies that $\la(\Pcal)=0$ in $\Rin_{\ov\tau}$.
This proves the $\spinc$-admissibility of the Heegaard diagram $\ov H$.
\\

Let us consider the Heegaard diagram $(\Sig,\ov{\alphas},\ov{\betas},\ov{\z})$ as the connected sum $$\left(\Sig,\alphas,\betas,\z\cup\big\{w\big\}-\big\{z_{\kappa}\big\}\right)
\#\left(S,\alpha_{\ell+1},\beta_{\ell+1},
\big\{v,z_{\kappa},z_{\kappa+1},z_{\kappa+2}\big\}\right),$$
where $S$ is a sphere, $w$ and $v$ are the corresponding
connected sum  points such that $w$ is in the same domain
as $z_{\kappa}$ in $(\Sig,\alphas,\betas,\z)$. If
$\alpha_{\ell+1}\cap\beta_{\ell+1}=\{x,y\}$ then
$\T_{\ov{\alpha}}\cap\T_{\ov{\beta}}=(\Ta\cap\Tb)\times \{x,y\},$
and for any $\x\in\Ta\cap\Tb$ we have
$$\relspinc(\x)=\is\left(\relspinc\left(\x\times\big\{x\big\}\right)\right)
=\is\left(\relspinc\left(\x\times\big\{y\big\}\right)\right)+\PD[\gamma_{\el+1}].$$
Let $C_x$ and $C_y$ be the submodules of
$$\mathrm{CF}\left(\Sig,\ov{\alphas},\ov{\betas},\ov{\z},\spinc;\Ring_\tau[\la_{\el+1}]\right)
=\mathrm{CF}\left(\Sig,\ov{\alphas},\ov{\betas},\ov{\z},\spinc\right)
\otimes_{\Ring_{\ov\tau}}\Ring_\tau[\la_{\el+1}]$$
generated by the intersection points containing $x$ and $y$ respectively.
Thus we have a module splitting
$$\mathrm{CF}\left(\Sig,\ov{\alphas},\ov{\betas},\ov{\z},\spinc;\Ring_\tau[\la_{\el+1}]\right)
=C_x\oplus C_y.$$

First we consider the $C_x$-components of the differential of
the complex on the generators of $C_x$.
Let $\x\times\{x\}$ be a generator of
$C_x$ and $\phi\in\pi_2(\x\times\{x\},\y\times\{x\})$ be the
homology class of a Whitney disk with $\mu(\phi)=1$. We may thus
decompose $\phi$ as $\phi=\phi_1\#\phi_2$ where $\phi_1\in\pi_2(\x,\y)$ and
$\phi_2\in\pi_2(x,x)$. Theorem~\ref{thm:gluing} then implies that
$$\mu(\phi)=\mu(\phi_1)+\mu(\phi_2)-2k=\mu(\phi_1)+2n_{z_{\kappa+1}}(\phi_2),$$
where $k=n_{w}(\phi_1)=n_v(\phi_2)$.
If $\sM(\phi)\neq \emptyset$ for long enough neck-length,
then $\phi_2$  admits holomorphic representatives and
$\Dcal(\phi_2)\ge 0$.  This implies that
$$\mu(\phi_2)-2n_{v}(\phi_2)=2n_{z_{\kappa+1}}(\phi)\ge 0,$$ and that
the equality happens if and only if $n_{z_{\kappa+1}}(\phi)=0$.
If $\mu(\phi_2)-2n_{v}(\phi_2)>0$ then $\mu(\phi_1)\le-1$ and $\sM(\phi_1)$
is generically empty.  Thus $n_{z_{\kappa+1}}(\phi_2)$ should be zero
and $\mu(\phi_2)=2n_{v}(\phi_2)=2k$.
Theorem~\ref{thm:gluing} then guarantees  that for a sufficiently large
connected sum length, we have an identification of $\sM(\phi)$.
\begin{displaymath}
\begin{split}
\sM(\phi)&=
\sM(\phi_1)\times_{\mathrm{Sym}^k(\D)}\sM(\phi_2)\\
&=\Big\{u_1\times u_2\in\sM(\phi_1)\times\sM(\phi_2)\ \big|\ \rho^{w}(u_1)=\rho^{v}(u_2)\Big\}\\
\Rightarrow\ \
\#\widehat{\sM}(\phi)&=\sum_{u_1\in\widehat{\sM}(\phi_1)}\#\Big\{u_2\in\sM(\phi_2)\ \big|\
\rho^{w}(u_1)=\rho^{v}(u_2)\Big\}.
\end{split}
\end{displaymath}
The coefficient of $\y\times \{x\}$ in the expression $\partial(\x\times\{x\})$ in $\mathrm{CF}\left(\Sig,\ov{\alphas},\ov{\betas},\ov{\z},\spinc;\Ring_\tau[\la_{\el+1}]\right)$
is thus equal to
\begin{displaymath}
\begin{split}
\sum_{\substack{\phi_1\in\pi_2^1(\x,\y)\\ \phi_2\in\pi_2(x,x)\\ n_{z_{\kappa+1}}(\phi_2)=0}}
\sum_{u_1\in\widehat{\sM}(\phi_1)}\epsilon(u_1)
\left(\prod_{i=1}^{\kappa}\la_{i}^{n_{z_i}(\phi_1)}\right)\#\Big\{u_2\in\sM(\phi_2)\
\big|\ \rho^{w}(u_1)=\rho^{v}(u_2)\Big\},
\end{split}
\end{displaymath}
where $\epsilon(u_1)$ denotes the sign associated with $u_1\in\ov\Mod(\phi_1)$
via a coherent system of orientations for the Heegaard diagram (which is suppressed from
the notation).
Lemma~\ref{lem:sphere} may be used to compute the interior sum. The total value of the
above sum is thus equal to
$$\sum_{\substack{\phi_1\in\pi_2^1(\x,\y)}}
\sum_{u_1\in\widehat{\sM}(\phi_1)}\epsilon(u_1)
\left(\prod_{i=1}^{\kappa}\la_{i}^{n_{z_i}(\phi_1)}\right),$$
which is the coefficient of $\y$ in $\partial\x$ in
$\mathrm{CF}(\Sig,\alphas,\betas,\z,\spinc)\otimes_{\Ring_\tau}\Ring_\tau[\la_{\kappa+1}]$.\\

With the same argument, the $C_y$-component of the differential of the generators in
$C_y$ is identified with differential of
$C_y=\mathrm{CF}(\Sig,\alphas,\betas,\z,\spinc)\otimes_{\Ring_\tau}\Ring_\tau[\la_{\kappa+1}]$.\\

We now consider the $C_x$-component of the differential of a generator in $C_y$.
For any $\phi\in\pi_2(\x\times\{y\},\y\times\{x\})$ we can
write $\phi=\phi_1\#\phi_2$. By theorem~\ref{thm:gluing} if $\mu(\phi)=1$
then $\mu(\phi_1)+\mu(\phi_2)-2n_{v}(\phi_2)=1$.
By Lipshitz' Index formula
we have $$\mu(\phi_2)=2n_v(\phi_2)+2n_{z_{\kappa+1}}(\phi_2)+1.$$
This implies
that $\mu(\phi_2)-2n_v(\phi_2)\ge 1,$ and that the equality holds if and only if
$n_{z_{\kappa+1}}(\phi_2)=0$.  Thus this last equality should be satisfied
 and $\mu(\phi_1)=0$.
Hence $\phi_1$ is constant, $\mu(\phi_2)=1$ and $n_v(\phi_2)=0$.
These conditions imply that the possible domains for $\phi_2$ are
two different bi-gons in $S$ connecting $y$ to $x$, which contain $z_{\kappa}$ and
$z_{\kappa+2}$ respectively. For either of these bi-gons
$\widehat{\sM}(\phi_2)$ consists of one element, while the orientation
assignment for these two bi-gons is different.
Using the relation $\ov f(\la_\el)=\ov f(\la_{\el+2})$ in $\Ring[\la_{\el+1}]$,
we may thus conclude that the coefficient of $\y\times\{x\}$
in $\partial \left(\x\times \{y\}\right)$
is zero, i.e. the $C_x$-component of the differential of the generators
in $C_y$ is trivial.\\

Finally, we consider the $C_y$-component of the differential of a
generator in $C_x$. Again, degenerate
$\phi\in\pi_2(\x\times\{x\},\y\times\{y\})$ with $\mu(\phi)=1$
and $\la(\phi)\neq 0$ as the connected sum $\phi=\phi_1\#\phi_2$.
We thus have $\mu(\phi_1)+\mu(\phi_2)-2n_{v}(\phi_2)=1$, implying
$\mu(\phi_2)-2n_{v}(\phi_2)\le1$. By Lipshitz' index  formula  we have
$$\mu(\phi_2)=2n_{v}(\phi_2)+2n_{z_{\kappa+1}}(\phi_2)-1,$$
which implies that $\mu(\phi_2)$ is an odd number and
$\mu(\phi_2)-2n_{v}(\phi_2)\ge-1$. Thus $\mu(\phi_2)-2n_{v}(\phi_2)$
is equal to 1 or -1.\\

If $\mu(\phi_2)-2n_{v}(\phi_2)=1$ then $\mu(\phi_1)=0$.
Thus $\phi_1$ is constant and $\Dcal(\phi)$ is the bi-gon
containing $z_{\kappa+1}$. In this case $\#\widehat{\sM}(\phi)=1$.
Thus the corresponding component of the differential, as a map from
$C_x$ to $C_y$, is given by
\begin{displaymath}
\begin{split}
&\partial_{xy}^1:C_x\lra C_y\\
&\partial_{xy}^1(\x\times\{x\}):=\la_{\kappa+1}(\x\times\{y\}).
\end{split}
\end{displaymath}

The second possibility is the case where $\mu(\phi_2)-2n_{v}(\phi_2)=-1$.
If $\mu(\phi_2)=n_{v}(\phi_2)=1$ then $\mu(\phi_1)=2$ and $n_{w}(\phi_1)=1$.
If furthermore $\Rr(\tau)$ has at least one component with nonzero genus
then by theorem~\ref{thm:gluing} for sufficiently large connected sum
length $\sM(\phi)$ is identified with
\begin{align}
\nonumber
\sM(\phi_1)\times_{\D}\sM(\phi_2)
&=\Big\{u_1\times u_2\in\sM(\phi_1)\times\sM(\phi_2)\ \big|\
\rho^{w}(u_1)=\rho^{v}(u_2)\Big\}\\\nonumber
&=\Big\{u_1\times u_2\in\sM(\phi_1)\times\sM(\phi_2)\ \big|\
\rho^{w}(u_1)=u_2^{-1}(v)\Big\}.
\end{align}
Now $\mu(\phi_2)=n_v(\phi_2)=1$ implies that the domain of $\phi_2$
is the bi-gon in $S$ containing $v$ and thus it has a unique holomorphic
representative up to translation. We can fix a holomorphic representative
$u_2$ such that $u_2^{-1}(v)=(t,0)$. Using an appropriate
system of coherent orientations we thus have
\begin{align}
\nonumber
\#\widehat{\sM}(\phi)
&=-\#\Big\{u_1\in\sM(\phi_1)\ \big|\ \rho^{w}(u_1)=(t,0)\Big\}\\\nonumber
&=-\#\sM(\phi_1,t)
\end{align}
Let us now assume that the point $v$ is chosen very close to the curve $\beta$.
By lemma~\ref{lem:degeneration}, for $t$ sufficiently large  $\sM(\phi_1,t)$ is nonempty
if and only if $\phi_1\in\pi_2^{\beta}(\x)$ is the class of a $\beta$ boundary degeneration.
If furthermore $l>1$ then $\#\sM(\phi_1,t)=1$. Thus this case contributes to the
$C_y$-component of the restriction of the differential to $C_x$ via a map
\begin{displaymath}
\begin{split}
&\partial_{xy}^2:C_x\lra C_y\\
&\partial^2_{xy}\left(\x\times\{x\}\right)
=-\left(\prod_{\substack{\gamma_{\el}\neq\gamma_i\in\partial R_{l}^+}}
\la_{i}\right).\left(\x\times\{y\}\right).
\end{split}
\end{displaymath}
Similarly, if $l=1$ then $\partial^2_{xy}(\x\times\{x\})=0$.\\

To deal with the other terms corresponding the homotopy
classes $\phi$ with $n_v(\phi_2)>1$ we define a one parameter
family of connected sum points $v(r)$ on $S$ such that when
$r$ goes to infinity, $v(r)$ tends towards a point $v_{\infty}$
on the curve $\beta$. \\

Let $\sM_{r}(\phi)$ be the moduli space of holomorphic
representations of $\phi$ when we used the connected
sum point $v(r)$ in $S$. Assume that  for a sequence
$\{r_i\}$ converges to infinity, the moduli space
$\sM_{r_i}(\phi)\neq \emptyset$ for all choice of connected
sum length. For sufficiently large connected sum length the
moduli space $\sM_{r_i}(\phi)$ is identified with the fibered
product $\sM(\phi_1)\times_{\Sym^{k}(\D)}\sM(\phi_2)$.
Consider a sequence $u_1^{i}\times u_2^i$ in the fibered product.
Let $\ovl{u}_1^{\infty}$ and $\ovl{u}_2^{\infty}$ be Gromov
limits of $\{u_1^i\}$ and $\{u_2^i\}$.
The assumption $\mu(\phi_1)=2$ implies that there are
three possible types for the limit $\ovl{u}_1^{\infty}$.
The limit can be a holomorphic disk or a singly broken
flow line or it can contain a boundary degeneration.
If it contain a boundary degeneration, $\la(\phi)\neq 0$
implies the remaining component has Maslov index zero
and it should be constant. Thus $k=1$ and this
situation is already considered in the previous case.\\

If $\ovl{u}_1^{\infty}$ is not a broken flow line and
it is a holomorphic disk, $\ovl{u}_2^{\infty}$ has a
component $u_2^{\infty}$ such that
$\rho^{w}(\ovl{u}_1^{\infty})=\rho^{v}(u_2^{\infty})$.
Since $v(r_i)$ tends toward $v_{\infty}$ on $\beta$,
$\rho^{v}(u_2^{\infty})$ includes some points on $\{0\}\times\R$.
Thus for large $i$, $\rho^w(u_1^{i})$ contains points
sufficiently close to $\{0\}\times\R$. By
lemma~\ref{lem:degeneration} the holomorphic curve
$u_1^i$  should be a boundary
degeneration for $i$ sufficiently large. This implies that $k=1$
and again, we are within the cases considered earlier, and there
is no new contribution to the $C_y$-component of the
restriction of the differential to $C_x$ from this case.\\

Finally, if $\ovl{u}_1^{\infty}$ is a broken
flow line i.e. it is of the form $\ovl{u}_1^{\infty}=a\star b$
and $\mu(a)=\mu(b)=1$, then  $\ovl{u}_2^{\infty}$ degenerates,
correspondingly, as  $\ovl{u}_2^{\infty}=a'\star b'$. The
Maslov index of $\phi_2$ is odd, thus one of $a'$ and $b'$
has odd Maslov index. Let us assume that $\mu(a')$ is odd. Then
$(a')^{-1}(v^{\infty})$ contains some points on $\{0\}\times\R$.
 If $a$ is the holomorphic representative
of a homology class $\phi_1'$ of a Whitney disk, then for $r$ sufficiently
large $\sM(\phi'_1)$ includes holomorphic representatives $a^{r}$ such
that $\rho^{w}(a^{r})$ contains points of distance less than $1/r$ to
$\{0\}\times\R$. Since  $\mu(\phi_1')=1$,
$\phi_1'$ has finitely many
holomorphic representative up to translation. Thus for
any holomorphic representative $u$ of $\phi_1'$,
$\rho^{w}(u)$ does not include points arbitrary
close to $\{0\}\times\R$, since $w$ is not on $\betas$.\\

Gathering the above considerations, we observe that if
either of the two assumptions in the second part of theorem~\ref{thm:gluing}
is satisfied the $C_y$ component of the restriction of the differential
to $C_x$ is given by the map
\begin{displaymath}
\begin{split}
\partial_{xy}&:C_x\lra C_y\\
\partial_{xy}\left(\x\times\{x\}\right)
&=\partial^1_{xy}\left(\x\times\{x\}\right)+\partial^2_{xy}\left(\x\times\{x\}\right)\\
&=\la_{\kappa+1}(\x\times\{y\})-\left(\prod_{\substack{\gamma_{\el}\neq\gamma_i\in\partial R_{l}^+}}
\la_{i}\right).\left(\x\times\{y\}\right)\\
&=\left(\la_{\kappa+1}-\la\right)(\x\times\{y\}).
\end{split}
\end{displaymath}

The proof in the case where all the components of
$\Rr(\tau)$ have genus zero and $\ell_2=g(\Sig)$ is
exactly the same as the last part of the proof of
proposition 6.5 in \cite{OS-linkinvariant}. This completes
the proof of proposition~\ref{prop:stabilization}.
\end{proof}

\newpage
\section{A triangle associated with surgery}\label{sec:exact-triangle}
 \subsection{The triangle associated with the surgery Heegaard quadruple}
Let us assume that
$$H=\Big(\Sig,\alphas=\{\alpha_1,...,\alpha_\ell\},\betas=\{\beta_1,...,
\beta_\ell\},\z_0=\{w_1,w_2,z_3,...,z_\el\}\Big)$$
is a Heegaard diagram for the sutured manifold $(X,\tau)$ {\emph{compatible}} with the
surgery $(X,\tau)\rightsquigarrow (X,\tau_\lambda)$, where the two first marked points
$w_1$ and $w_2$ correspond to the sutures $\gamma_1$ and $\gamma_2$ which are cut by
the surgery simple closed curve $\lambda$.
More precisely, assume that the
marked points $w_1$ and $w_2$
are both placed very close to a  point  on $\beta_\ell$, such that
$\beta_\ell$ separates them from each other.
We will assume that
$$\Sig-\alphas=\coprod_{i=1}^k A_i,\ \ \&\ \ \Sig-\betas=\coprod_{j=1}^l B_j$$
are the connected components of  $\Sig-\alphas$ and
$\Sig-\betas$, respectively.
We may assume that $A_i$ corresponds to $R_i^-\subset \Rr^-(\tau)$ and
that $B_j$ corresponds to $R_j^+\subset \Rr^+(\tau)$ for
$i=1,...,k$ and $j=1,...,l$.
We will also assume that $w_1,w_2\in A_1\cap B_1$, that $A_1$ and $B_1$ are both
surfaces of genus zero, and that the only marked points in $B_1$ are $w_1$ and $w_2$.
All these assumptions may be achieved if $\gamma_1$ and $\gamma_2$ are on the
boundary of $R_1^-$ and $R_1^+$, where both these components have trivial genus and
$R_1^+$ is a cylinder (with $\gamma_1$ and $\gamma_2$ its boundary components).
We will  assume that $z_i$ corresponds to $\gamma_i$ for
$3\leq i\leq \el$.
The surgery $(X,\tau)\rightsquigarrow (X,\tau_\lambda)$ is determined by
 a simple closed curve $\lambda$   on $\partial X$ which
cuts $\gamma_1$ and $\gamma_2$ is a pair of transverse intersection points
and remains disjoint from the curves in $\{\gamma_3,...,\gamma_\el\}$.
This curve determines a simple closed curve on $\Sig$, still denoted by $\lambda$,
which contains $w_1$ and $w_2$ as close-by points. We may assume that  the only intersection point of
this curve with the curves in $\betas$ is at the point $p\in\beta_\ell$,
located in the middle of the shorter arc on $\lambda$ with end points $w_1$ and $w_2$.
Moreover, we may assume that
$\lambda$ is completely included in the closure $\ovl A_1$ of the sub-surface
$A_1$ of $\Sig$.
Such a Heegaard diagram is called compatible with the
surgery $(X,\tau)\rightsquigarrow (X,\tau_\lambda)$.
\\

Suppose that $\lambda$ is as above.
A neighborhood $N(\Lambda)$ of the union of curves
$$\Lambda=\beta_\ell\cup \lambda\subset \Sig$$
may be identified with the complement of a very
 ball, or in fact a single point,
in the standard torus $T=S^1\times S^1$, which will be denoted by $T^\circ$.
The covering of $T$ by $\R^2$ gives a covering of $T^\circ $ by $\R^2\setminus \Z^2$.
Any line with rational slope and not passing through the lattice points in $\Z^2$
gives a simple closed curve in $T^\circ$, and thus in the neighborhood $N(\Lambda)$
of $\Lambda$. If the slope of the line is $p/q\in\Q$, we may denote this
simple closed curve by $\mu(p/q)$. Thus, $\mu(p/q)$ is homologous to
$p\beta_\ell+q\lambda$. Any such simple closed curve is called a surgery
curve. If $p/q$ is an integer, $\mu(p/q)$ determines a surgery on $(X,\tau)$
of the type discussed in section~\ref{sec:surgery}. More generally,
for any rational number $p/q\in\Q$, the simple closed  curve
$\mu(p/q)$ determines a sutured manifold,
which plays the role of Morse surgeries on knots inside closed three-manifolds.
Such curves will be called surgery curves in short.
\\

Let us assume that $\mu_0,\mu_1$ and $\mu_2$ are three surgery curves for the
Heegaard diagram $H$. Let us assume
$$m_0=\#\big(\mu_1\cap\mu_2\big),\ \ m_1=\#\big(\mu_2\cap\mu_0\big)
,\ \ \&\ \ m_2=\#\big(\mu_0\cap\mu_1\big)$$
are the algebraic intersection numbers of these curves. We will assume that
the three curves (and the corresponding rational surgery coefficients) are
chosen so that $m_0,m_1$ and $m_2$ are all positive integers.
Let us assume
\begin{displaymath}
\begin{split}
\p_0:=\mu_1\cap\mu_2&=\big\{p_0=p_0^1,...,p_0^{m_0}\big\},\ \
\p_1:=\mu_2\cap\mu_0=\big\{p_1=p_1^1,...,p_1^{m_1}\big\},\\
&\ \ \&\ \ \p_2:=\mu_0\cap\mu_1=\big\{p_2=p_2^1,...,p_2^{m_2}\big\}.
\end{split}
\end{displaymath}
We assume that the order of appearance of the points on the first
curve in any of the above intersections is the same as the
cyclic order determined by the indices. Furthermore, suppose  that
the three marked points $p_0,p_1$ and $p_2$
are the vertices of a very
small triangle $\Delta$, which is one of the connected components in
$$\Sig-\alphas-\big\{\beta_1,...,\beta_{\ell-1}\big\}
-\big\{\mu_0,\mu_1,\mu_2\big\}.$$
We may assume that the intersection point $p$ is included in $\Delta$ as an interior marked point.\\

We may choose the marked points $z_0,z_1$ and $z_2$ outside $\Delta$ and very close to its edges,
so that $z_0$ is close to the edge $e_0$ connecting $p_1$ to $p_2$, $z_1$ is close to the edge
$e_1$ connecting $p_2$ to $p_0$, and $z_2$ is close to the edge $e_2$ connecting $p_0$ to $p_1$.
The notation is illustrated in figure~\ref{fig:HD-012}.
We will denote  by $\z$ 
the following set of marked points
$$\z=\big\{z_0,z_1,...,z_\el,p\big\}.$$

\begin{figure}[ht]
\def\svgwidth{10cm}
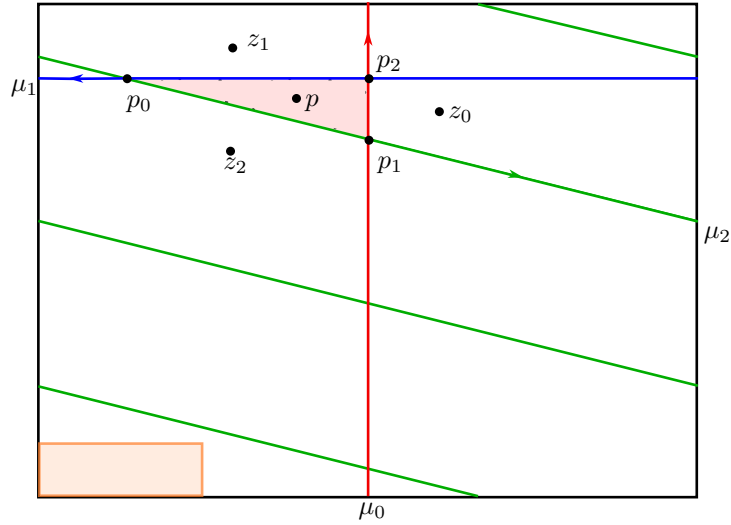
\caption{A neighborhood of the curves $\mu_0,\mu_1$ and $\mu_2$ is illustrated.
One should take the connected sum of the torus obtained by identifying the opposite
edges of the rectangle, with another Riemann surface to obtain the Heegaard surface
$\Sig$. The attaching circle of the connected sum tube lands in the shaded
area in the lower left corner of the figure. The $\alpha$ curves live
close to the boundary of the rectangle, or on the Riemann surface
which is attached to this torus.
The marked points $\{p,z_0,z_1,z_2\}$ and the intersection points $p_0,p_1$ and
$p_2$ are illustrated. For this picture $m_0=m_2=1$ while $m_1=3$.
}\label{fig:HD-012}
\end{figure}

Consider the Heegaard diagrams
\begin{displaymath}
\begin{split}
H_{i}=\Big(\Sig,\alphas,\betas^{i}&=\{\beta_1^i,...,\beta_{\ell-1}^i,\mu_i\},\z\Big),
\ \ i\in\{0,1,2\},
\end{split}
\end{displaymath}
where we assume that $\beta_j^i$  are small Hamiltonian isotopes of the curve
$\beta_j$, for $i=0,1,2$, so that any pair of curves in $\{\beta_j^0,\beta_j^1,\beta_j^2\}$
intersect each-other
in a pair of transverse canceling intersection points for $j=1,...,\ell-1$.
We would like to study the Heegaard quadruple (with $\el+2$ marked points)
$$R=(\Sig,\alphas,\betas^0,\betas^1,\betas^2,\z)$$
in this section and construct a triangle of chain complexes associated with it,
which generalizes the exact triangles associated with surgery on knots in the
context of Heegaard Floer theory of closed three-manifolds and knots inside them.\\

The Heegaard diagram $(\Sig,\alphas,\betas^i,\z)$, for $i=0,1,2$ determines a sutured manifold
which will be denoted by $(Y,\vsi^i)$.
In fact, instead of gluing a disk to $\mu_i$, one may fill out the suture $\gamma_p$
corresponding to the marked point $p$ and obtain the same three manifold $Y$.
The identification is illustrated in figure~\ref{fig:filling-suture-vs-curve}.
Thus the three-manifold $Y$ does not depend on $i$,
while the sutured manifold structure $\vsi^i$ is determined by $i$.
The above identification of the three manifold $Y$
gives an identification of the
spaces of relative $\SpinC$ structures as well:
\begin{equation}\label{eq:SpinC(Y)}
\RelSpinC(Y,\vsi):=\RelSpinC(Y,\vsi^0)=\RelSpinC(Y,\vsi^1)=\RelSpinC(Y,\vsi^2).
\end{equation}
Note that $\RelSpinC(Y,\vsi)$ is just a notation we use for this common
space of relative $\SpinC$ structures, and that the identification of the
$\SpinC$ spaces is not natural.\\

\begin{figure}[ht]
\def\svgwidth{8.67cm}
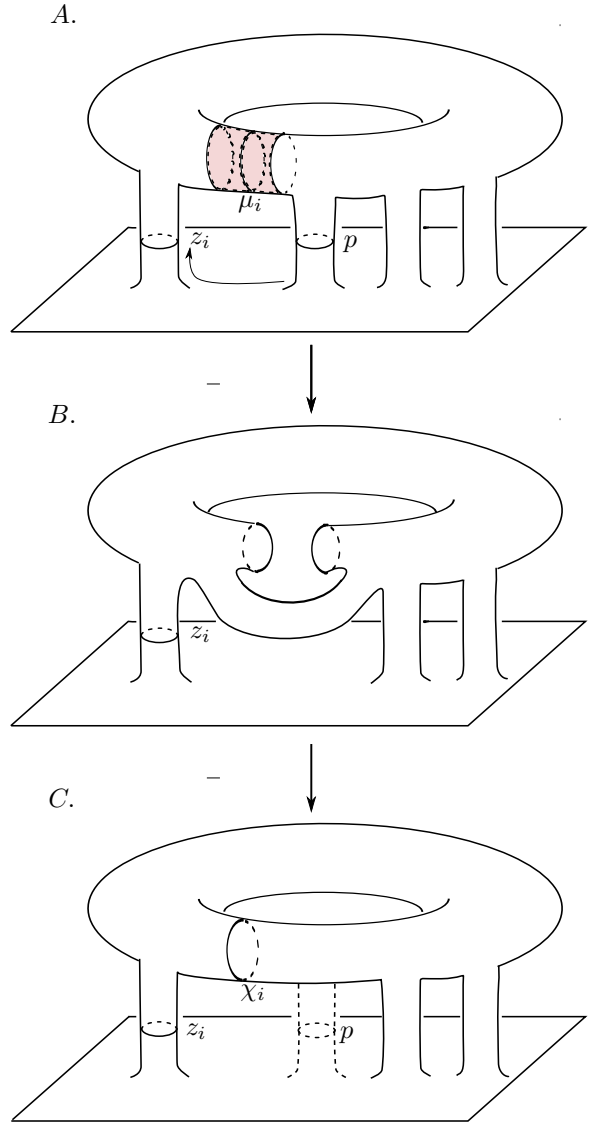
\caption{ Instead of attaching a $2$-handle along the curve $\mu_i$,
one may fill in the suture which corresponds to the marked point $p$.
In part A of the figure, a $2$=handle is attached to $\mu_i$ (think of
$X$ as the three manifold outside the torus and above the plane illustrated in this picture).
Then we may slide the $1$-handle corresponding to $p$ over the $1$-handle
corresponding to $z_i$, as illustrated in part B. The result, after smoothing
the appropriate corners, is the picture illustrated in part C, in which
the suture corresponding to $p$ is filled out and instead, no $2$-handle is attached
to $\mu_i$.}\label{fig:filling-suture-vs-curve}
\end{figure}

For $i$ a cyclic index
in $\frac{\Z}{3\Z}=\{0,1,2\}$, let $(X,\tau^i)$ be the sutured manifold obtained from
$(Y,\vsi^i)$ by filling out the sutures corresponding to $z_{i+1}$ and $z_{i+2}$ (note that
we are taking indices modulo $3$).
The intersection points $p_i\in\mu_{i+1}\cap\mu_{i+2}$, for $i
\in \frac{\Z}{3\Z}=\{0,1,2\}$
determine an identification of three subsets
$$\ovl\Ss_i\subset \SpinC(\ovl{Y}^{\vsi^i})=\SpinC(\ovl{X}^{\tau^i}),\ \ i\in\{0,1,2\}.$$
Let us denote these identified subsets by $\SpinC(\ovl X)$. Again in a sense, we are
abusing the language with this notation, at least when $\la$ is not null-homologous.
We will fix a class $\spinc\in\SpinC(\ovl X)$ for the rest of this section.
We will assume that the Heegaard diagrams $H_i,\ i=0,1,2$ are
$\spinc$-admissible. When $\la$ is null-homologous this is guaranteed if
the Heegaard diagram $H$ is $\spinc$-admissible.
In fact, we will drop the admissibility issues, as well as orientation issues,
from our discussion in the remainder of this section. Taking care of these issues is
completely straight forward, and follows the lines of the arguments given in the earlier
sections.
\\

The algebra $\Ring_{\vsi^i}$ is independent of $i$, and will  be denoted by
$\ov\Ring$.
Let us denote the generator
corresponding to the marked point $z_j$ by $\la_j$, for $j=0,1,...,\el$.
The generator associated with the marked point $p$ will be denoted by $\la_p$.
We would like to consider the following quotient of $\ov\Ring$:
$$\Ring=\frac{\ov\Ring}{\big\langle\la_p-(\la_0^{m_0-1}\la_1^{m_1-1}\la_2^{m_2-1})
\big\rangle_{\ov\Ring}}.$$
Furthermore, let us denote the generator associated with the marked point $z_j$
in $\Ring_{\tau^i}$ by $\zeta_j$ for $j=3,...,\el$, and denote the
generator associated with $w_1$ and $w_2$ by $\zeta_1$ and $\zeta_2$ respectively.
Note that
$\Ring_{i}=\Ring_{\tau^i}$ does not depend on $i\in\{0,1,2\}$, either.
Each Heegaard diagram $H_i$ determines an embedding of $\Ring_{i}$ in $\Ring$.
More precisely, we may define
\begin{displaymath}
\imath^i:\Ring_{i}\ra \Ring,\ \ \
\imath^i(\zeta_j)
=\begin{cases}
\la_i\ \ \ &\text{if }j=1\\
\frac{\la_p\la_0\la_1\la_2}{\la_i}\ \ \ &\text{if }j=2\\
\la_j\ \ \ &\text{if }3\leq j\leq \el\\
\end{cases}.
\end{displaymath}
 As mentioned earlier, $\Ring_0,\Ring_1$ and
$\Ring_2$ are isomorphic. However, the index is used to distinguish them as sub-rings of
$\Ring$, using the embedding $\imath^i:\Ring_i\ra \Ring$.\\

Let $\betas_0$ denote the set of curves $\{\beta_1,...,\beta_{\ell-1}\}$,
$\D_\alpha$ denote a set of $\ell$ copies of $D^2\times [-\epsilon,\epsilon]$
(for some small positive real number $\epsilon$) corresponding to the curves in $\alphas$
and $\D_{\beta_0}$ denote a set of $\ell-1$ copies of $D^2\times [-\epsilon,\epsilon]$
 corresponding to curves in $\betas_0$.
 Denote small tubular neighborhoods of the curves in $\alphas$ and the curves in $\betas_0$ by
 $\mathrm{nd}(\alphas)$ and $\mathrm{nd}(\betas_0)$, respectively. These neighborhoods
 may be identified with subsets of $\partial \D_\alpha$
 and $\partial \D_{\beta_0}$ respectively.
Under the identification of $Y$ with the three-manifold
$$\Big([0,1]\times\Sig\setminus (\z-\{p\}) \Big)
\bigcup_{\mathrm{nd}(\alpha)\times\{0\}}\D_\alpha\bigcup_{\mathrm{nd}(\beta_0)\times\{1\}}\D_{\beta_0},$$
each marked point $z_j$ determines an oriented simple closed curve on the boundary of $Y$. The Poincar\'e
dual of this curve determines an element $\chi_j\in\Ht^2(Y,\partial Y;\Z)$
for $3\leq j\leq \ell$. For $i\in\{0,1,2\}$ we will denote the element of
$\Ht^2(Y,\partial Y;\Z)$ corresponding to the marked point $z_i$ by $\eta_i$.
The assumptions on the Heegaard diagram imply that
$$\eta_0+\eta_1+\eta_2=0.$$
The Poincar\'e duals of the curves corresponding to the marked point $z_j$ and $p$
in  $(Y,\vsi^i)$ will be denoted by $\chi(i,j)$ and $\chi(i,p)$ respectively,
for $i=0,1,2$ and $0\leq j\leq \ell$. Furthermore, let $\chi_i$ denote
the Poincar\'e dual $\PD[\mu_i]$ of the simple closed curve $\mu_i\subset \partial Y$ for $i=0,1,2$.
One may check that
\begin{displaymath}
\begin{split}
&\chi(i,j)=\begin{cases}
\chi_j\ \ \ &\text{if }j\neq i\ \&\ j\in\{3,...,\el\}\\
\eta_j\ \ \ &\text{if }j\neq i\ \&\ j\in\{0,1,2\}\\
\chi_i+\eta_i\ \ \ &\text{if } j=i
\end{cases}\\
&\chi(i,p)=-\chi_i.
\end{split}
\end{displaymath}

Associated with any of the Heegaard diagrams $H_i,\ i=0,1,2$ (and independent of $i$)
we define a filtration map
\begin{displaymath}
\begin{split}
&\chi:G(\Ring)\lra \Ht^2(Y,\partial Y;\Z),\\
&\chi(\la_j):=\begin{cases}
\chi_j\ \ \ &\text{if }j\in\big\{0,1,...,\el\big\}\\
-(\chi_0+\chi_1+\chi_2)\ \ \ &\text{if } j=p
\end{cases}.
\end{split}
\end{displaymath}
Note that with this assignment we may compute
$$\chi\circ \imath^i\Big(\prod_{j=1}^\el\zeta_j^{i_j}\Big)
=(i_1-i_2)\chi_i+\sum_{j=3}^\el i_j\chi_j,\ \ \forall\ i\in\big\{0,1,2\big\}.$$
This implies that the filtration $\chi:G(\Ring)\ra \Ht^2(Y,\partial Y;\Z)$
is compatible with the filtration of $G(\Ring_i)$ by $\Ht^2(X,\partial X;\Z)$. \\

Let us consider the following quotient of $\RelSpinC(Y,\vsi)$:
$$\Ss=\RelSpinC(X,\tau):=
\frac{\RelSpinC(Y,\vsi)}{\big\langle\eta_0,\eta_1,\eta_2\big\rangle_\Z}
={\SpinC(X,\tau)}.$$
The last equality follows since $\RelSpinC(X,\tau^i)$ is obtained from
$\RelSpinC(Y,\vsi^i)$ by setting trivial the sutures corresponding to
$z_j$ with $j\in\{0,1,2\}-\{i\}$. Thus, it is equal to the quotient of
$\RelSpinC(Y,\vsi^i)$ by the action of $\eta_j$, with $j\in\{0,1,2\}-\{i\}$.
Since we have $\eta_0+\eta_1+\eta_2=0$, this means that  $\eta_0,\eta_1,$ and $\eta_2$
act trivially on $\RelSpinC(X,\tau^i)$.\\

Correspondingly, consider the following $\Z$ module associated with
the three-manifold $Y$:
$$\Hbb:=\frac{\Ht^2(Y,\partial Y;\Z)}{\big\langle\eta_0,\eta_1,\eta_2\big\rangle_\Z}
={\Ht^2(X,\partial X;\Z)}.$$
We continue to denote the image of $\chi_i\in\Ht^2(Y,\partial Y;\Z)$ in $\Hbb$ by $\chi_i$.
Clearly, $\Hbb$ acts on $\Ss$.
 From the definition, we also have a natural quotient map
$$\RelSpinC(Y,\vsi)\lra\RelSpinC(X,\tau)=\Ss$$
which will  be denoted by $(.)$ (thus, this
map sends a relative $\SpinC$ class $\relspinc$ to its class
$(\relspinc)\in\Ss$.\\

The filtration map $\chi:G(\Ring)\ra \Ht^2(Y,\partial Y;\Z)$ may be composed with
the quotient map from $\Ht^2(Y,\partial Y;\Z)$ to $\Hbb$ to define a new filtration map,
yet denoted by $\chi$.
\\

For $i=0,1,2$, consider the filtered $\tab$ chain complex
\begin{displaymath}
\begin{split}
\E_i(\spinc):&=
\CFT(Y,\vsi^i)\otimes_{\ov \Ring}\Ring=
\CFT(X,\tau^i,\spinc)\otimes_{\Ring_{\tau^i}}\Ring\\
&=\Big\langle  \x\ \big|\ \x\in\Ta\cap\mathbb{T}_{\beta^i},\ \
 \&\ \relspinc_\z(\x)\in\spinc\Big\rangle_\Ring.
\end{split}
\end{displaymath}
The set of marked points $\z=\{z_0,...,z_\el,p\}$
defines a map
\begin{displaymath}
\begin{split}
&\la_\z:\coprod_{i=0}^3\ \ \coprod_{\x,\y\in\Ta\cap\mathbb{T}_{\beta^i}}\pi_2(\x,\y)\lra G(\Ring)\\
&\la_\z(\phi):=\la_p^{n_p(\phi)}.\prod_{i=0}^\el \la_i^{n_i(\phi)}
=\la_p^{n_p(\phi)}\la_0^{n_{z_0}(\phi)}\la_1^{n_{z_1}(\phi)}...\la_\el^{n_{z_\el}(\phi)}.\\
\end{split}
\end{displaymath}
The differential $\partial_i$ of the complex $\E_i(\spinc)$,
as an $\Ring$-module homomorphism, is  defined by
\begin{displaymath}
\begin{split}
\partial_{i}\big(\x\big):=
\sum_{\y\in\Ta\cap\mathbb{T}_{\beta^i}}
\sum_{\substack{\phi\in\pi_2^1(\x,\y)}}
\big(\m(\phi)\la_\z(\phi)\big)\y.
\end{split}
\end{displaymath}
Here $n_i(\phi)$ denotes the intersection number $n_{z_i}(\phi)$, and $\pi_2^1(\x,\y)\subset \pi_2(\x,\y)$
consists of all homotopy types $\phi$ of Whitney disks connecting $\x$ to $\y$ such that $\mu(\phi)=1$.
We  define a map from the set of generators of $\E_i(\spinc)$ to
$\Ss$ by setting
\begin{displaymath}
\begin{split}
&\relspinc^i: G(\Ring)\times (\Ta\cap\mathbb{T}_{\beta^i})
\lra \Ss\\
&\relspinc^i(\la \x):=(\relspinc(\x))+\chi(\la).
\end{split}
\end{displaymath}
Abusing the notation, we will sometimes denote $\relspinc^i(\la\x)$ by
$\relspinc(\la\x)=\relspinc(\x)+\chi(\la)$, dropping the index $i$ and
the quotient map $(.):\RelSpinC(Y,\vsi)\ra \Ss$ from the notation.\\

\begin{lem}
 If a generator
$\la\y$, with $\la\in G(\Ring)$, appears with non-zero coefficient in $\partial_{i}\big(\x\big)$,
we will have $\relspinc^i(\x)=\relspinc^{i}(\la\y)$ in $\Ss$.
\end{lem}
\begin{proof}
Without loosing on generality, let us assume that $i=0$.
Suppose that $\relspinc(\x),\relspinc(\y)\in\spinc$, and that there is a Whitney disk
$\phi\in\pi_2(\x,\y)$ contributing to $\partial_{0}(\x)$ with $\la=\la_\z(\phi)$.
Then we will have $n_1(\phi)=n_2(\phi)=n_p(\phi)$.
The existence of this disk implies that
\begin{displaymath}
\begin{split}
\relspinc(\x)&=\relspinc(\y)+\big(n_0(\phi)-n_p(\phi)\big)\chi_0
+\sum_{j=3}^\el n_j(\phi).\chi_j\\
\Rightarrow\ \ \relspinc^0(\x)&=\relspinc^0(\y)+\chi\Big(\la_p^{n_p(\phi)}
\prod_{j=0}^\el\la_j^{n_j(\phi)}\Big)=\relspinc^0(\la\y).
\end{split}
\end{displaymath}
For the equality in the second line, we use the equality $\chi(\la_p)=-\left(\chi_0+\chi_1+\chi_2\right)$.
This completes the proof of the lemma.
\end{proof}

The above assignment of relative $\SpinC$ structures is thus
respected by the differential $\partial_{i}$ of $\E_i(\spinc)$, and
$\E_i(\spinc)$ is thus decomposed as
\begin{displaymath}
\E_i(\spinc)=\bigoplus_{\relspinc\in \spinc\subset
\Ss}\E_i(\relspinc).
\end{displaymath}

Associated with the $\SpinC$ class $\spinc$, we will describe a
triangle of chain maps
\begin{equation}\label{eq:hol-triangle}
\begin{diagram}
\E_0(\spinc)&&\rTo{\f^{\spinc}_2}&&
\E_1(\spinc)\\
&\luTo{\f_1^{\spinc}}&&\ldTo{\f_0^{\spinc}}&\\
&&\E_{2}(\spinc)&&
\end{diagram}
\end{equation}
such that the compositions $\f^{\spinc}_{i+1}\circ \f^{\spinc}_i$, $i\in\frac{\Z}{3\Z}=\{0,1,2\}$
 are chain homotopic to zero. \\

To define the chain map $\f_{i-1}^{\spinc}$, note that the special Heegaard diagram
$(\Sig,\betas^i,\betas^{i+1},\z)$  is admissible for
all the corresponding $\SpinC$ classes (c.f. the arguments of subsection~\ref{subsec:special-HD}).
We may thus compute
\begin{displaymath}
\begin{split}
\CFT\left(L_{i-1},\nu_{i-1};\Ring\right)&
=\CFT(\Sig,\betas^i,\betas^{i+1},\z)\otimes \Ring\\
\end{split}
\end{displaymath}
where $(L_{i-1},\nu_{i-1})$ is the sutured manifold corresponding to the special Heegaard diagram  $(\Sig,\betas^i,\betas^{i+1},\z)$.
The intersection point $p_{i-1}\in\mu_i\cap\mu_{i+1}$ determines a unique
$\SpinC$ class
$$\spinc_{i-1}\in\SpinC\left(\ovl{L_{i-1}}\right),\ \ \
c_1(\spinc_{i-1})=0,$$
as well as a top generator $\Theta_{i-1}$ corresponding to $\spinc_{i-1}$ (which is a closed element)
in the above Heegaard Floer complex. The generator $\Theta_{i-1}$ is obtained as the
union of $p_{i-1}$ and the positive intersection points of $\beta^i_j$ and $\beta^{i+1}_j$ for
$j=1,...,\ell-1$.
The generator $\Theta_{i-1}$ corresponds to a relative $\SpinC$ class which will be denoted by
$\relspinc_{i-1}\in\spinc_{i-1}$.
Consider the holomorphic triangle map
\begin{displaymath}
\begin{split}
f^\spinc_{i-1}:\E_{i}(\spinc)\otimes_\Ring{\mathrm{CF}}
(\Sig,\betas^i,\betas^{i+1},\z;\spinc_{i-1};\Ring)\lra
\E_{i+1}(\spinc).
\end{split}
\end{displaymath}
On a generator  $\x \otimes \q$ of the left hand side,
with $\x\in\Ta\cap\mathbb{T}_{\beta^i}$ and
$\q$ a generator corresponding to the
$\SpinC$ class $\spinc_{i-1}$,  $f^{\spinc}_{i-1}(\x\otimes\q)$ is defined by
\begin{equation}\label{eq:Phi-f}
\begin{split}
f^{\spinc}_{i-1}\big(\x\otimes \q\big):=
\sum_{\y\in \Ta\cap \Tb}\sum_{\substack{\Delta\in \pi_2^0(\x,\q,\y)\\ }}
\Big(\m(\Delta)\la_{\z}(\Delta)\Big)\y,
\end{split}
\end{equation}
where  $\pi_2^0(\x,\q,\y)$ denotes
the subset of $\pi_2(\x,\q,\y)$ consisting of the triangle classes
$\Delta$ such that $\mu(\Delta)=0$. The map
$f^\spinc_{i-1}$ is then extended, as an $\Ring$-module homomorphism, to all of
$\E_{i}(\spinc)\otimes_\Ring{\mathrm{CF}}
(\Sig,\betas^i,\betas^{i+1},\z;\spinc_{i-1};\Ring)$.\\

One should also fix the $\SpinC$ class of the triangles
contributing to the sum in equation~\ref{eq:Phi-f}.
Let us assume that the intersection points $\x_i\in\Ta\cap\Tbi$ for $i=0,1,2$ are
fixed so that $\relspinc^i(\x)\in\spinc\subset\Ss$. Furthermore, assume that,
after possible re-labeling of the curves in $\alphas$, we have
$$\x_i=\big\{x_i^1,...,x_i^\ell\big\},\ \ x_i^j\in
\begin{cases}\alpha_j\cap\beta^i_j\ \ &\text{if }1\leq j<\ell\\
\alpha_\ell\cap\mu_i\ \ &\text{if }j=\ell\end{cases}.$$
Also, for $j=1,...,\ell-1$, we will assume that $x_0^j,x_1^j$ and $x_2^j$ are
very close to each other, and correspond to one another by the Hamiltonian isotopies
considered above.
We may always change the $\alpha$ curves in the Heegaard diagram by
isotopy so that the above condition is satisfied. In order to specify the
class of triangles used in equation~\ref{eq:Phi-f}, we need to specify
triangle classes $\Delta_{i}\in\pi_2(\x_{i+1},\Theta_{i},\x_{i-1})$ for
any $i\in\frac{\Z}{3\Z}=\{0,1,2\}$.  The domain $\Dcal(\Delta_i)$ consists of a union of $\ell$
triangles. The first $\ell-1$ triangles are small triangles determined by the small Hamiltonian
isotopy changing the simple closed curves in $\betas^{i+1}-\{\mu_{i+1}\}$ to those in
$\betas^{i-1}-\{\mu_{i-1}\}$. Two of the vertices of the $j$-th triangle are the
intersection points $x_{i+1}^j$ and $x_{i-1}^j$, while the last vertex
belongs to the top generator $\Theta_i$. The $\ell$-th triangle connects
three intersection points between $\mu_{i},\mu_{i+1}$ and  $\alpha_\ell\in\alphas$.
With this notation fixed, let
$$\Dcal=\Dcal(\Delta_0)+\Dcal(\Delta_1)+\Dcal(\Delta_2).$$
 We assume that no $\alpha$ curve appears in $\partial \Dcal$.
Furthermore, we may assume that $n_p(\Dcal)=-1$ while $n_j(\Dcal)=0$
for $j=0,1,...,\el$.
$\Dcal$ is then the domain of a triangle class $\ov\Delta\in\pi_2(\Theta_0,\Theta_1,\Theta_2)$
with small area.
Note that achieving all these properties
may be done through a correct choice of the last triangle among the
$\ell$ triangles chosen above.

The choice of this last triangle class (with the above
properties) determines how
the map $f^\spinc_{i-1}$ changes the relative $\SpinC$ classes.
We will specify this last choice after the following lemma. \\

\begin{lem}\label{lem:f-respects-spinc}
There exists a cohomology class $h_i\in\Hbb$ for $i=0,1,2$ with the following property.
If for a generator $\x$ of $\E_{i}(\spinc)$ we have
$$\relspinc(\x)=\relspinc\in\spinc\subset\Ss,$$
and for the intersection point $\q\in\mathbb{T}_{\beta^i}\cap\mathbb{T}_{\beta^{i+1}}$
we have $\relspinc(\q)=\relspinc_{i-1}$, then
$$f_{i-1}^{\spinc}\big(\x\otimes \q\big)\in \E_{i+1}(\relspinc+h_{i-1}).$$
Furthermore, the cohomology classes satisfy
$$h_0+h_1+h_2=-\left(\chi_0+\chi_1+\chi_2\right).$$
\end{lem}
\begin{proof}
Once again, it suffices to prove the lemma for $i=0$. The cyclic symmetry of all definitions
then implies the lemma in general.
Let $\q\in\mathbb{T}_{\beta^0}\cap\mathbb{T}_{\beta^{1}}$
be an intersection point corresponding to the relative $\SpinC$ class $\relspinc_{2}$.
Suppose that $\y\in\Ta\cap\mathbb{T}_{\beta^1}$ is a generator such that
$\pi_2(\x,\q,\y)$ is non-empty, and that $\Delta$ is a
triangle class in this set. Then, using the fact that $n_2(\phi)=n_p(\phi)$ we will have
\begin{displaymath}
\begin{split}
\relspinc^0(\x)&=\big(\relspinc(\y)+h_{2}\big)
+\big(n_0(\phi)-n_p(\phi)\big)\chi_0+\big(n_1(\phi)-n_p(\phi)\big)\chi_1
+\sum_{j=3}^\el n_j(\phi)\chi_j\\
&=\big(\relspinc^1(\y)+h_{2}\big)-n_p(\phi)(\chi_0+\chi_1+\chi_2)+\sum_{j=0}^\el n_j(\phi)\chi_j\\
&=\relspinc^1(\la_\z(\phi).\y)+h_2.
\end{split}
\end{displaymath}
Here $h_2$ is a cohomology class in $\Hbb$ which depends on our identification of
spaces of relative $\SpinC$ classes associated with the diagram, and is chosen once
for all. Consider triangle classes
$$\Delta_{i-1}\in \pi_2(\x_i,\Theta_{i-1},\x_{i+1}),\ \ i\in\{0,1,3\}=\frac{\Z}{3\Z}$$
corresponding to the Heegaard diagrams
$(\Sig,\alphas,\betas^i,\betas^{i+1},\z)$ defined earlier.
We have assumed that the triangle classes are chosen so that
$$\Dcal=\Dcal(\Delta_0)+\Dcal(\Delta_1)+\Dcal(\Delta_2)=\Dcal(\ov\Delta)$$
is the  domain of the triangle class $\ov\Delta\in\pi_2(\Theta_0,\Theta_1,\Theta_2)$
so that $n_p(\Dcal)=-1$ and $n_j(\Dcal)=0$ for $j=0,...,\el$.
The above computation then implies that
\begin{displaymath}
\begin{split}
\relspinc^0(\x_0)&=\relspinc^1\left(\la_\z(\Delta_2).\x_1\right)+h_2\\
&=\relspinc^2\left(\la_\z(\Delta_0)\la_\z(\Delta_2).\x_2\right)+h_0+h_2\\
&=\relspinc^0\left(\la_\z(\Delta_0)\la_\z(\Delta_1)\la_\z(\Delta_2).\x_0\right)+h_0+h_1+h_2.\\
&=\relspinc^0\left(\la_\z(\ov\Delta).\x_0\right)+h_0+h_1+h_2.\\
\Rightarrow\ \ \ 0&=h_0+h_1+h_2-\chi\left(\la_p\right)=\sum_{i=0}^3 (h_i+\chi_i)
\end{split}
\end{displaymath}
This completes the proof of the lemma.
\end{proof}
In fact, we may choose the  identification of equation~\ref{eq:SpinC(Y)}
for $\SpinC$ classes   so that with the
notation of the above lemma we have
$$f_{i}^{\spinc}\big(\x\otimes \q\big)\in \E_{i-1}(\relspinc+(m_{i}-1)\chi_{i}),\ \
\forall\ i\in\frac{\Z}{3\Z},$$
or equivalently, $h_i=(1-m_i)\chi_i$. This last condition determines the triangle
classes in a unique way.
The closed top generator
$$\Theta_{i-1}\in\mathrm{CF}\left(L_{i-1},\nu_{i-1};\spinc_{i-1}\right)
\otimes\Ring$$
may then be used to define the map $\f^\spinc_{i-1}$ by
\begin{displaymath}
\begin{split}
\f^\spinc_{i-1}&:\E_{i}(\spinc)\lra \E_{i+1}(\spinc),\ \ \
\f^\spinc_{i-1}(\x):=f^\spinc_{i-1}(\x\otimes \Theta_{i-1}).
\end{split}
\end{displaymath}
For a relative $\SpinC$ class $\relspinc\in\spinc\subset \Ss$, the
restriction of $\f^\spinc_{i}$ to $\E_{i+1}(\relspinc)\subset \E_{i+1}(\spinc)$
will be denoted by $\f^\relspinc_{i}$. Lemma~\ref{lem:f-respects-spinc}
implies that the image of $\f^\relspinc_{i}$ is in $\E_{i+2}(\relspinc+(m_{i}-1)\chi_{i})$.\\

Straight forward arguments in Heegaard Floer homology (c.f. section 7 of \cite{OS-3m1})
may be used to show the following proposition, using the closed-ness of the
generators $\Theta_0,\Theta_1$ and $\Theta_2$:
\begin{prop}\label{prop:Phi-gh}
The maps $f^\relspinc_{i}$, for $\relspinc\in\spinc\subset \Ss$,
as defined above are all chain maps, which are induced by $\ta$ chain maps
$$\f^\spinc_i:\E_{i+1}(\spinc)\lra \E_{i+2}(\spinc),
\ \ i\in\frac{\Z}{3\Z}=\big\{0,1,2\big\}.$$
\end{prop}

\subsection{Compositions in the triangle are null-homotopic}
The maps defined in the previous subsection give a triangle of  $\ta$ chain maps
between filtered $\tab$ chain complexes:
\begin{equation}\label{eq:hol-triangle-2}
\begin{diagram}
\E_0(\spinc)=\bigoplus_{\relspinc\in\spinc}\E_0(\relspinc)&&\rTo{\f^{\spinc}_2}&&
\E_1(\spinc)=\bigoplus_{\relspinc\in\spinc}\E_1(\relspinc)\\
&\luTo{\f_1^{\spinc}}&&\ldTo{\f_0^{\spinc}}&\\
&&\E_{2}(\spinc)=\bigoplus_{\relspinc\in\spinc}\E_{2}(\relspinc)&&
\end{diagram}.
\end{equation}
The maps in this triangle change the associated relative $\SpinC$ class
in a controlled way, as described in lemma~\ref{lem:f-respects-spinc}.\\

Our first observation is the following theorem.
\begin{thm}\label{thm:exact-triangle}
With the notation of the previous sub-section,  the compositions
$\f^\spinc_{i+1}\circ \f^\spinc_i$, $i\in\frac{\Z}{3\Z}=\{0,1,2\}$
from the triangle in equation~\ref{eq:hol-triangle-2}
are  $\ta$ chain homotopic to zero for each $\SpinC$ class
$\spinc\in\SpinC(\ovl X)$.
More precisely, there are $\ta$ homotopy maps
\begin{displaymath}
\begin{split}
&H_i^\spinc:\E_{i-1}(\spinc)\ra \E_{i+1}(\spinc), \ \ \
i\in\frac{\Z}{3\Z}=\big\{0,1,2\big\},\ \ s.t.\\
&H_i^\spinc\circ \partial_{i-1}+\partial_{i+1}\circ H_i^\spinc
=\f^{\spinc}_{i-1}\circ\f^\spinc_{i+1} \ \ \forall\ i\in\frac{\Z}{3\Z}.
\end{split}
\end{displaymath}
\end{thm}
\begin{proof}
Throughout this proof, we will assume, for the sake of simplicity, that
the top generators $\Theta_i,\ i\in\frac{\Z}{3\Z}$
are represented by a single intersection point of the corresponding tori.
Define the homotopy map $H_i^\spinc$
from the Heegaard quadruple
$(\Sig,\alphas,\betas^{i-1},\betas^{i},\betas^{i+1},\z)$ by
\begin{equation}\label{eq:Phi-Homotopy}
\begin{split}
H_i^\spinc&:\E_{i-1}(\spinc)\lra\E_{i+1}(\spinc)\\
H_i^\spinc\big(\x\big)&:=
\sum_{\substack{\y\in\Ta\cap\mathbb{T}_{\beta^{i+1}}\\
\square\in\pi_2^{-1}(\x,\Theta_{i+1},\Theta_{i-1},\y)}}
\left(\m(\square)\la_\z(\square)\right).\y.
\end{split}
\end{equation}
Here $\pi_2^{j}(\x,\Theta_{i+1},\Theta_{i-1},\y)$ denotes the subset of $\pi_2(\x,\Theta_{i+1},\Theta_{i-1},\y)$
consisting of the squares $\square$ with $\mu(\square)=j$, and $\m(\square)$ denotes the number of
points in the moduli space $\Mod(\square)$, counted with sign. Furthermore $\la_\z(\square)$
is defined by
$$\la_\z(\square)=\la_p^{n_p(\square)}\prod_{j=0}^\el \la_j^{n_j(\square)}\in\Ring.$$
In equation~\ref{eq:Phi-Homotopy}, we only count square classes which may be represented
as the juxtaposition of the small triangle class $\ov\Delta$ in $\pi_2(\Theta_{i+1},\Theta_{i-1},\Theta_i)$ with
the triangle class $\Delta_i$ in $\pi_2(\x,\Theta_{i},\y)$.
We will drop this condition from the notation for
the sake of simplicity.
\begin{lem}\label{lem:H-respects-spinc}
With the above notation fixed, for any relative $\SpinC$ class
$\relspinc\in\spinc\subset \Ss$
the image of
$$H_i^\relspinc=H_i^\spinc|_{\E_{i-1}(\relspinc)}:\E_{i-1}(\relspinc)\lra \E_{i+1}(\spinc)$$
is in the sub-complex $$\E_{i+1}\left(\relspinc-(m_i\chi_i+\chi_{i-1}+\chi_{i+1})\right)\subset \E_{i+1}(\spinc).$$
\end{lem}
\begin{proof}
Without loosing on generality, we may assume that $i=0$.
Let $\square\in\pi_2^{-1}(\x,\Theta_{1},\Theta_{2},\y)$ be a square connecting $\x$  to $\y$.
We can thus find an element $h\in\Hbb$ such that for all such generators and
square classes we have
\begin{displaymath}
\begin{split}
\relspinc^2(\x)&=\relspinc^1(\y)+h+\sum_{i=0}^2(n_i(\square)-n_p(\square))\chi_i
+\sum_{j=3}^\el n_j(\square)\chi_j\\
&=\relspinc^1(\y)+h+\chi\left(\la_p^{n_p(\square)}\prod_{j=0}^\el \la_j^{n_j(\square)} \right)
=\relspinc^1\left(\la_\z(\square).\y\right)+h.
\end{split}
\end{displaymath}
Considering the square classes which are obtained as the juxtaposition of triangles
corresponding to $f^\spinc_{i-1}$ and $f^\spinc_{i+1}$,
and using the coherence of the system of $\SpinC$ classes, we may compute $h$:
\begin{displaymath}
\begin{split}
h=h_2+h_1
&=(m_2-1)\chi_2+(m_1-1)\chi_1\\
&=-(m_0\chi_0+\chi_1+\chi_2).
\end{split}
\end{displaymath}
This completes the proof of the lemma.
\end{proof}

If $\y$ is an intersection point in $\Ta\cap\mathbb{T}_{\beta^{i+1}}$  and
if $\square\in \pi_2(\x,\Theta_{i+1},\Theta_{i-1},\y)$ is a square with $\mu(\square)=0$,
we may consider the
moduli space $\Mod(\square)$, which is a smooth, oriented  $1$-dimensional manifold with boundary.
The boundary points of this moduli space correspond to different types of degenerations of $\square$.
Four types of these degenerations, are degenerations of $\square$ to a bi-gon and a square.
Since $\Theta_{i-1}$ and
$\Theta_{i+1}$ are closed elements in their corresponding chain complexes,
counting such degenerations
contribute to the coefficient of $\la_\z(\square).\y$
in the expression
$$\left( H_i^\spinc\circ \partial_{i-1}+\partial_{i+1}\circ H_i^\spinc\right)(\x).$$
Then we have the possibility of a degeneration of $\square$ as $\Delta\star \Delta'$
with $\Delta\in\pi_2(\x,\q,\y)$ and
$\Delta'\in \pi_2(\Theta_{i+1},\Theta_{i-1},\q)$ for some
$\q\in \mathbb{T}_{\beta^{i-1}}\cap\mathbb{T}_{\beta^{i+1}}$ satisfying
$\mu(\Delta)=\mu(\Delta')=0$.
Such degenerations correspond to the appearance of $\y$ in the expression
$$\Psi_i\big(\x\otimes \Phi_i(\Theta_{i+1}\otimes\Theta_{i-1} )\big),$$
where the holomorphic triangle maps $\Psi_{i}$ and $\Phi_i$ are defined by
\begin{displaymath}
\begin{split}
&\Psi_{i}:\E_{i-1}(\spinc)\otimes{\mathrm{CF}}
(\Sig,\betas^{i-1},\betas^{i+1},\z;\spinc_i;\Ring)\lra \E_{i+1}(\spinc)\\
&\Psi_i\big(\x\otimes \p\big)
:=\sum_{\substack{\w\in\Ta\cap\mathbb{T}_{\beta^{i+1}}
\\ \Delta\in\pi_2^0(\x,\p,\w)}}
\left(\m(\Delta)\la_\z(\Delta)\right)\w\\
&\Phi_i(\Theta_{i+1}\otimes\Theta_{i-1} ):=
\sum_{\substack{\p\in\mathbb{T}_{\beta^{i-1}}\cap\mathbb{T}_{\beta^{i+1}}
\\ \Delta\in\pi_2^0(\Theta_{i+1},\Theta_{i-1},\p)}}
\left(\m(\Delta)\la_\z(\Delta)\right)\p.
\end{split}
\end{displaymath}
Since the Heegaard diagram $(\Sig,\betas^{i-1},\betas^i,\betas^{i+1},\z)$ is a
 standard diagram, one may easily observe that $\Phi_i(\Theta_{i+1}\otimes\Theta_{i-1})=0$.
 The reason for this vanishing is that
holomorphic triangles which contribute to the above sum come in pairs.
This is in fact the same phenomena as what happens in the surgery exact sequence of
Ozsv\'ath and Szab\'o \cite{OS-Zsurgery}. The relation
$$\la_p=\la_0^{m_0-1}\la_1^{m_1-1}\la_2^{m_2-1}$$
then guarantees that the element of $\Ring$ associated with both triangles
in any pair is the same (note that the corresponding signs associated by
the orientation convention are different).
Thus, the triangles in each pair will cancel each other
to give
$$\Phi_i(\Theta_{i+1}\otimes\Theta_{i-1})=0\ \ \Rightarrow\ \
\Psi_i\big(\x\otimes \Phi_i(\Theta_{i+1}\otimes\Theta_{i-1} )\big)=0.$$

Finally, the last type of degeneration for the domain $\square$ is a degeneration of
$\square$ as $\square=\Delta\star\Delta'$, where $\Delta\in\pi_2^0(\x,\Theta_{i+1},\w)$ and
$\Delta'\in\pi_2^0(\w,\Theta_{i-1},\y)$ for some $\w\in \Ta\cap\mathbb{T}_{\beta^i}$.
Counting the end points of $\Mod(\square)$ corresponding to such
degenerations gives the coefficient of $\la_\z(\square).\y$
in $\big(\f^\spinc_{i-1}\circ\f^\spinc_{i+1}\big)(\x)$.\\

Gathering all this data we conclude that the following relation is satisfied.
\begin{displaymath}
\begin{split}
H_i^\spinc\circ \partial_{i+1}
+\partial_{i-1}\circ H_i^\spinc
=\f^\spinc_{i-1}\circ \f^\spinc_{i+1},\ \ \forall \ i\in\frac{\Z}{3\Z}=\big\{0,1,2\big\}.
\end{split}
\end{displaymath}
implying that $\f^\spinc_{i-1}\circ \f^\spinc_{i+1}$ is $\ta$-chain homotopic
to zero, and that the decomposition into relative $\SpinC$ classes in $\Ss$
is almost respected by the maps in the sense described in lemmas~\ref{lem:H-respects-spinc}
and~\ref{lem:f-respects-spinc}.\\
\end{proof}

\subsection{Exactness and computation of chain homotopy type}
We would like to apply lemma~\ref{lem:quasi-isomorphism} to the triangle of equation~\ref{eq:hol-triangle-2}.
For this purpose, we have to refine the coefficient ring as follows. Let
$$\Ringg:=\frac{\Ring\big[\xi_p\big]}{\big\langle 1-\la_p\xi_p \big\rangle}$$
be the algebra constructed from $\Ring$ by adding an inverse for $\la_p$.
There is a natural homomorphism
$$\rho:\Ring\lra \Ringg$$
which may be used to take the tensor product of any $\Ring$ module with $\Ringg$ and construct
a $\Ringg$ module from it.
We will sometimes denote $\xi_p$ by $\la_p^{-1}$. In the algebra $\Ringg$, the element
$\la_i^{m_i-1}$ is invertible, and it thus makes sense to talk about
$\la_i^{1-m_i}$. In fact, if $m_i>1$  the element $\la_i$ itself will be invertible.
The most interesting case is, however, when some or all of $m_i$ are equal to $1$.
In particular, if $m_0=m_1=m_2=1$, $\la_p=1\in \Ring$ and consequently we will have
$\Ringg=\Ring$.\\

We may define, for $i\in\frac{\Z}{3\Z}=\{0,1,2\}$,
\begin{displaymath}
\begin{split}
&\g^\spinc_i:\D_{i+1}(\spinc):=\E_{i+1}(\spinc)\otimes_\Ring\Ringg
\lra \D_{i-1}(\spinc):=\E_{i-1}(\spinc)\otimes_\Ring\Ringg\\
&\g^\spinc_i(a)=\la_i^{1-m_i}.\f^\spinc_i(a),\ \ \ \forall\ a\in \D_{i+1}(\spinc).
\end{split}
\end{displaymath}
From lemma~\ref{lem:f-respects-spinc} we know that $\g^\spinc_i$ is a filtered
$\tabb$ map between filtered $\tabb$ chain complexes $\D_{i+1}(\spinc)$
and $\D_{i-1}(\spinc)$ which decomposes as a sum of maps
\begin{displaymath}
\begin{split}
&\g^\relspinc_i:\D_{i+1}(\relspinc)
\lra \D_{i-1}(\relspinc),\ \ \ \forall\ \relspinc\in\spinc\subset \Ss.
\end{split}
\end{displaymath}
We may also modify the maps $H^\spinc_i$ so that they respect the relative
$\SpinC$ decompositions. According to lemma~\ref{lem:H-respects-spinc} the
following definition gives the appropriate chain homotopy maps
\begin{displaymath}
\begin{split}
&G^\spinc_i=\bigoplus_{\relspinc\in\spinc}G^\relspinc_i:\D_{i-1}(\spinc)=\bigoplus_{\relspinc\in\spinc}
\D_{i-1}(\relspinc)
\lra \D_{i+1}(\spinc)=\bigoplus_{\relspinc\in\spinc}
\D_{i+1}(\relspinc)\\
&G^\spinc_i(a)=\big(\xi_p\la_i^{m_i-1}\big).H^\spinc_i(a),\ \ \ \forall\ a\in \D_{i-1}(\spinc).
\end{split}
\end{displaymath}

\begin{thm}\label{thm:quasi-isomorphism}
With our previous notation fixed and for any $i\in\frac{\Z}{3\Z}=\{0,1,2\}$,
the map from $\D_{i}(\spinc)$ to the mapping cone
of $\g^\spinc_i$ defined by
\begin{displaymath}
\begin{split}
&\Inv^\spinc_i:\D_i(\spinc)
\lra \D_{i+1}(\spinc)\oplus\D_{i-1}(\spinc)\\
&\Inv^\spinc_i(\z):=\big(\g^\spinc_{i-1}(\z),G^\spinc_{i+1}(\z)\big)
\end{split}
\end{displaymath}
is a filtered chain homotopy equivalence of filtered $\tabb$ chain complexes.
In particular, the decompositions of the two sides into relative $\SpinC$
classes $\relspinc\in\spinc\subset\Ss$ is respected
by this chain homotopy equivalence.
\end{thm}
\begin{proof}
For any integer $j\in\Z$ let us define
\begin{displaymath}
A_j:=\begin{cases}
\D_{0}(\spinc)\ \ \ \ &\text{if }j=0\ (\mathrm{mod}\ 3)\\
\D_{1}(\spinc)\ \ \ \ &\text{if }j=1\ (\mathrm{mod}\ 3)\\
\D_{2}(\spinc)\ \ \ \ &\text{if }j=2\ (\mathrm{mod}\ 3)\\
\end{cases}
\end{displaymath}
Denote the differential of $A_j$ by $d_j$.
Furthermore, define $f_j:A_j\ra A_{j+1}$ to be $\g^\spinc_2$,
$\g^\spinc_0$ or $\g^\spinc_1$ for $j=0,1$ or $2$ modulo $3$, respectively.
Let $H_j:A_j\ra A_{j+2}$, depending on whether $j=0,1$ or $2$ modulo $3$
be the maps $G^\spinc_1, G^\spinc_2$ and $G^\spinc_0$, respectively.
By lemma~\ref{lem:quasi-isomorphism},
 in order to show that the map $\Inv^\spinc_i$
is a chain homotopy equivalence of filtered $\tabb$ chain complexes we have to show that
the differences $\phi_i=f_{i+2}\circ H_i-H_{i+1}\circ f_i:A_i\ra A_{i+3}$ are chain homotopy equivalences.
Checking that all the constructions respect the decomposition into relative $\SpinC$ classes
in $\Ss$ is straight-forward from the lemmas~\ref{lem:f-respects-spinc} and~\ref{lem:H-respects-spinc}.
\\

As in \cite{OS-branched-double-cover} and \cite{OS-Zsurgery}, checking the above claim is
done by considering holomorphic pentagons associated with Heegaard
diagrams of the form
$$(\Sig,\alphas,\betas^j,\betas^{j+1},\betas^{j+2},\betas^{j+3},\z),$$
where $\betas^j$ denotes a set of $\ell$ simple closed curves which are
Hamiltonian isotopes of the curves in $\betas^i$
where  $i\in\frac{\Z}{3\Z}$ is equal to $0,1$ or $2$ and $j$ is congruent to
$i$ modulo $3$.
Let us denote the top generator of the Heegaard Floer homology group associated with
$(\Sig,\betas^j,\betas^{j+1},\z)$ by $\Theta_j$, by little abuse of notation.
More generally, the top generator associated with
$(\Sig,\betas^i,\betas^j,\z)$ will be denoted
by $\Theta_{ij}$.
For any three indices $i<j<k$, there is a triangle,
with small area (assuming that the Hamiltonian isotopies
changing the curve collection to each other are small)
which connects $\Theta_{ij},\Theta_{jk}$ and $\Theta_{ik}$.
Denote this triangle class by $\Delta_{ijk}$.
\\

Without loosing on generality, we may assume that $j=0$ modulo $3$.
Choose a generator $\x\in\Ta\cap\Tbjj$ so that
$\relspinc(\x)=\relspinc\in\spinc$. The curves
in $\betas^{j+3}$ are Hamiltonian isotopes of those in
$\betas^j$. Thus there is a natural {\emph{closest point}}
map $$I:\Ta\cap\Tbjj\ra \Ta\cap\mathbb{T}_{\beta^{j+3}}.$$
There is a natural triangle class connecting
$\Theta_{j,j+3},\x$ and $I(\x)$ which will be denoted
by $\Delta_\x$.\\

Let us denote the complex associated with $(\Sig,\betas_j,\betas_{j+1},\z)$
and the coefficient ring $\Ringg$ with $B_j$, and the
complex associated with $(\Sig,\betas_j,\betas_{j+2},\z)$ (again with coefficient ring $\Ringg$)
by $C_j$, and finally
the complex associated with $(\Sig,\betas_j,\betas_{j+3},\z)$ by $D_j$.
We omit the straight forward details of the definitions.
\\

Define a map $\Pcal_j:A_j\ra A_{j+3}\cong A_j$ by
\begin{displaymath}
\begin{split}
\Pcal_j(\x)=\sum_{\substack{\y\in\Ta\cap \mathbb{T}_{\beta_{j+3}}\\
\pentagon\in\pi_2^{-2}(\x,\Theta_{j+1},\Theta_{j+2},\Theta_{j+3},\y)
}}\left(\m(\pentagon)\xi_p\la_\z(\pentagon)\right)\y.
\end{split}
\end{displaymath}
The class of the pentagons counted in the above sum is determined by juxtaposing a
triangle class $\Delta_\x\in\pi_2(\x,\Theta_{j,j+3},I(\x))$ with an standard  square class
$\ov\square\in\pi_2(\Theta_j,\Theta_{j+1},\Theta_{j+2},\Theta_{j,j+3})$ with small
area. As usual, we will drop this class from the notation.\\

Let us assume that $\pentagon\in\pi_2^{-1}(\x,\Theta_{j+1},\Theta_{j+2},\Theta_{j+3},\y)$
is a pentagon class 
which has Maslov index $-1$. Consider the ends of the
smooth orientable one dimensional moduli space $\Mod(\pentagon)$, which correspond to
the degenerations discussed in theorem~\ref{thm:general-associativity}.\\

Considering the possible degenerations
at the boundary of $\Mod(\pentagon)$, theorem~\ref{thm:general-associativity} implies
\begin{equation}\label{eq:quasi-iso}
\begin{split}
\phi_j(a_j)=&(\Pcal_j\circ d_j-d_j\circ \Pcal_j)(a_j)\\
&\ +\xi_p.I_j\Big(a_j\otimes K_j(\Theta_j\otimes\Theta_{j+1}\otimes
\Theta_{j+2})\Big),\ \ \forall\ a_j\in A_j,
\end{split}
\end{equation}
where the maps $I_j:A_j\otimes D_j\ra A_{j+3}$ and $K_j(\Theta_j\otimes\Theta_{j+1}\otimes
\Theta_{j+2})$
are defined as follows.
\begin{equation}\label{eq:H-square}
\begin{split}
&I_j\big(\x\otimes\q\big):=\sum_{\y\in\Ta\cap\mathbb{T}_{\beta^{j+3}}}\sum_{\substack{\Delta\in \pi_2^0(\x,\q,\y)\\
}}\left(\m(\Delta)\la_\w(\Delta)\right)\y\\
&K_j(\Theta_j\otimes\Theta_{j+1}\otimes\Theta_{j+2}):=
\sum_{}
\sum_{\substack{\q\in\Tbjj\cap\mathbb{T}_{\beta^{j+3}}\\
\square\in \pi_2^{-1}(\Theta_j,\Theta_{j+1},\Theta_{j+2},\q)\\
}}\left(\m(\square)\la_\z(\square)\right)\q.
\end{split}
\end{equation}
Two of the terms appearing in theorem~\ref{thm:general-associativity} vanish and are not present in the
equation~\ref{eq:quasi-iso}. These are the terms that correspond to degenerations containing
a triangle in $\pi_2(\Theta_j,\Theta_{j+1},\q)$ for some $\q\in\Tbjj\cap\mathbb{T}_{\beta^{j+2}}$,
or a triangle in $\pi_2(\Theta_{j+1},\Theta_{j+2},\q)$ for some $\q\in\mathbb{T}_{\beta^{j+1}}\cap\mathbb{T}_{\beta^{j+3}}$.
The total contribution of such triangles vanishes, since they come in canceling pairs. Thus the terms containing such
degenerations would vanish as well.\\

Note that the map
$\x\mapsto I_j(\x\otimes \Theta_{j,j+3})$ is a perturbation of the isomorphism
$I$ with a map $\epsilon:A_j\ra A_{j+3}$ which
takes a generator $\x$ to generators with smaller energy than $I(\x)$, when we equip
$A_{j+3}$ with an appropriate energy filtration.
This follows since the contributions from triangle classes other than
$\Delta_\x$ will contribute more than the small energy associated with
$\Delta_\x$.
Standard arguments in Heegaard Floer
theory (c.f. Ozsv\'ath and Szab\'o's original paper \cite{OS-3m1})
may then be applied to construct an explicit
inverse for this map up to filtered $\tabb$ chain homotopy.
In order to complete the proof of the theorem, it is thus enough to show that
$$K_j(\Theta_j\otimes\Theta_{j+1}\otimes\Theta_{j+2})=\la_p\Theta_{j,j+3}.$$
This can be proved directly, since the Heegaard quadruple
$$(\Sig,\betas^j,\betas^{j+1},\betas^{j+2},\betas^{j+3},\z)$$
is a special Heegaard diagram, which may be analyzed without too much difficulty.
The only difference with earlier considerations of Ozsv\'ath and Szab\'o
(e.g. in \cite{OS-branched-double-cover}, subsection 4.2)
is the following. There is a preferred square class which contributes to the
second sum of equation~\ref{eq:H-square}. This square class has small total area, and multiplicity $1$
at $p$. The contribution of this class would give $\la_p\Theta_{j,j+3}$.
The rest of contributing square classes come in pairs and the elements of $\Ringg$ associated
with both elements in each pair are the same (with opposite sign), since the relation
$$\la_p=\la_0^{m_0-1}\la_1^{m_1-1}\la_2^{m_2-1}$$ is satisfied in $\Ring$ and hence in $\Ringg$.
Thus the two square classes in each pair cancel each other.
\end{proof}

\subsection{Special cases}
Let us now consider a few special cases, which correspond to the existing exact sequences
in Heegaard Floer homology. We use the observation of this subsection as an indication of how
the new language developed in this paper for understanding Heegaard Floer theory of Ozsv\'ath
and Szab\'o may be used to understand the existing objects in a uniform way, and extend
them to the more general setup.
\\

The first case we would like to consider, is the case where
$m_0=m_1=m_2=1$. This would be the case if $\mu_0$ and $\mu_1$ cut each other in a single
transverse point (so that $\#(\mu_0.\mu_1)=1$) and $\mu_2=-(\mu_0+\mu_1)$.
In this case, we will have
\begin{displaymath}
\begin{split}
\Ring&=\frac{\ov\Ring}{\big\langle \la_p=\la_0^{m_0-1}\la_1^{m_1-1}\la_2^{m_2-1}\big\rangle}
=\frac{\ov\Ring}{\big\langle\la_p=1\big\rangle},\ \ \&\\
\Ringg&=\frac{\Ring[\xi_p]}{\big\langle \xi_p\la_p=1\big\rangle}=\Ring
=\frac{\ov\Ring}{\big\langle\la_p=1\big\rangle}.
\end{split}
\end{displaymath}
Thus the algebra $\Ringg$ is the algebra associated with any of the diagrams
$$(\Sig,\alphas,\betas^i,\z-\{p\}),\ \ \ i=0,1,2.$$
In particular, if the simple closed curve $\la$  determines the surgery, and
$\mu_0$, $\mu_1$ and $\mu_2$ correspond to $\infty$, $0$ and $1$ surgeries
respectively, the above conditions are satisfied.
For an arbitrary ring $R$ such that there is a ring homomorphism
$\rho_R:\Ringg\ra R$, one may define the $R$ chain complexes
$$\D_i(\spinc; R)=\D_i(\spinc)\otimes_\Ringg R,\ \ i=0,1,2.$$
The triangle of theorem~\ref{thm:quasi-isomorphism} gives a triangle of
$R$ chain maps between $\D_i(\spinc;R)$, $i=0,1,2$, and the conclusion of
the theorem remains true (however, we may need to drop the filtration
from the conclusions if $\rho_R$ does not respect the
filtration). In particular, the homology groups
$$\Hbb_i(\spinc;R):=H_*(\D_i(\spinc;R),\partial_{\D_i}),\ \ i=0,1,2$$
fit into an exact triangle. The exact triangle of
\cite{Ef-splicing} and \cite{Ef-c-splicing}
is a special case of such exact triangles.
\\

Consider the algebra $\ovl \Ring$ and the quotient map $r:\Ring=\Ringg \ra \ovl \Ring$ where
$$\ovl\Ring=\frac{\Ring}{\big\langle\la_0=\la_1=1\big\rangle}
=\frac{\Ringg}{\big\langle\la_0=\la_1=1\big\rangle}.$$
Correspondingly, let $\ovl \Hbb$ be the quotient of $\Hbb$ by the action of
$\chi_0,\chi_1$ and $\chi_2$. From the filtered $\tabb$ chain complexes
$\D_i(\spinc)$ we may construct the filtered $\tabl$ chain complexes
$$\C_i(\spinc):=\D_i(\spinc)\otimes_\Ringg\ovl\Ring,\ \ \ i\in\frac{\Z}{3\Z}=\big\{0,1,2\big\}.$$
The complex $\C_i(\spinc)$ may be identified as $\CFT(X_{\tau^i}(1),\tau^i(1),\spinc)$ when $i=0,1$ and
as $\CFT(X_{\tau^2}(2),\tau^2(2),\spinc)$ when $i=2$. Here $X_{\tau^i}(j)$ denotes the
three-manifold obtained from $X$ by filling out the $j$-th suture in $\tau^i$ and $\tau^i(j)$ denotes
the induced set of sutures on the boundary of $X_{\tau^i}(j)$.
The triangle of theorem~\ref{thm:quasi-isomorphism} thus generalizes the exact sequence in homology,
which appears as theorem 1.7 in \cite{OS-3m2}. Also, theorem 4.7 from \cite{OS-branched-double-cover}
is a special case of theorem~\ref{thm:quasi-isomorphism}.
Furthermore, theorem 8.2 from \cite{OS-knot}
is also a corollary in this situation.\\

Let us now assume that $m_1=m_2=1$, while $m_0=m$ is an arbitrary integer. In particular,
if $\mu_0$, $\mu_1$ and $\mu_2$ correspond to the surgery coefficients $\infty, n$ and $n+m$,
for some integer $n\in\Z$ this would be the case. We will thus have
\begin{displaymath}
\begin{split}
\Ring=\frac{\ov\Ring}{\big\langle \la_p=\la_0^{m-1}\big\rangle}.
\end{split}
\end{displaymath}
We may thus define the following quotient ring of $\Ringg$:
\begin{displaymath}
\begin{split}
\Ring_m&=\frac{\ov\Ring}{\big\langle \la_p-\la_0^{m-1}, \la_0^m-1\big\rangle}
=\frac{\Ringg}{\big\langle\xi_p=\la_0\big\rangle}.
\end{split}
\end{displaymath}
The filtration module $\Hbb=\Ht^2(X,\partial X;\Z)$ and the space $\Ss=\SpinC(X,\tau)$ corresponding
to the coefficient ring $\Ringg$ should be changed to
\begin{displaymath}
\begin{split}
\Hbb_m&:=\frac{\Hbb}{\big\langle m\chi_0\big\rangle_\Z}
=\frac{\Ht^2(X,\partial X;\Z)}{\big\langle m\chi_0\big\rangle_\Z},\ \ \&\\
\Ss_m&:=\frac{\Ss}{\big\langle m\chi_0\big\rangle_\Z}
=\frac{\SpinC(X,\tau)}{\big\langle m\chi_0\big\rangle_\Z}.
\end{split}
\end{displaymath}
Theorem~\ref{thm:quasi-isomorphism} then gives the main result of \cite{Ef-surgery}
as a special case. The surgery exact sequence of theorem 3.1 in \cite{OS-Zsurgery}
is in turn a consequence of this last result. In a similar way, theorem 6.2 from
\cite{OS-Qsurgery} follows from this last consideration.
\\
\newpage

\end{document}